\documentclass[a4,12pt,epsfig]{article}
\usepackage{graphicx}
\usepackage{subfig}
\usepackage{epsfig}
\usepackage{amsmath}
\usepackage{amssymb}
\usepackage{amsfonts}
\usepackage{euscript}
\newcommand{\bfm}[1]{{\mbox{\boldmath{$#1$}}}}
\begin{document}
\bibliographystyle{plain}
\newcommand{\bea}{\begin{eqnarray}}
\newcommand{\eea}{\end{eqnarray}}
\newcommand{\bfmN}{{\mbox{\boldmath{$N$}}}}
\newcommand{\bfmx}{{\mbox{\boldmath{$x$}}}}
\newcommand{\bfmv}{{\mbox{\boldmath{$v$}}}}
\newcommand{\se}{\setcounter{equation}{0}}
\newtheorem{corollary}{Corollary}[section]
\newtheorem{example}{Example}[section]
\newtheorem{definition}{Definition}[section]
\newtheorem{theorem}{Theorem}[section]
\newtheorem{proposition}{Proposition}[section]
\newtheorem{lemma}{Lemma}[section]
\newtheorem{remark}{Remark}[section]
\newtheorem{result}{Result}[section]
\newcommand{\vtwo}{\vskip 4ex}
\newcommand{\vthree}{\vskip 6ex}
\newcommand{\vfour}{\vspace*{8ex}}
\newcommand{\hone}{\mbox{\hspace{1em}}}
\newcommand{\hon}{\mbox{\hspace{1em}}}
\newcommand{\htwo}{\mbox{\hspace{2em}}}
\newcommand{\hthree}{\mbox{\hspace{3em}}}
\newcommand{\hfour}{\mbox{\hspace{4em}}}
\newcommand{\von}{\vskip 1ex}
\newcommand{\vone}{\vskip 2ex}
\newcommand{\n}{\mathfrak{n} }
\newcommand{\m}{\mathfrak{m} }
\newcommand{\q}{\mathfrak{q} }
\newcommand{\aF}{\mathfrak{a} }

\newcommand{\kl}{\mathcal{K}}
\newcommand{\p}{\mathcal{P}}
\newcommand{\Lt}{\mathcal{L}}
\newcommand{\bv}{{\mbox{\boldmath{$v$}}}}
\newcommand{\bc}{{\mbox{\boldmath{$c$}}}}
\newcommand{\bx}{{\mbox{\boldmath{$x$}}}}
\newcommand{\br}{{\mbox{\boldmath{$r$}}}}
\newcommand{\bs}{{\mbox{\boldmath{$s$}}}}
\newcommand{\bb}{{\mbox{\boldmath{$b$}}}}
\newcommand{\ba}{{\mbox{\boldmath{$a$}}}}
\newcommand{\bn}{{\mbox{\boldmath{$n$}}}}
\newcommand{\bp}{{\mbox{\boldmath{$p$}}}}
\newcommand{\by}{{\mbox{\boldmath{$y$}}}}
\newcommand{\bz}{{\mbox{\boldmath{$z$}}}}
\newcommand{\be}{{\mbox{\boldmath{$e$}}}}
\newcommand{\proof}{\noindent {\sc Proof :} \par }
\newcommand{\bP}{{\mbox{\boldmath{$P$}}}}

\newcommand{\M}{\mathcal{M}}
\newcommand{\R}{\mathbb{R}}
\newcommand{\Q}{\mathbb{Q}}
\newcommand{\Z}{\mathbb{Z}}
\newcommand{\N}{\mathbb{N}}
\newcommand{\C}{\mathbb{C}}
\newcommand{\xar}{\longrightarrow}
\newcommand{\ov}{\overline}
 \newcommand{\rt}{\rightarrow}
 \newcommand{\om}{\omega}
 \newcommand{\wh}{\widehat }
 \newcommand{\wt}{\widetilde }
 \newcommand{\g}{\Gamma}
 \newcommand{\lm}{\lambda}

\newcommand{\eN}{\EuScript{N}}
\newcommand{\ncom}{\newcommand}
\newcommand{\norm}{\|\;\;\|}
\newcommand{\inp}[2]{\langle{#1},\,{#2} \rangle}
\newcommand{\nrm}[1]{\parallel {#1} \parallel}
\newcommand{\nrms}[1]{\parallel {#1} \parallel^2}
\title{Quasilinear viscous approximations to scalar conservation laws}
\author{ Ramesh Mondal\footnote{ramesh@math.iitb.ac.in}, S. Sivaji Ganesh\footnote{siva@math.iitb.ac.in} and S. Baskar\footnote{baskar@math.iitb.ac.in}}
\maketitle{}
\begin{abstract}
We prove the convergence of quasilinear parabolic viscous approximations to the entropy solution (in the sense of Bardos-Leroux-Nedelec) of a scalar conservation law, considered on a  bounded domain in $\R^d$. 
\end{abstract}

\section{Introduction}
Let $\Omega$ be a bounded domain in $\mathbb{R}^{d}$ with smooth boundary $\partial \Omega$. For $T >0$, denote $\Omega_{T}:= \Omega\times(0,T)$. The aim of this article
is to prove that the sequence of solutions (called quasilinear viscous approximations) of the generalized viscosity problems  
\begin{subequations}\label{ibvp.parab}
\begin{eqnarray}
 u_{t} + \nabla \cdot f(u) = \varepsilon\,\nabla\cdot\left(B(u)\,\nabla u\right)&\mbox{in }\Omega_{T},\label{ibvp.parab.a} \\
    u(x,t)= 0&\,\,\,\,\mbox{on}\,\, \partial \Omega\times(0,T),\label{ibvp.parab.b}\\
u(x,0) = u_0(x)& x\in \Omega,\label{ibvp.parab.c}
\end{eqnarray}
\end{subequations}
indexed by $\varepsilon>0$, converges to the entropy solution (in the sense of \cite{MR542510}) of the  initial-boundary value problem (IBVP) for the scalar conservation laws given by

\begin{subequations}\label{ivp.cl}
\begin{eqnarray}
  u_t + \nabla\cdot f(u) =0& \mbox{in }\Omega_T,\label{ivp.cl.a}\\
u(x,t)= 0&\mbox{on}\,\,\partial \Omega\times(0,T),\label{ivp.cl.b}\\
  u(x,0) = u_0(x)& x\in \Omega.\label{ivp.cl.c}
  \end{eqnarray}
\end{subequations}
as $\varepsilon\rightarrow0$.  In this context, we have three results (Theorem~\ref{theorem1}, Theorem~\ref{theorem2} and 
Theorem~\ref{theorem3}) depending on the regularity of the data. First we establish the convergence under more regular 
hypothesis on the data (see Hypothesis A), then use it to prove the result for data with lesser regularity (see Hypothesis B) 
and finally, we prove the result for data with lesser regularity and initial data without having compact essential support in $\Omega$ (see Hypotheses C).
We now state the hypotheses, and then the main results of this article.
 
{\bf Hypothesis A}
\begin{enumerate}
 \item Let $f\in \left(C^4(\R)\right)^d$, $f^\prime\in \left(L^\infty(\R)\right)^d$, and denote 
 $$\|f^\prime\|_{\left(L^\infty(\R)\right)^d}:=\max_{1\leq j\leq d}\,\sup_{y\in\R}|f^\prime_j(y)|.$$
 \item Let $B\in C^3(\R)\cap L^\infty(\R)$, and there exists an $r>0$ such that $B\geq r$.
 \item  Let $0<\beta<1$, and $u_0\in C^{4+\beta}_{c}({\Omega})$, where 
 $$C^{4+\beta}_{c}({\Omega}):=\left\{u_{0}\in C^{4+\beta}({\Omega})\,;\,\,\mbox{supp}(u_{0})\,\,\mbox{is compact in }\,\,\Omega\right\},$$
 and we denote $I=[-\|u_0\|_{\infty},\|u_0\|_{\infty}]$.
 \end{enumerate}
For definitions of H\"{o}lder spaces used in this work, we refer the reader to \cite{lad-etal_68a}.

\noindent{\bf Hypothesis B}
\begin{enumerate}
 \item Let $f\in \left(C^1(\R)\right)^d$, $f^\prime\in \left(L^\infty(\R)\right)^d$, and denote 
 $$\|f^\prime\|_{\left(L^\infty(\R)\right)^d}:=\max_{1\leq j\leq d}\,\sup_{y\in\R}|f^\prime_j(y)|.$$
 \item Let $B\in C^3(\R)\cap L^\infty(\R)$, and there exists an $r>0$ such that $B\geq r$.
 \item  Let $u_0$ belong to the space $W^{1,\infty}_c({\Omega})$ consisting of those elements of $W^{1,\infty}({\Omega})$ whose essential support is a compact subset of $\Omega$, and we denote $I=[-\|u_0\|_{\infty},\|u_0\|_{\infty}]$.
\end{enumerate}
\noindent{\bf Hypothesis C}
\begin{enumerate}
 \item Let $f\in \left(C^1(\R)\right)^d$, $f^\prime\in \left(L^\infty(\R)\right)^d$, and denote 
 $$\|f^\prime\|_{\left(L^\infty(\R)\right)^d}:=\max_{1\leq j\leq d}\,\sup_{y\in\R}|f^\prime_j(y)|.$$
 \item Let $B\in C^3(\R)\cap L^\infty(\R)$, and there exists an $r>0$ such that $B\geq r$.
 \item Let $u_{0}\in H^{1}_{0}(\Omega)\cap C(\overline{\Omega})$. Let $u_{0\varepsilon}$ be in $\mathcal{D}(\Omega)$ such that 
 for all $\varepsilon>0$, there exist constants $A, C>0$ such that $\|u_{0\varepsilon}\|_{L^{\infty}(\Omega)}\leq A$,
 $\|\Delta u_{0\varepsilon}\|_{L^{1}(\Omega)}\leq C$ and $u_{0\varepsilon}\to u_0$  in $H^{1}_{0}(\Omega)$ as $\varepsilon\to 0$.
 Denote $I:=\left[-A, A\right]$.
\end{enumerate}
\begin{theorem}\label{theorem1}
Let $f,\,u_{0},\, B$ satisfy { Hypothesis A}. Then the sequence of solutions to \eqref{ibvp.parab} converges {\it a.e.} to the unique entropy solution of \eqref{ivp.cl} as $\varepsilon\rightarrow0$.
\end{theorem}

In the literature, convergence of viscous approximations where $B\equiv 1$ are obtained under the assumption that the initial condition $u_0$ is a function of bounded variation. We are able to generalize this result to the case of a non-constant $B$ only when $u_0$ belongs to the space $W^{1,\infty}_c({\Omega})$, and is the content of the next result.

\begin{theorem}\label{theorem2}
Let $f,\,u_{0},\, B$ satisfy { Hypothesis B}. Then the sequence of solutions to \eqref{regularized.IBVP} converges {\it a.e.} to
the unique entropy solution of \eqref{ivp.cl} as $\varepsilon\rightarrow0$. 
 \end{theorem}
In Theorem \ref{theorem2}, we dealt with initial data with compact essential support in $\Omega$. But in the following result, we are able to capture the unique entropy solution of IBVP for conservation laws \eqref{ivp.cl} as the 
{\it a.e.} limit of quasilinear viscous approximations with initial data in $H^{1}_{0}(\Omega)\cap C(\overline{\Omega})$. 
\begin{theorem}\label{theorem3}
Let $f,\,u_{0},\, B$ satisfy { Hypothesis C}. Then the sequence of solutions to \eqref{regularized.IBVP.Compact} converges {\it a.e.} 
to the unique entropy solution of \eqref{ivp.cl} as $\varepsilon\rightarrow0$.
\end{theorem}
Problems of the form \eqref{ibvp.parab} are of interest in many physical situations. For instance, they appear in the viscous shallow water problem \cite{MR3000715}, and also in the equations of gas dynamics for viscous heat conducting fluid in eulerian coordinates \cite[p.256]{MR1301779}. Apart from the physical point of view, the study of viscous IBVP along with the convergence results play an important role in the numerical analysis of hyperbolic conservation laws. For instance, Kurganov and Liu \cite{MR2979844} proposed a finite volume method for solving general multi-dimentional system of conservation laws, in which they introduce an adaptive way of adding viscosity in the shock regions in order to capture  numerically stable entropy solution of hyperbolic conservation laws. As a result, the coefficient of the added numerical viscosity appears as a function of the solution and hence the convergence of the scheme to the entropy solution is similar to the problem addressed in this article.

When $\Omega=\mathbb{R}$ and $B\equiv1$, Oleinik  \cite{MR0151737} established the convergence of viscous approximations to the entropy solution of \eqref{ivp.cl}. This result is generalized to $\Omega=\mathbb{R}^{d}$ by Kruzhkov \cite{MR0255253}. For a special class of $2\times 2$ systems of hyperbolic conservations laws, DiPerna \cite{MR684413} showed the convergence of viscous approximations to a weak solution of the corresponding conservation laws. Bianchini and Bressan  \cite{MR2150387} established the BV estimates of viscous approximations (with the viscosity coefficient being the identity matrix) and showed that the {\it a.e.} limit of the viscous approximations is the unique entropy solution of the corresponding system of conservation laws.

For hyperbolic problems posed on a bounded domain, it is incorrect to impose boundary conditions as information from initial conditions propagates throughout the domain, and the information reaching the points on the boundary may not coincide with the prescribed boundary conditions. In other words IBVP for hyperbolic problems are ill-posed in general \cite{MR542510}. Thus we need to give meaning to the way boundary conditions are realized. The entropy condition of Bardos-Leroux-Nedelec (BLN entropy condition) introduced in \cite{MR542510} gives a way to interpret the boundary conditions.  Since the BLN entropy condition  is introduced for solutions which are of bounded variation as these functions have a well-defined trace on the boundary, it can not be extended to a problem where the solutions are not of bounded variation. For entropy conditions we refer the reader to Otto \cite{MR1387428}, Carrillo \cite{MR1709116}, Vallet \cite{MR1794912} in the $L^{\infty}$-setting,  and to Porretta and Vovelle \cite{MR1974461}, Ammar {\it et al.} \cite{MR2254186} in $L^{1}$- setting. Dubois and Le Floch \cite{MR922200} derived a boundary entropy condition satisfied by the limit of viscous approximations for a system of conservation laws in the domain $\left\{(x,t):\,\,x>0,\,t>0\right\}$, assuming that   the sequence of quasilinear viscous approximations $\left\{u^{\varepsilon}\right\}$ is bounded in $W^{1,1}_{loc}(\mathbb{R}^{+}\times\mathbb{R}^{+},\mathbb{R}^{d})$ and converges in $L^{1}_{loc}(\mathbb{R}^{+}\times\mathbb{R}^{+})$ to a limit function $u$.

Theorem~\ref{theorem1} and Theorem~\ref{theorem2} are proved by showing that the sequence $\left\{u^{\varepsilon}\right\}$ has an {\it a.e.} convergent subsequence, and its limit is an entropy solution of \eqref{ivp.cl}. Since entropy solution is unique \cite{MR542510}, we conclude that the entire sequence of viscous approximations $\left\{u^{\varepsilon}\right\}$ converges {\it a.e.} to the unique entropy solution. Existence of an {\it a.e.} convergent subsequence of $\left\{u^{\varepsilon}\right\}$ is proved by establishing  a uniform BV-estimate on the sequence $u^{\varepsilon}$.  When $B\equiv1$, Bardos {\it et al.} \cite{MR542510} established such a BV-estimate by employing clever multipliers. Their technique is used to obtain estimates on time derivative of $u^{\varepsilon}$, and the same technique does not yield  $L^{1}(\Omega_{T})$ estimates on the gradient of $u^{\varepsilon}$. However, using a different multiplier, we establish a bound on gradients of $u^{\varepsilon}$.  
 
The plan of the paper is as follows. In section 2, we show the well-posedness of solution to the IBVP \eqref{ibvp.parab}. In section 3, we prove the higher regularity of solutions of  \eqref{ibvp.parab} using an existence-cum-regularity result in H\"{o}lder spaces.
In section 4, we prove the BV estimates on solutions to \eqref{ibvp.parab}. In section 5, we prove Theorem \ref{theorem1}, 
in section 6, we prove Theorem \ref{theorem2} and finally in section 7, we prove Theorem \ref{theorem3}.

\section{Well-posedness of generalized viscosity problem}\label{section.wpgvp}
Even though we assume higher smoothness on initial condition $u_0$ and viscosity coefficient $B$ in  Theorem~\ref{theorem1} and Theorem~\ref{theorem2}, for the purposes of this section it is enough to assume that $u_0\in L^2(\Omega)$, and $B$ is continuous on $\R$ and bounded. 
\subsection{Existence}
We show the existence of a weak solution $u^\varepsilon$ to the IBVP \eqref{ibvp.parab} in the standard function space $W(0,T)$ defined by
\begin{eqnarray*}
W(0,T):=\left\{u\in L^{2}(0,T; H^{1}_{0}(\Omega))\,\,:\,\,u_{t}\in L^{2}(0,T\,;\,H^{-1}(\Omega))\right\}, 
\end{eqnarray*}
in the following sense: 
 For a.e. $0\leq t \leq T$, and for every $v\in H_{0}^{1}(\Omega)$ the folllowing equalities hold:
 \begin{subequations}\label{ibvp.parab.weak}
 \begin{eqnarray}
  \langle u_{t},v \rangle + \epsilon\hspace{0.1cm}\displaystyle\sum_{j=1}^{d}\int_{\Omega}\hspace{0.1cm}B(u)\hspace{0.1cm}\frac{\partial u}{\partial x_{j}}\hspace{0.1cm}\frac{\partial v}{\partial x_{j}}\hspace{0.1cm}dx + \displaystyle\sum_{j=1}^{d}\int_{\Omega}\hspace{0.1cm}f_{j}^{'}(u)\hspace{0.1cm}\frac{\partial u}{\partial x_{j}}\hspace{0.1cm}v\hspace{0.1cm}dx =0,\label{ibvp.parab.weak.a}\\
u(0)=u_{0}(x),\label{ibvp.parab.weak.b}
 \end{eqnarray}
\end{subequations}
where $\langle\cdot,\cdot\rangle$ denotes the  duality pairing between $H^{-1}(\Omega)$ and $H^{1}_{0}(\Omega)$.
Note that the condition \eqref{ibvp.parab.weak.b} is meaningful since the space $W(0,T)$ can be identified with a subset of $C([0,T];L^2(\Omega))$.
Since $\varepsilon$ remains fixed throughout this section, we write $u$ instead of $u^\varepsilon$.

Existence of a solution to the problem  \eqref{ibvp.parab} is obtained as a limit of a sequence of approximations $u_m$  given by Galerkin method. In Galerkin method, we take a sequence of functions $(w_{k})$ which is an orthogonal basis for the space $H_{0}^{1}(\Omega)$, and an orthonormal basis of $L^{2}(\Omega)$.  For each $m\in \mathbb{N}$, let $V_{m}$ denote the subspace of $L^{2}(\Omega)$ spanned by the first $m$ functions, namely, $w_{1},w_{2},\cdots,w_{m}$ of the sequence $(w_{k})$. The Galerkin approximation $u_m\in V_m$ is required to satisfy the equation \eqref{ibvp.parab.weak.a} with $u=u_m$ for every $v\in V_m$ along with the initial condition
\begin{equation}\label{galerk.approx.initial}
    (u_{m}(0),w_{k}) = (u_{0},w_{k}), \hspace{0.2cm}k=1,2,\cdots,m.
 \end{equation}
Thus $u_m$ satisfies for $k=1,2,\cdots,m$ and for {\em a.e.} $t\in [0,T]$ 
\begin{equation}\label{eqnchap2114}
 \langle u_{m}^{\prime},w_{k}\rangle + \epsilon\displaystyle\sum_{j=1}^{d}\int\limits_{\Omega}\hspace{0.1cm}B(u_{m})\hspace{0.1cm}\frac{\partial u_{m}}{\partial x_{j}}\frac{\partial w_{k}}{\partial x_{j}}dx + \displaystyle\sum_{j=1}^{d}\int\limits_{\Omega}\hspace{0.1cm}f_{j}^{'}(u_{m})\frac{\partial u_{m}}{\partial x_{j}}w_{k}dx =0,~k=1,2,\ldots,m 
  \end{equation}
 and the initial condition \eqref{galerk.approx.initial}. Since $u_{m}:[0,T]\to H_{0}^{1}(\Omega)$, we may write
\begin{equation}\label{galerk.approx}
 u_{m}(t)=\displaystyle\sum_{k=1}^{m}c_{m}^{k}(t)\hspace{0.1cm}w_{k}.
\end{equation}
Substituting the expression \eqref{galerk.approx} in \eqref{eqnchap2114} and \eqref{galerk.approx.initial} yields the following initial value problem for the system of ordinary differential equations  satisfied by the coefficients $c_m^k$ in  \eqref{galerk.approx} given by
\begin{subequations}\label{ode.system}
 \begin{eqnarray}
 {c_{m}^{k}}^{\prime} + \epsilon\displaystyle\sum_{j=1}^{d}\left(B(u_{m})~\frac{\partial u_{m}}{\partial x_{j}}\hspace{0.1cm},~ \frac{\partial w_{k}}{\partial x_{j}}\right) + \displaystyle\sum_{j=1}^{d}\left(f_{j}'(u_{m})~\frac{\partial u_{m}}{\partial x_{j}} ,w_{k}\right) =0,\\
c_{m}^{k}(0) = (u_{0},w_{k}), 
 \end{eqnarray}
\end{subequations}
using the fact that the duality product $ \langle u_{m}^{\prime},w_{k}\rangle$ reduces to the $L^2(\Omega)$ inner product of 
$u_{m}^{\prime},\,\,w_{k}\in L^{2}(\Omega)$.

Setting $\bfm{c}_m(t) = (c_m^1(t), c_m^2(t),\cdots,c_m^m(t))^T$, and $\bfm{g} = (g_1, g_2,\cdots,g_m)^T$, the system of equations \eqref{ode.system} may be written as ${\bfm{c}_m}' = \bfm{g}(\bfm{c}_m)$, where
\begin{multline}\label{ode.rhs}
g_i(\bfm{c}_m) = - \epsilon\hspace{0.1cm}\displaystyle\sum_{j=1}^{d}\hspace{0.1cm}\left(B\left(\displaystyle\sum_{k=1}^{m}c_{m}^{k}\hspace{0.1cm}w_{k}\right)\hspace{0.1cm}\displaystyle\sum_{k=1}^{m}c_{m}^{k}\hspace{0.1cm}\frac{\partial w_{k}}{\partial x_{j}},\frac{\partial w_{i}}{\partial x_{j}}\right) \\
+ \displaystyle\sum_{j=1}^{d}\left(f_{j}{'}\left(\displaystyle\sum_{k=1}^{m}c_{m}^{k}(t)\hspace{0.1cm}w_{k}\right)\displaystyle\sum_{k=1}^{m}c_{m}^{k}(t)\hspace{0.1cm}\frac{\partial w_{k}}{\partial x_{j}}, w_{i}\right),
\end{multline}
for $i=1,2,\cdots,m$.  It is a simple matter to check that the following inequality holds, in view of the expressions \eqref{ode.rhs}:
\begin{equation}\label{eqnchap12}
 \|\bfm{g}(\bfm{c})\|_2\leq P\hspace{0.1cm}\|\bfm{c}\|_2 \hspace{0.5cm}\mbox{for all } \bfm{c}\in\mathbb{R}^{m},
\end{equation}
where $\|\cdot\|_2$ denotes the Euclidean norm in $\R^m$, and $P$ is given by
$$P= \sqrt{\displaystyle\sum_{i=1}^{d}M_{g_{i}}^{2}},$$
where
\begin{multline*}
M_{g_{i}} = \left\{\varepsilon \|B\|_{L^{\infty}(\mathbb{R})}\displaystyle\sum_{j=1}^{d}\hspace{0.1cm}\left(\displaystyle\sum_{k=1}^{m}\Big\|\frac{\partial w_{k}}{\partial x_{j}}\Big\|_{L^{2}(\Omega)}^{2}\Big\| \frac{\partial w_{i}}{\partial x_{j}}\Big\|_{L^{2}(\Omega)}^{2}\right)^{\frac{1}{2}}\right.\\
 \left . +M\displaystyle\sum_{j=1}^{d}\hspace{0.1cm}\left(\displaystyle\sum_{k=1}^{m}\Big\|\frac{\partial w_{k}}{\partial x_{j}}\Big\|_{L^{2}(\Omega)}^{2} \Big\| w_{i}\Big\|_{L^{2}(\Omega)}^{2}\right)^{\frac{1}{2}}\right\},
 \end{multline*}
\bea\label{bound.fdash}
M =\displaystyle\max_{1\leq j\leq d}\left(\sup\left\{\big|f_{j}^{'}(y)\big|\hspace{0.2cm}:\hspace{0.2cm}y\in \mathbb{R}\right\}\right).
\eea

Since $g$ is continuous, it follows from a result in the theory of ordinary differential equations \cite[p.73]{vra_04a} that the initial value problem \eqref{ode.system} has a global solution. Thus each of the Galerkin approximations $u_m$ are defined on the interval $[0,T]$.

The following result gives an a priori bound on the Galerkin approximations $u_{m}$ and their derivatives $u_m'$, which helps in extracting a subsequence of the sequence $(u_m)$ that serves our purpose (of establishing the existence of solutions to IBVP for generalized viscosity problem).  Since its proof is on the lines of \cite[p.354]{eva_98a}, we omit the same. The only difference in its proof when compared to  \cite[p.354]{eva_98a} is the presence of terms involving the nonlinearity $f_j^\prime(.)$, $B(.)$ which is estimated by its $L^\infty(\mathbb{R})$ norm.

\begin{theorem}(Energy Estimate)\label{chap2thm2}
There exists a constant $C > 0$  such that for every $m\in\mathbb{N}$, 
\begin{equation}\label{eqnchap29}
 \max_{0\leq t\leq T}\|u_{m}\|_{L^{2}(\Omega)} + \|u_{m}\|_{L^{2}(0,T;H^{1}_{0}(\Omega))} + \|u_{m}'\|_{L^{2}(0,T;H^{-1}(\Omega))} \leq  C\|u_{0}\|_{L^{2}(\Omega)}.
  \end{equation}
\end{theorem}

The next result asserts the existence of convergent subsequence of Galerkin approximations.
\begin{theorem}\label{chap2prop5}
There exists a $u\in L^{2}(0,T;L^{2}(\Omega))$ and a subsequence of the sequence of Galerkin approximations $(u_{m})$ that converges to $u$  in $L^{2}(0,T;L^{2}(\Omega))$ as $n\to\infty$. Consequently, a further subsequence of it converges pointwise for a.e. $t\in [0,T]$ and a.e. $x\in \Omega$.
\end{theorem}

We prove Theorem~\ref{chap2prop5} by closely following the arguments from  \cite[p.469]{lad-etal_68a}. However for the sake of clarity, we provide all the details.

\begin{lemma}\cite[p.469]{lad-etal_68a}\label{sivalemma} For each $k\in \mathbb{N}$ and $m\geq k$ , let $c_m^k$ be the coefficients defined by \eqref{galerk.approx}. Then
 \begin{enumerate}
  \item[(i).] there exists a subsequence $(m^{(1)},m^{(2)},\cdots, m^{(n)}, \cdots)$ of  indices $m$, and a sequence $(c^k)$ of functions in $C[0,T]$ such that for each fixed $k\in \mathbb{N}$ the sequence $(c_{m^{(n)}}^k)_{m^{(n)}\geq k}$ converges to  $c^k$ in $C[0,T]$.
\item[(ii).] for each fixed $t\in[0,T]$, the sequence $(c^k)$ obtained in $(i)$ belongs to $\ell^{2}$.
 \end{enumerate}
\end{lemma}
\begin{proof}
Recall that for each $k\in \mathbb{N}$ and every $m\geq k$,  the function  $c_{m}^k:[0,T]\to \mathbb{R}$ is given by
\begin{equation}\label{l.weak.solution}
c_m^k(t)= \left(u_{m}(\cdot,t),w_{k}\right).
\end{equation}
Uniform boundedness of the sequence $(c_m^k)_{m\geq k}$ is a direct consequence of the energy estimate (see Theorem \ref{chap2thm2}). By a simple calculation involving the weak formulation \eqref{eqnchap2114} and using the Cauchy-Schwarz and H\"{o}lder inequalities, we arrive at
\begin{equation*}
 \Big| c_m^k(t+h)-c_m^k(t)\Big|\leq \|u_{m}\|_{L^{2}(0,T;H^{1}_{0}(\Omega))}\hspace{0.1cm}\left(\epsilon \displaystyle\sum_{j=1}^{d}\|\frac{\partial w_{k}}{\partial x_{j}}\|_{L^{2}(\Omega)}+ \displaystyle\sum_{j=1}^{d}\|f_{j}'\|_{\infty}\right)\hspace{0.1cm}\sqrt{h}.
\end{equation*}
Using the energy estimate  and the last inequality, equicontinuity of the sequence $(c_m^k)_{m\geq k}$ follows. Applying Ascoli-Arzela theorem with $k=1$, we conclude the existence of a subsequence $(m_1^{(1)},m_1^{(2)},\cdots, m_1^{(n)}, \cdots)$ of  indices $m$, and a $c^1\in C[0,T]$ such that  the sequence $(c_{m_1^{(n)}}^1)$ converges to  $c^1$ uniformly on $[0,T]$.  Applying Ascoli-Arzela theorem with $k=2$ for the sequence $(c_{m_1^{(n)}}^2)$, we get a further subsequence of the subsequence $m_1^{(n)}$ denoted by $m_2^{(n)}$, and  a $c^2\in C[0,T]$ such that  the sequence $(c_{m_2^{(n)}}^2)$ converges to  $c^2$ uniformly on $[0,T]$. Continuing in this fashion, we get a sequence $(c^k)$ of functions in $C[0,T]$. By a diagonal argument we obtain the subsequence $(m_n^{(n)})$ for which the assertion $(i)$ of the lemma holds. For reasons of notational convenience we still index this subsequence using $m\in\mathbb{N}$.

Let us turn to the proof of $(ii)$ now. For each $t\in[0,T]$ we have
\begin{equation*}\label{eqnchap2204}
 \Big|c^{k}(t)\Big|= \displaystyle\lim_{m\to\infty}\Big|\left(u_{m}(t),w_{k}\right)\Big|,
\end{equation*}
and hence we get
\begin{eqnarray*}
 \displaystyle\sum_{k=1}^{\infty}\Big|c^{k}(t)\Big|^{2} &=& \displaystyle\sum_{k=1}^{\infty}\left(\displaystyle\lim_{m\to\infty}\Big|\left(u_{m}(t),w_{k}\right)\Big|\right)^{2}\nonumber\\
  &=& \displaystyle\lim_{p\to\infty}\left(\displaystyle\lim_{m\to\infty}\left(\displaystyle\sum_{k=1}^{p}\Big|\left(u_{m}(t),w_{k}\right)\Big|^{2}\right)\right).
\end{eqnarray*}
Using Bessel's inequality, the last equation yields  
\begin{eqnarray*}
 \displaystyle\sum_{k=1}^{\infty}\Big|c^{k}(t)\Big|^{2} \le \displaystyle\limsup_{m\to\infty}\|u_{m}\|_{L^{2}(\Omega)}^{2}.
 \end{eqnarray*}
Using the energy estimate again, we get $(c^k)\in \ell^{2}$. This completes the proof of the lemma.
\end{proof}

\noindent In the proof of Theorem~\ref{chap2prop5} the following lemma from \cite{lad-etal_68a} will be used.
\begin{lemma}\cite[p.72]{lad-etal_68a}\label{chap2prop3}
Let $\left\{\psi_{k}\right\}_{k=1}^{\infty}$ be an orthonormal basis for $L^{2}(\Omega)$. Then for any $\nu > 0$, there exists a number $N_{\nu}$ such that for all $v\in H^{1}(\Omega)$ the inequality 
 \begin{equation}\label{eqnchap247}
  \|v\|_{L^{2}(\Omega)}\leq \left(\displaystyle\sum_{k=1}^{N_{\nu}}(v,\psi_{k})^{2}\right)^{\frac{1}{2}} + \nu \|v\|_{H^{1}(\Omega)}.
 \end{equation}
holds. 
\end{lemma}
{\bf Proof of Theorem~\ref{chap2prop5}:} Let the subsequence of indices asserted by Lemma~\ref{sivalemma} still be denoted by $m$ by dropping the superscript. For proving Theorem~\ref{chap2prop5}, it is enough to show that the sequence $u_m$   is a Cauchy sequence in $L^{2}(0,T;L^{2}(\Omega))$. Let $\nu>0$ be given. Let $m_1, m_2$ be any two natural numbers satisfying $m_1\geq m_{2}$. Applying Lemma~\ref{chap2prop3} with the given $\nu>0$, we get a $N_{\nu}$ such that the inequality  \eqref{eqnchap247} holds for $v= u_{m_1}-u_{m_{2}}$. Thus we have
\begin{equation}\label{eqnchap261}
 \|u_{m_1}-u_{m_{2}}\|_{L^{2}(\Omega)}\leq \left(\displaystyle\sum_{k=1}^{N_{\nu}}(u_{m_1}-u_{m_{2}}, w_{k})^{2}\right)^{\frac{1}{2}} + \nu \|u_{m_1}-u_{m_{2}}\|_{H^{1}_{0}(\Omega)}.
\end{equation}
Squaring the above inequality and using the inequality $2ab \leq a^{2} + b^{2}$ which is valid for any $a,b\in \mathbb{R}$, we get
\begin{equation}\label{ladylemmaconsequence}
 \|u_{m_1}-u_{m_{2}}\|_{L^{2}(\Omega)}^{2}\leq (1+ \nu)\hspace{0.1cm}\displaystyle\sum_{k=1}^{N_{\nu}}(u_{m_1}-u_{m_{2}},w_{k})^{2} + (\nu^{2} + \nu) \|u_{m_1}-u_{m_{2}}\|_{H^{1}_{0}(\Omega)}^{2}.
\end{equation}

On the other hand, assertion $(ii)$ of  Lemma~\ref{sivalemma} says that the function 
$$  q(x,t): =\displaystyle\sum_{k=1}^{\infty}c^{k}(t)w_{k}$$
is such that $q(.,t)\in L^2(\Omega)$. Note that $  \left(u_{m}-q, w_{j}\right)= \left(u_{m}, w_{j}\right) -c^{j}(t). $
 In view of  \eqref{l.weak.solution}, we have $ \left(u_{m}-q, w_{j}\right)\rightarrow 0$ uniformly in  $[0,T]$. Hence  $ \left(u_{m}-q, w_{j}\right)$ is a Cauchy sequence in $C[0,T]$. Thus, for the given $\nu>0$, and the integer $N_{\nu}>0$, there exists an $N\in\mathbb{N}$ such that for  $m_1\geq m_{2}\geq N$ we get 
 \begin{equation}\label{eqnchap264}
\sup\limits_{t\in [0,T]}\left\{ \displaystyle\sum_{k=1}^{N_{\nu}}\big|(u_{m_1}-u_{m_{2}},w_{k})\big|^{2}\right\} < \nu.
 \end{equation}

Integrating w.r.t. $t$ variable on both sides of the  inequality \eqref{ladylemmaconsequence}, and using \eqref{eqnchap264}, we obtain
\begin{eqnarray}\label{eqnchap265}
 \int_{0}^{T}\|u_{m_1}-u_{m_{2}}\|_{L^{2}(\Omega)}^{2}\,dt &\leq&  (1+ \nu)T\hspace{0.1cm}\nu + (\nu^{2} + \nu)\|u_{m_1}-u_{m_{2}}\|_{L^{2}(0,T;H^{1}_{0}(\Omega))}^{2}
\end{eqnarray}
for all $m_{1}\geq m_{2}\geq N$. In view of the inequality  
\begin{eqnarray*}\label{eqnchap266}
 \|u_{m_1}-u_{m_{2}}\|_{L^{2}(0,T;H^{1}_{0}(\Omega))}^{2} 
 &\leq& 2\left(\|u_{m_1}\|_{L^{2}(0,T;H^{1}_{0}(\Omega))}^{2} + \|u_{m_{2}}\|_{L^{2}(0,T;H^{1}_{0}(\Omega))}^{2}\right),
\end{eqnarray*}
using energy estimate, the inequality \eqref{eqnchap265} becomes
\begin{equation}\label{eqnchap267}
 \|u_{m_1}-u_{m_{2}}\|_{L^{2}(0,T;H^{1}_{0}(\Omega))}^{2} \leq (1+ \nu)T\hspace{0.1cm}\nu + 4 C^{2}(\nu^{2}+ \nu)\|u_{0}\|_{L^{2}(\Omega)}^{2}.\nonumber
\end{equation}
This shows that the sequence $(u_m)$ is Cauchy in $L^2(0,T;L^2(\Omega))$.  This completes the proof of the theorem.

We now prove the existence of a weak solution to the problem \eqref{ibvp.parab}.
\begin{theorem}[Existence of a weak solution]\label{chap2thm3}
There exists a weak solution of the generalized viscosity problem \eqref{ibvp.parab}.
\end{theorem}
\begin{proof} Let $u$ be as asserted by Theorem~\ref{chap2prop5}. We need to show that the function $u$ satisfies \eqref{ibvp.parab.weak.a} and the initial condition \eqref{ibvp.parab.weak.b}.
Note that
$$S = \left\{\displaystyle\sum_{k=1}^{N}d^{k}(t)\hspace{0.1cm}w_{k}\hspace{0.1cm}:\hspace{0.1cm}d^{k}\in C^{1}([0,T]),\, N\in\N\right\}$$
is a dense subspace of $L^{2}(0,T; H^{1}_{0}(\Omega))$. Let $v\in S$ be chosen such that $N\le m$.  Then by definition, we have the set of $N$ $C^1$ functions $\{d^k\}$. Multiplying the equation \eqref{eqnchap2114} by $d^k$, summing over  $k=1,2,\ldots,N$, and integrating over the interval $[0,T]$, we get
\begin{equation}\label{eqnchap2999}
 \int_{0}^{T}\hspace{0.1cm}\langle u_{m}^{'}, v \rangle\hspace{0.1cm}dt + \epsilon \displaystyle\sum_{j=1}^{d}\int_{0}^{T}\left(\frac{\partial u_{m}}{\partial x_{j}},B(u_{m})\frac{\partial v}{\partial x_{j}}\right)\hspace{0.1cm}dt + \displaystyle\sum_{j=1}^{d}\int_{0}^{T}\left(\frac{\partial u_{m}}{\partial x_{j}},f'_{j}(u_{m})v\right)\hspace{0.1cm}dt =0. 
\end{equation}
As a consequence of the energy estimate, Theorem~\ref{chap2prop5}, and using Banach-Alaoglu theorem, there exists a subsequence $\left\{u'_{m_{l}}\right\}_{l=1}^{\infty}$ of $\left\{u'_{m}\right\}_{m=1}^{\infty}$ (as usual, we drop the index $l$) such that $u'_{m}\rightharpoonup u'$  in $L^{2}(0,T;H^{-1}(\Omega))$. Therefore, we have
\begin{equation}\label{weaksolution.convergence.eqn1}
 \int_{0}^{T}\hspace{0.1cm}\langle u'_{m}, v\rangle\hspace{0.1cm}dt \to \int_{0}^{T}\hspace{0.1cm} \langle u', v\rangle \hspace{0.1cm}dt, \,\, \forall v\in L^{2}(0,T; H^{1}_{0}(\Omega)).
\end{equation}
For $j=1,2,\ldots,d$, since $f_j$  are $C^1$ functions and $B$ is a continuous function, by Theorem~\ref{chap2prop5},  for a.e. $(x, t)\in \Omega_T$ we have the following convergences as $m\to\infty$ :
 \begin{eqnarray}\label{eqnchap292}
   f'_{j}(u_{m}) \to f'_{j}(u),\,\,\,   B(u_{m}) \to  B(u),\nonumber
  \end{eqnarray}
As a consequence, we have
\begin{equation}\label{eqnchap293}
\left .\begin{array}{rl}
 f'_{j}(u_{m})v \to  f'_{j}(u)v\hspace{0.2cm}&\mbox{in}\hspace{0.2cm}L^{2}(0,T;L^{2}(\Omega)),\quad\\
 B(u_{m})\frac{\partial v}{\partial x_j} \to  B(u)\frac{\partial v}{\partial x_j}\hspace{0.2cm}&\mbox{in}\hspace{0.2cm}L^{2}(0,T;L^{2}(\Omega))
 \end{array}\right\}
\end{equation}
By energy estimate \eqref{eqnchap29}, we have
\begin{eqnarray}\label{weaksolution.convergence.eqn2}
\frac{\partial u_{m}}{\partial x_{j}}\rightharpoonup \frac{\partial u}{\partial x_{j}}
\mbox{ in } L^2(0,T;L^2(\Omega)).
\end{eqnarray}
Using the information \eqref{weaksolution.convergence.eqn1}-\eqref{weaksolution.convergence.eqn2} in the equation \eqref{eqnchap2999}, we get for each $v\in S$ with $N\leq m$,
\begin{equation*}\label{eqnchap295}
 \int_{0}^{T}\hspace{0.1cm}\langle u', v\rangle\hspace{0.1cm}dt + \epsilon\displaystyle\sum_{j=1}^{d}\int_{0}^{T}\left(\frac{\partial u}{\partial x_{j}},B(u)\frac{\partial v}{\partial x_{j}}\right)\hspace{0.1cm}dt + \displaystyle\sum_{j=1}^{d}\int_{0}^{T}\left(\frac{\partial u}{\partial x_{j}},f'_{j}(u)v\right)\hspace{0.1cm}dt =0. 
\end{equation*}
Since $S$ is dense in $L^{2}(0,T;H^{1}_{0}(\Omega))$, the above equation holds for all $v\in L^{2}(0,T;H^{1}_{0}(\Omega))$ and can be written as
\begin{equation}\label{eqnchap297}
\int_{0}^{T}\hspace{0.1cm}\langle u', v\rangle\hspace{0.1cm}dt + \epsilon\displaystyle\sum_{j=1}^{d}\int_{0}^{T}\left(B(u)\frac{\partial u}{\partial x_{j}},\frac{\partial v}{\partial x_{j}}\right)\hspace{0.1cm}dt + \displaystyle\sum_{j=1}^{d}\int_{0}^{T}\left(f'_{j}(u)\frac{\partial u}{\partial x_{j}},v\right)\hspace{0.1cm}dt =0. 
\end{equation}
By choosing the test functions $v(x,t)=\varphi(t) w(x)$ where $\varphi\in {\cal D}(0,T)$, $w\in H^1_0(\Omega)$, we conclude that $u$ satisfies \eqref{ibvp.parab.weak.a}. 

We now turn to the proof of \eqref{ibvp.parab.weak.b}. We know that $u\in L^{2}(0,T;H^{1}_{0}(\Omega))$ and $u{'}\in L^{2}(0,T;H^{-1}(\Omega))$.  This implies that $u\in C([0,T];L^{2}(\Omega))$ and therefore $u(0)\in L^2(\Omega)$. 

Choosing $v=c(t)w_k$ where $v\in C^1[0,T]$ is such that $c(0)=1, c(T)=1$ in the equation \eqref{eqnchap297},  integration by parts yields 
 \begin{equation}\label{eqnchap299}
 \int_{0}^{T}\hspace{0.1cm}-(u, v')\hspace{0.1cm}dt + \epsilon\displaystyle\sum_{j=1}^{d}\int_{0}^{T}\left(B(u)\frac{\partial u}{\partial x_{j}},\frac{\partial v}{\partial x_{j}}\right)\hspace{0.1cm}dt + \displaystyle\sum_{j=1}^{d}\int_{0}^{T}\left(f_{j}'(u)\frac{\partial u}{\partial x_{j}},v\right)\hspace{0.1cm}dt = (u(0),v(0)). 
\end{equation}
Again from \eqref{eqnchap2999}, get
\begin{multline}\label{eqnchap2100}
 \int_{0}^{T}\hspace{0.1cm}-(u_{m}, v')\hspace{0.1cm}dt +\epsilon \displaystyle\sum_{j=1}^{d}\int_{0}^{T}\left(B(u_{m})\frac{\partial u_{m}}{\partial x_{j}},\frac{\partial v}{\partial x_{j}}\right)\hspace{0.1cm}dt\\
  + \displaystyle\sum_{j=1}^{d}\int_{0}^{T}\left(f_{j}'(u_{m})\frac{\partial u_{m}}{\partial x_{j}},v\right)\hspace{0.1cm}dt = (u_{m}(0),v(0)). 
\end{multline}
Now passing to the limit in \eqref{eqnchap2100} as $m\to \infty$, we get
\begin{equation}\label{eqnchap2101}
 \int_{0}^{T}\hspace{0.1cm}-(u, v')\hspace{0.1cm}dt + \epsilon\displaystyle\sum_{j=1}^{d}\int_{0}^{T}\left(B(u)\frac{\partial u}{\partial x_{j}},\frac{\partial v}{\partial x_{j}}\right)\hspace{0.1cm}dt + \displaystyle\sum_{j=1}^{d}\int_{0}^{T}\left(f_{j}'(u)\frac{\partial u}{\partial x_{j}},v\right)\hspace{0.1cm}dt = (u_{0},v(0)). 
\end{equation}
Since $(w_k)$ is an orthonormal for $L^2(\Omega)$,  the equations \eqref{eqnchap299} and \eqref{eqnchap2101} yield $u(0)=u_0$.  This completes the proof of the theorem.
\end{proof}

\subsection{Uniqueness}

Any solution of generalized viscosity problem \eqref{ibvp.parab} satisfies the following maximum principle \cite[p.60]{MR1304494}:
\begin{theorem}[Maximum principle]\label{chap3thm1}
Let $f:\mathbb{R}\to\mathbb{R}^{d}$ be a $C^{1}$ function and $u_{0}\in L^{\infty}(\Omega)$. Then any solution $u$ of generalized viscosity problem \eqref{ibvp.parab} in $W(0,T)$  satisfies the bound
\begin{equation}\label{eqnchap303}
||u^{\varepsilon}(\cdot,t)||_{L^{\infty}(\Omega)}\hspace{0.1cm}\leq\hspace{0.1cm}||u_{0}||_{L^{\infty}(\Omega)}\hspace{0.1cm}a.e.\,\,t\in(0,T).
\end{equation}
\end{theorem}

We adapt the proof of a result in \cite[p.150]{lad-etal_68a} concerning linear equations to our situation, and obtain the following uniqueness theorem.
\begin{theorem}\textbf{$\left(\mbox{Uniqueness theorem}\right)$}\cite[p.150]{lad-etal_68a}\label{chap10thm1}
Let  $f\in \left(C^{1}(\mathbb{R})\right)^{d}, B\in C^{1}(\mathbb{R})\cap L^{\infty}(\mathbb{R}), u_{0}\in L^{\infty}(\Omega)$. Then the generalized viscosity problem \eqref{ibvp.parab} has at most one solution in $W(0,T)$.
 \end{theorem}
Let $u_1$ and $u_2$ be solutions of \eqref{ibvp.parab} in $W(0,T)$ and denote $w := u_1-u_2$. In view of \eqref{eqnchap297}, the function $w$   satisfies the following equation for all $\eta\in L^{2}(0,T;H^{1}_{0}(\Omega))$:  
\begin{eqnarray}\label{uniqueeqn3chap10}
 \int_{0}^{T}{_{_{H^{-1}(\Omega)}}}\!\langle w_{t},\eta\rangle\!{_{_{H^{1}_{0}(\Omega)}}}\hspace{0.1cm}dt + \varepsilon\displaystyle\sum_{j=1}^{d}\int_{0}^{T}\int_{\Omega}\left(B(u_1)\frac{\partial u_1}{\partial x_{j}}-B(u_2)\frac{\partial u_2}{\partial x_{j}}\right)\frac{\partial \eta}{\partial x_{j}}\hspace{0.1cm}dx\hspace{0.1cm}dt\nonumber\\
 - \displaystyle\sum_{j=1}^{d}\int_{0}^{T}\int_{\Omega}\left(f_{j}(u_1)-f_{j}(u_2)\right)\hspace{0.1cm}\frac{\partial \eta}{\partial x_{j}}\hspace{0.1cm}dx\hspace{0.1cm}dt &=0.\hspace*{0.2in} 
\end{eqnarray}
Using fundamental theorem of calculus, we may write
 \begin{eqnarray*}
B(u_1)\frac{\partial u_1}{\partial x_{j}}-B(u_2)\frac{\partial u_2}{\partial x_{j}}  &=&w \int_{0}^{1} B^{\prime}(\tau u_1 +(1-\tau)u_2)\Big[\tau (u_1)_{x_{j}}+ (1-\tau)(u_2)_{x_{j}}\Big]d\tau\nonumber\\
&&\hspace*{2cm}+ \frac{\partial w}{\partial x_{j}}\int_{0}^{1} B(\tau u_1 + (1-\tau)u_2)\hspace{0.1cm}d\tau,\\
f_{j}(u_1)-f_{j}(u_2)&=& w(x,t)\int_{0}^{1} f_{j}^{\prime}(\tau u_1 + (1-\tau)u_2)\hspace{0.1cm}d\tau,\,\,j=1,2,\cdots,d.
 \end{eqnarray*}

Thus the function $w$ satisfies, in view of the equation \eqref{uniqueeqn3chap10}, for all $ \eta\in H^{1}(\Omega\times(0,T))$ such that $\eta(x,T)=0$ on $\Omega$ and $\eta=0$ on $\partial\Omega\times (0,T)$,
\begin{eqnarray}\label{eqn11chap10}
 \int_{0}^{T}{_{_{H^{-1}(\Omega)}}}\!\langle w_{t},\eta\rangle\!{_{_{H^{1}_{0}(\Omega)}}}\hspace{0.1cm}dt + \varepsilon\displaystyle\sum_{j=1}^{d}\int_{0}^{T}\int_{\Omega}\left(\tilde{a}(x,t)\frac{\partial w}{\partial x_{j}} +\tilde{b_{j}}(x,t) w\right)\frac{\partial\eta}{\partial x_{j}}\hspace{0.1cm}dx\hspace{0.1cm}dt\nonumber \\ -\displaystyle\sum_{j=1}^{d}\int_{0}^{T}\int_{\Omega}\tilde{c_{j}}(x,t)w\hspace{0.1cm}\frac{\partial \eta}{\partial x_{j}}\hspace{0.1cm}dx\hspace{0.1cm}dt &=0  
  \end{eqnarray} 
 where
\begin{subequations}\label{conv.eqn1}
\begin{eqnarray}
 \tilde{a}(x,t) :&=& \int_{0}^{1}B(\tau u_1 + (1-\tau)u_2)\hspace{0.1cm}d\tau,\\
 \tilde{b_{j}}(x,t) :&=& \int_{0}^{1}B^{\prime}(\tau u_1 + (1-\tau)u_2)\left[\tau (u_1)_{x_{j}} + (1-\tau)(u_2)_{x_{j}}\right]\hspace{0.1cm}d\tau,\\
 \tilde{c_{j}}(x,t) :&=& \int_{0}^{1} f_{j}^{\prime}(\tau u_1 + (1-\tau)u_2)\hspace{0.1cm}d\tau.
\end{eqnarray}
\end{subequations}

Thus $w$ is a weak solution of
\begin{eqnarray*}\label{ibvp.w.formal}
 \mathcal{L}w\equiv w_{t} -\varepsilon\sum_{j=1}^d\frac{\partial}{\partial x_j}\left(\tilde{a}(x,t)w_{x_j}\right) -\varepsilon\sum_{j=1}^d\frac{\partial}{\partial x_j}\left(\tilde{b_j}(x,t)w\right)+\sum_{j=1}^d\frac{\partial}{\partial x_j}\left(\tilde{c_j}(x,t)w\right)=0&\mbox{in }\Omega_{T}
\end{eqnarray*}
satisfying $w(x,t)= 0$ for $(x,t)\in \partial \Omega\times (0,T)$, and $w(x,0) = 0$ for $x\in \Omega$.
Let us introduce a sequence of auxiliary IBVPs corresponding to operators $ \mathcal{L}_{m}$, which are  obtained by regularizing the coefficients of the formal adjoint to $\mathcal{L}$, defined by the IBVP
\begin{subequations}\label{eqn33chap10}
\begin{eqnarray}
\mathcal{L}_{m}(\eta) =h(x,t)   &\mbox{in }\Omega_{T}, \\
\eta(x,t)= 0&\mbox{on }\partial\Omega\times[0,T],\\
\eta(x,T) = 0& \mbox{for } x\in \Omega,
\end{eqnarray}
\end{subequations}
where
\begin{eqnarray}
 \mathcal{L}_{m}(\eta):= -\eta_{t} -\varepsilon\displaystyle\sum_{j=1}^{d}\frac{\partial}{\partial x_{j}}\left(a^{m}(x,t)\eta_{x_{j}}\right) +\varepsilon\displaystyle\sum_{j=1}^{d}b_{j}^{m}(x,t)\eta_{x_{j}}-\displaystyle\sum_{j=1}^{d}c_{j}^{m}(x,t)\eta_{x_{j}},
\end{eqnarray}
where the coefficient functions $a^{m}, b_{j}^{m}, c_{j}^{m}$ are obtained by first extending the functions $\tilde{a},\tilde{b}_{j}$ and $\tilde{c}_{j}$ by zero outside $\Omega\times(0,T)$ and then regularizing the resultant functions using a sequence $\rho_{\frac{1}{m}}$, and then restricting the regularized functions to $\Omega_T$. Here we use standard sequence of mollifiers, but defined for $(x,t)\in\R^{d+1}$.

Study of the auxiliary problems \eqref{eqn33chap10} is based on the following existence-cum-higher regularity theorem for linear parabolic problems:
\begin{theorem}\cite[p.320]{lad-etal_68a}\label{thm2chap10}
 Let $l> 0$ be a real number such that $l\notin\mathbb{N}$ . Let the coefficients of the differential operator appearing in the IBVP
\begin{subequations}\label{eqn34chap10}
 \begin{eqnarray*}
  \frac{\partial u}{\partial t}-\displaystyle\sum_{i,j=1}^{d}a_{ij}(x,t)\frac{\partial^{2}u}{\partial x_{i}\partial x_{j}} + \displaystyle\sum_{i=1}^{d}a_{i}(x,t)\frac{\partial u}{\partial x_{j}}+ a(x,t)u &=& h(x,t)\hspace{0.5cm}\mbox{in}\hspace{0.1cm}\Omega_{T} \quad\\
 u&=&\Phi\hspace{1.2cm}\mbox{on}\hspace{0.1cm}\partial \Omega\times[0,T]\quad\\
  u(x,0) &=& u_{0}(x)\hspace{0.7cm}\mbox{on}\hspace{0.1cm}\Omega\times\left\{t=0\right\}
 \end{eqnarray*}
\end{subequations} belong to the class $C^{l,l/2}(\overline{\Omega_{T}})$ and the boundary $\partial\Omega$ belongs to the class $C^{l+2}$. Then for any $h\in C^{l,l/2}(\overline{\Omega_{T}})$, $u_{0}\in C^{l+2}(\overline{\Omega})$, $\Phi\in C^{l+2,l/2 +1}(\overline{\partial\Omega\times[0,T]})$ satisfying the compatibility condition
 of order $[l/2] +1$ {\it i.e.,} for $k=0,1,\cdots,[l/2] +1$ and $x\in\partial\Omega$, 
 \begin{equation}\label{compatibility.condns}
  \frac{\partial^{k} u}{\partial t^{k}}(x,t)\Big|_{t=0} =\frac{\partial^{k}\Phi}{\partial t^{k}}(x,t)\Big|_{t=0}
  \end{equation}
 holds. Then the problem \eqref{eqn34chap10} has a unique solution in $ C^{l+2,l/2 +1}(\overline{\Omega_{T}})$. Further $u$ satisfies
 $$|u|_{l+2;\Omega_{T}}\leq C\left(|h|_{l;\Omega_{T}} + |u_{0}|_{l+2;\Omega} +|\Phi|_{l+2;\Omega_{T}}\right).$$
\end{theorem}
We have the following result concerning the auxiliary problems \eqref{eqn33chap10}:
\begin{theorem}\label{unique.maximum}
Let $m\in\mathbb{N}$, $0<l<1$. Let $h$ be in $C^{l,\frac{l}{2}}(\overline{\Omega_{T}})$ such that $h(.,T)=0$. Then the  IBVP   \eqref{eqn33chap10}has a unique classical solution $\eta^{m}$. Further
\begin{enumerate}
\item for all $m\in\mathbb{N}$ and  $(x,t)\in \overline{\Omega_{T}}$, we have
\begin{equation}\label{uniquenessequation3}
\Big|\eta^{m}(x,t)\Big|\leq e^{T}\displaystyle\max_{(x,t)\in\Omega_{T}}\Big|h(x,t)\Big|. 
\end{equation}
\item there exists a $C> 0$ such that for all $m\in\mathbb{N}$, we have
\begin{equation}\label{Uniquenessequation4}
\|\nabla\eta^{m}\|_{L^{2}(\Omega_{T})}\leq C. 
\end{equation}
\end{enumerate}
\end{theorem}

\noindent {\bf Proof of Theorem~\ref{unique.maximum}}
Introducing the following change of variables
\begin{equation}\label{uniquenessequation7}
 \xi(x,t)=x,\,\,\tau(x,t)= T-t,
\end{equation}
and setting
\begin{equation}\label{eqn902chap10}
\eta(x,t):=U(\xi(x,t),\tau(x,t)),
\end{equation}
the backward parabolic equation satisfied by $\eta$ in the IBVP \eqref{eqn33chap10} is transformed into a forward parabolic equation that $U$ satisfies and the  IBVP satisfied by $U$ is given by  
\begin{subequations}\label{eqn35chap10}
 \begin{eqnarray}
  U_{\tau}(\xi,\tau) -\varepsilon\displaystyle\sum_{j=1}^{d}\alpha^{m}(\xi,\tau)U_{\xi_{j}\xi_{j}}(\xi,\tau) -\displaystyle\sum_{j=1}^{d}\beta_{j}^{m}(\xi,\tau) U_{\xi_{j}}(\xi,\tau) &=&\overline{h}(\xi,\tau)\hspace{0.1cm}\mbox{in}\hspace{0.1cm}\Omega_{T},\quad\\
U(\xi,\tau) &=&0\hspace{.2cm}\mbox{on}\hspace{0.2cm}\partial \Omega\times[0,T],\quad\quad\quad{}\\
U(\xi,0) &= &0\hspace{0.2cm}\mbox{on}\hspace{0.1cm}\Omega,
 \end{eqnarray}
\end{subequations}
where
\begin{eqnarray}\label{eqn905chap10}
 \alpha^{m}(\xi,\tau) &=& a^{m}(\xi,T-\tau),\nonumber\\
 \beta_{j}^{m}(\xi,\tau) &=& \varepsilon\frac{\partial a^{m}}{\partial \xi_{j}}(\xi,T-\tau) -\varepsilon b^{m}_{j}(\xi,T-\tau) + c_{j}^{m}(\xi,T-\tau),\nonumber\\
 \overline{h}(\xi,\tau) &=& h(\xi, T-\tau).
\end{eqnarray}
On applying Theorem \ref{thm2chap10} to the IBVP \eqref{eqn35chap10}, in view of \eqref{eqn902chap10}, we get the existence of a unique classical solution $\eta^{m}$ to the IBVP \eqref{eqn33chap10}. 

The estimate \eqref{uniquenessequation3} follows by applying Theorem~2.1 of \cite[p.13]{lad-etal_68a} to the IBVP \eqref{eqn35chap10}, and using equation \eqref{eqn902chap10}.

We now turn to the proof of the estimate \eqref{Uniquenessequation4}. Note that 
\begin{eqnarray}\label{eqn46chap10}
 \int_{0}^{T}\int_{\Omega}h\hspace{0.1cm}\eta^{m}\hspace{0.1cm}dx dt &=& \int_{0}^{T}\int_{\Omega}\mathcal{L}_{m}(\eta^{m})\hspace{0.1cm}\eta^{m}\hspace{0.1cm}dx\hspace{0.1cm}dt\nonumber\\
 &=& -\int_{0}^{T}{_{_{H^{-1}(\Omega)}}}\!\langle \eta^{m}_{t},\eta^{m}\rangle\!{_{_{H^{1}_{0}(\Omega)}}}\hspace{0.1cm}dt + \varepsilon\displaystyle\sum_{j=1}^{d}\int_{0}^{T}\int_{\Omega}a^{m}(x,t)\left(\eta^{m}_{x_{j}}\right)^{2} dx\hspace{0.1cm}dt\nonumber\\
 &&+ \varepsilon\displaystyle\sum_{j=1}^{d}\int_{0}^{T}\int_{\Omega}b_{j}^{m}(x,t)\eta^{m}_{x_{j}}\hspace{0.1cm}\eta^{m}-\displaystyle\sum_{j=1}^{d}\int_{0}^{T}\int_{\Omega}c_{j}^{m}(x,t)\eta^{m}_{x_{j}}\hspace{0.1cm}\eta^{m}dx\hspace{0.1cm}dt.\nonumber\\
 {}
\end{eqnarray}
Using $a^{m}\geq r$, $\eta^{m}(x,T)=0$ and $\frac{d}{dt}\langle\eta^{m}_{t},\eta^{m}\rangle = \frac{1}{2}\frac{d}{dt}\|\eta^{m}\|^{2}_{L^{2}(\Omega)}$ in equation \eqref{eqn46chap10}, we arrive at

\begin{eqnarray}\label{eqn54chap10}
  \varepsilon r\displaystyle\sum_{j=1}^{d}\int_{0}^{T}\int_{\Omega}\left(\eta^{m}_{x_{j}}\right)^{2} dx\hspace{0.1cm}dt + \frac{1}{2}\|\eta^{m}(0)\|^{2}_{L^{2}(\Omega)}  \hspace{8cm}  \nonumber\\
 \leq \|h\|_{L^{\infty}(\Omega_{T})}\|\eta^{m}\|_{L^{\infty}(\Omega_{T})}\mbox{Vol}(\Omega_{T}){} + \,\, \varepsilon\|\eta^{m}\|_{L^{\infty}(\Omega_{T})}\displaystyle\sum_{j=1}^{d}\int_{0}^{T}\int_{\Omega}\Big|b_{j}^{m}(x,t)\Big|\hspace{0.1cm}\Big|\eta_{x_{j}}^{m}\Big|\hspace{0.1cm}dx\,dt\nonumber\\
  +\, \, \|\eta^{m}\|_{L^{\infty}(\Omega_{T})}\displaystyle\sum_{j=1}^{d}\int_{0}^{T}\int_{\Omega}\Big|c_{j}^{m}(x,t)\Big|\hspace{0.1cm}\Big| \eta^{m}_{x_{j}}\Big|\hspace{0.1cm}dx\hspace{0.1cm}dt.\nonumber\\
 {}
\end{eqnarray}

Let $\varepsilon < 1$. Using H\"{o}lder inequality, the inequality \eqref{eqn54chap10} yields
\begin{equation}\label{uniquenessequation33}
 \varepsilon r\displaystyle\sum_{j=1}^{d}\|\eta^{m}_{x_{j}}\|^{2}_{L^{2}(\Omega_{T})} \leq \|h\|_{L^{\infty}(\Omega_{T})}\|\eta^{m}\|_{L^{\infty}(\Omega_{T})}\mbox{Vol}(\Omega_{T}) +\displaystyle\sum_{j=1}^{d}\gamma_{j}\hspace{0.1cm}\|\eta^{m}_{x_{j}}\|_{L^{2}(\Omega_{T})},
\end{equation}
where 
$$ \gamma_{j}:= \|\eta^{m}\|_{L^{\infty}(\Omega_{T})}\|b_{j}^{m}\|_{L^{2}(\Omega_{T})} + \|\eta^{m}\|_{L^{\infty}(\Omega_{T})}\|c_{j}^{m}\|_{L^{2}(\Omega_{T})}.$$
Since $b_{j}^{m} \to \hat{b}_{j}$ and  $c_{j}^{m} \to \hat{c}_{j}$ in $L^{2}(\Omega_{T})$, 
there exist $R_{j} > 0$ such that for $m\in\mathbb{N}$
\begin{equation}\label{uniquenessequation18}
\|b_{j}^{m}\|_{L^{2}(\Omega_{T})}\leq R_{j},\,\, \|c_{j}^{m}\|_{L^{2}(\Omega_{T})}\leq R_{j}
\end{equation}
 
Using \eqref{uniquenessequation3} and inequalities \eqref{uniquenessequation18} in \eqref{uniquenessequation33}, we obtain
\begin{equation}\label{uniquenessequation21}
 \varepsilon r\displaystyle\sum_{j=1}^{d}\|\eta^{m}_{x_{j}}\|^{2}_{L^{2}(\Omega_{T})} \leq \|h\|^2_{L^{\infty}(\Omega_{T})}\mbox{Vol}(\Omega_{T})+2e^{T}\|h\|_{L^{\infty}(\Omega_{T})}\displaystyle\sum_{j=1}^{d}R_{j}\hspace{0.1cm}\|\eta^{m}_{x_{j}}\|_{L^{2}(\Omega_{T})}.
\end{equation}
From \eqref{uniquenessequation21}, we get for any $\alpha>0$
\begin{eqnarray}\label{uniquenessequation23}
 \varepsilon r\displaystyle\sum_{j=1}^{d}\|\eta^{m}_{x_{j}}\|^{2}_{L^{2}(\Omega_{T})} &\leq& \|h\|^2_{L^{\infty}(\Omega_{T})}\mbox{Vol}(\Omega_{T})
 +\displaystyle\sum_{j=1}^{d}\alpha \|\eta^{m}_{x_{j}}\|^{2}_{L^{2}(\Omega_{T})}
 +\frac{e^{2T}\|h\|^2_{L^{\infty}(\Omega_{T})}}{\alpha}\displaystyle\sum_{j=1}^{d}R^{2}_{j}.\nonumber\\
 {}
\end{eqnarray}
By choosing an $\alpha$ satisfying $0< \alpha < \frac{\varepsilon r}{2}$, the inequality \eqref{uniquenessequation23} yields
\begin{equation}\label{uniquenessequation26}
 \frac{\varepsilon r}{2}\displaystyle\sum_{j=1}^{d}\|\eta^{m}_{x_{j}}\|^{2}_{L^{2}(\Omega_{T})}\leq e^{T} \|h\|^{2}_{L^{\infty}(\Omega_{T})}\mbox{Vol}(\Omega_{T}) + \frac{e^{2T}\|h\|^2_{L^{\infty}(\Omega_{T})}}{\alpha}\displaystyle\sum_{j=1}^{d}R^{2}_{j}.
\end{equation}
This proves \eqref{Uniquenessequation4}, and completes the proof of Theorem~\ref{unique.maximum}.

\noindent\textbf{Proof of Theorem \ref{chap10thm1}:}\\
Recall that $w := u-v$, where $u$ and $v$ are solutions to \eqref{ibvp.parab} in $W(0,T)$. We want to show that $w=0$ $\it{a.e.}$ $(x,t)\in\Omega_{T}$. For that, it is enough to show that for all $ h\in C^{l,\frac{l}{2}}(\overline{\Omega_{T}})$ with $h(.,T)=0$ the folllowing equality holds:
\begin{equation}\label{paper.uniqueness.eqn1}
 \int_{\Omega_{T}}w\,h\,dx\,dt =0.
\end{equation}
For a given $ h\in C^{l,\frac{l}{2}}(\overline{\Omega_{T}})$ with $h(.,T)=0$, let $\eta^m$ denote solution to the auxiliary IBVP \eqref{eqn33chap10} corresponding to the operator $\mathcal{L}_{m}$. Note that
\begin{eqnarray}\label{paper.uniqueness.eqn2}
  \int_{\Omega_{T}}w\,h\,dx\,dt &=& \int_{\Omega_{T}}w \mathcal{L}_{m}(\eta^{m})\hspace{0.1cm}dx\hspace{0.1cm}dt\nonumber\\
  &=&-\int_{0}^{T}{_{_{H^{-1}(\Omega)}}}\!\langle \eta^{m}_{t}, w\rangle\!{_{_{H^{1}_{0}(\Omega)}}}\hspace{0.1cm}dt + \varepsilon \displaystyle\sum_{j=1}^{d}\int_{\Omega_{T}}a^{m}(x,t)\eta^{m}_{x_{j}}w_{x_{j}}\hspace{0.1cm}dx \hspace{0.1cm}dt \nonumber\\
&&+ \varepsilon \displaystyle\sum_{j=1}^{d}\int_{\Omega_{T}}b_{j}^{m}(x,t)\eta^{m}_{x_{j}}w\hspace{0.1cm}dx \hspace{0.1cm}dt -  \displaystyle\sum_{j=1}^{d}\int_{\Omega_{T}}c_{j}^{m}(x,t)\eta^{m}_{x_{j}}w\hspace{0.1cm}dx \hspace{0.1cm}dt.\nonumber\\
{}
 \end{eqnarray}
Applying integration by parts formula  in $W(0,T)$ \cite[p.427]{sal_08a} and using \eqref{eqn11chap10} in \eqref{paper.uniqueness.eqn2}, we get 
\begin{eqnarray}\label{eqn58chap10AB}
\int_{\Omega_{T}}w\,h\hspace{0.1cm}dx\hspace{0.1cm}dt &=& \varepsilon \displaystyle\sum_{j=1}^{d}\int_{\Omega_{T}}\Big[a^{m}(x,t)-\hat{a}(x,t)\Big]\eta^{m}_{x_{j}}w_{x_{j}}\hspace{0.1cm}dx \hspace{0.1cm}dt\nonumber\\
&&+ \varepsilon \displaystyle\sum_{j=1}^{d}\int_{\Omega_{T}}\Big[b_{j}^{m}(x,t)-\hat{b}_{j}(x,t)\Big]\eta^{m}_{x_{j}}w\hspace{0.1cm}dx \hspace{0.1cm}dt\nonumber\\
&&-\,\,\displaystyle\sum_{j=1}^{d}\int_{\Omega_{T}}\Big[c_{j}^{m}(x,t)-\hat{c}_{j}(x,t)\Big]\eta^{m}_{x_{j}}w\hspace{0.1cm}dx \hspace{0.1cm}dt.
\end{eqnarray}
Since $a^{m}\to\hat{a}$ in $L^{2}(\Omega_{T})$, by passing to a subsequence, we may also assume that $a^{m}\to\hat{a}$ {\it a.e.} $\Omega_{T}$. We will continue to index the subsequence by $m$ itself.  By a similar argument, we obtain pointwise convergences for the other sequences $b_{j}^{m}, c_{j}^{m}$. Thus in each of the three integral terms on RHS of \eqref{eqn58chap10AB}, we have a sequence that goes to zero pointwise. It follows easily from the bounds given by Theorem~\ref{unique.maximum}  that these integrands also satisfy hypothesis of dominated convergence theorem. Thus we conclude that the RHS of \eqref{eqn58chap10AB} converges to zero as $m\to\infty$, and thereby completing the proof of the theorem. 

\subsection{Continuous dependence}

Note that we have shown the  IBVP \eqref{ibvp.parab} admits a unique solution in the space $W(0,T)$. This solution also turns out to be a classical 
solution {\it i.e.,} it belongs to the space $C^{2+\beta,\frac{2+\beta}{2}}(\overline{\Omega_T})$. The following theorem follows by applying a result from \cite[p.452]{lad-etal_68a}.

\begin{theorem}\label{regularity.theorem1}
Let $f$, $B$, $u_0$ satisfy Hypothesis A. Then there exists a unique solution $u^\varepsilon$ of \eqref{ibvp.parab} in the space
$C^{2+\beta,\frac{2+\beta}{2}}(\overline{\Omega_T})$. Further, for each $i=1,2,\cdots,d$ the second order partial derivatives $u^\varepsilon_{x_i t}$ belong to $L^2(\Omega_T)$.  
\end{theorem}
 
Let $u_{0}, v_{0}$ satisfy Hypothesis A. Whenever $u,v$  are solutions to the IBVP \eqref{ibvp.parab} belonging to $C^{2+\beta,\frac{2+\beta}{2}}(\overline{\Omega_T})$ satisfying the 
 initial conditions $u_{0}, v_{0}$   respectively, the function $w:= u-v$ satisfies a linear parabolic equation with variable coefficients that depend on 
$u$ and $v$. Applying the maximum principle for linear parabolic equations  \cite[p.13]{lad-etal_68a} yields the following continuous dependence result.
 
\begin{theorem}[Continuous Dependence]\label{Continuous.Dependence}
 Let $f, B, u_0,v_0$   satisfy Hypothesis A. Let $u,v$  be classical solutions to the IBVP \eqref{ibvp.parab} belonging to $C^{2+\beta,\frac{2+\beta}{2}}(\overline{\Omega_T})$ satisfying the 
 initial conditions $u_{0}, v_{0}$ respectively. Then there exists a constant $C$ depending only on $T$ such that
 \begin{equation}\label{continuous.dependence.eqn1}
  \|u-v\|_{L^{\infty}(\Omega_{T})}\leq C \|u_{0}-v_{0}\|_{L^{\infty}(\Omega)}.\nonumber
 \end{equation}
\end{theorem}
Thus we conclude the well-posedness of the IBVP \eqref{ibvp.parab}.

 \section{Higher regularity result for viscous approximations}\label{section.higherregularity}
In this section we establish the following result which shows that the classical solutions to the IBVP  \eqref{ibvp.parab} possess higher regularity when the initial conditions are more regular.
\begin{theorem}[higher regularity]\label{chapHR85thm5}
Let $f,B,u_0$ satisfy Hypothesis A. 
Then the solutions of the IBVP \eqref{ibvp.parab} belong to the space $C^{4+\beta,\frac{4 + \beta}{2}}(\overline{\Omega_{T}})$. Further,  $u_{tt}^{\varepsilon}\in C(\overline{\Omega_{T}})$.
\end{theorem}

Due to the presence of a  quasilinear higher order term in the equation \eqref{ibvp.parab.a}, the standard methods of inductively proving higher 
regularity results by differentiating the equation as many times as needed fail. This is becuase the form of the equation obtained after 
differentiating the equation \eqref{ibvp.parab.a} each differentiation is different. Since we know that $u^\varepsilon$ are classical solutions 
to \eqref{ibvp.parab}, we apply  the higher regularity result Theorem~\ref{thm2chap10} for linear parabolic equations repeatedly, after deriving a linear parabolic equation satisfied by the solution to \eqref{ibvp.parab}. 

\noindent{\bf Proof of Theorem~\ref{chapHR85thm5}:}

Since $u^{\varepsilon}$ is a classical solution to the IBVP \eqref{ibvp.parab}, we have
\begin{subequations}\label{Regularlineareqn1}
\begin{eqnarray}
 u^{\varepsilon}_{t}-\displaystyle\sum_{j=1}^{d}\varepsilon\,B(u^{\varepsilon})\,\frac{\partial^{2} u^{\varepsilon}}{\partial x_{j}^{2}}+ \displaystyle\sum_{i=1}^{d}\left(f_{i}^{\prime}(u^{\varepsilon})-\varepsilon\, B^{'}(u^{\varepsilon})\frac{\partial u^{\varepsilon}}{\partial x_{i}}\right)\,\frac{\partial u^{\varepsilon}}{\partial x_{i}}=0&\mbox{in }\Omega_{T},\label{Regularlineareqn1.a} \\
    u^{\varepsilon}(x,t)= 0&\,\,\,\,\mbox{on}\,\, \partial \Omega\times(0,T),\hspace{1cm}\\
u^{\varepsilon}(x,0) = u_0(x)& x\in \Omega.
\end{eqnarray}
\end{subequations}
Thus we may view $u^\varepsilon$ as a solution of the following IBVP for a linear parabolic equation
\begin{subequations}\label{Regularlineareqn2}
\begin{eqnarray}
v^{\varepsilon}_{t}-\displaystyle\sum_{j=1}^{d}\varepsilon\,B(u^{\varepsilon})\,\frac{\partial^{2} v^{\varepsilon}}{\partial x_{j}^{2}}+ \displaystyle\sum_{i=1}^{d}\left(f_{i}^{\prime}(u^{\varepsilon})-\varepsilon\, B^{'}(u^{\varepsilon})\frac{\partial u^{\varepsilon}}{\partial x_{i}}\right)\,\frac{\partial v^{\varepsilon}}{\partial x_{i}} =0&\mbox{in }\Omega_{T},\label{Regularlineareqn2.a} \\
    v^{\varepsilon}(x,t)= 0&\,\,\,\,\mbox{on}\,\, \partial \Omega\times(0,T),\hspace{1cm}\\
v^{\varepsilon}(x,0) = u_0(x)& x\in \Omega.
\end{eqnarray}
\end{subequations}
Thus we are in a position to utilize Theorem~\ref{thm2chap10} after duly checking the relevant hypotheses on the coefficient functions $$ B(u^{\varepsilon}),\,\, f_{i}^{\prime}(u^{\varepsilon}),\,\,B^{\prime}(u^{\varepsilon})\,u^{\varepsilon}_{x_{i}}\,\,(i=1,2,\cdots, d),$$
and initial-boundary data of the problem \eqref{Regularlineareqn2}.  Hypotheses on coefficient functions are checked by first writing out the relevant H\"{o}lder difference quotient, and estimating the same using mean value theorem, and bounding it from above by using 
the boundedness of the domain $\Omega$, and of the partial derivatives of $u^\varepsilon$, regularity of $B,f$.  Hence we give only a sample of these proofs. Strictly speaking we cannot apply mean value theorem since the domain $\Omega$ may not be convex. However it turns out that  H\"{o}lder continuous functions can always be extended to bigger convex domains with the same H\"{o}lder regularity, and preserving H\"{o}lder norms, see for example \cite[p.297]{lad-etal_68a}. Thus we may extend $u^{\varepsilon}\in C^{2+\beta,\frac{2+\beta}{2}}(\overline{\Omega_{T}})$ (by Theorem~\ref{regularity.theorem1}) to a bigger convex domain, and we still use $u^{\varepsilon}$ to denote the extended function. Thus we are in a position to apply mean value theorem using the extended functions.

{\bf Step 1:}  We show that the coefficient functions have the property
$$ B(u^{\varepsilon}),\,\, f_{i}^{\prime}(u^{\varepsilon}),\,\,B^{\prime}(u^{\varepsilon})\,u^{\varepsilon}_{x_{i}}\in C^{1+\beta,\frac{1+\beta}{2}}(\overline{\Omega_{T}}).$$

\noindent Note that for $k\in\left\{1,2,\cdots,d\right\}$, 
\begin{eqnarray}\label{Regularitylinear13}
 \left[\frac{\partial B(u^{\varepsilon})}{\partial x_{k}}\right]_{x,\beta,\overline{\Omega_{T}}}&=&\displaystyle\sup_{(x,t),(x^{'},t)\in\overline{\Omega_{T}},x\neq x^{'}} \frac{\Big|B^{\prime}(u^{\varepsilon}(x,t))u^{\varepsilon}_{x_{k}}(x,t)-B^{\prime}(u^{\varepsilon}(x^{'},t))u^{\varepsilon}_{x_{k}}(x^{'},t)\Big|}{|x-x^{'}|^{\beta}},\nonumber\\
 &\leq& \displaystyle\sup_{(x,t),(x^{'},t)\in\overline{\Omega_{T}},x\neq x^{'}} \frac{\Big|\left(B^{\prime}(u^{\varepsilon}(x,t))u^{\varepsilon}_{x_{k}}(x,t)-B^{\prime}(u^{\varepsilon}(x^{'},t))u^{\varepsilon}_{x_{k}}(x,t)\right)\Big|}{|x-x^{'}|^{\beta}}\nonumber\\
 &&+\, \displaystyle\sup_{(x,t),(x^{'},t)\in\overline{\Omega_{T}},x\neq x^{'}} \frac{\Big|\left(B^{\prime}(u^{\varepsilon}(x^{'},t))u^{\varepsilon}_{x_{k}}(x,t)-B^{\prime}(u^{\varepsilon}(x^{'},t))u^{\varepsilon}_{x_{k}}(x^{'},t) \right)\Big|}{|x-x^{'}|^{\beta}}.\nonumber\\
\end{eqnarray}
Applying mean value theorem, the inequality \eqref{Regularitylinear13} yields
\begin{eqnarray}\label{Regularitylinear15}
 \left[\frac{\partial B(u^{\varepsilon})}{\partial x_{k}}\right]_{x,\beta,\overline{\Omega_{T}}} &\leq&  \,\displaystyle\sup_{(x,t)\in\overline{\Omega_{T}}}\Big|\frac{\partial u^{\varepsilon}}{\partial x_{k}}\Big|\,\displaystyle\sup_{(x,t),(x^{'},t)\in\overline{\Omega_{T}},x\neq x^{'}}\Big|\nabla_{x}\left(B^{'}(u^{\varepsilon})\right)(\xi_{(x,t),(x^{'},t)})\Big||x-x^{'}|^{1-\beta}\nonumber\\
 &&+ \,\|B^{'}\|_{L^{\infty}(I)}\,\left[\frac{\partial u^{\varepsilon}}{\partial x_{k}}\right]_{x,\beta;\overline{\Omega_{T}}}.
\end{eqnarray}
Applying similar arguments, we get
\begin{eqnarray}\label{Regularitylinear18}
 \left[\frac{\partial B(u^{\varepsilon})}{\partial x_{k}}\right]_{t,\frac{\beta}{2},\overline{\Omega_{T}}}&\leq& \,\displaystyle\sup_{(x,t)\in\overline{\Omega_{T}}}\Big|\frac{\partial u^{\varepsilon}}{\partial x_{k}}\Big|\,\displaystyle\sup_{(x,t),(x,t^{'})\in\overline{\Omega_{T}},t\neq t^{'}}\Big|\frac{\partial}{\partial t}\left(B^{'}(u^{\varepsilon})\right)(\xi_{(x,t),(x,t^{'})})\Big||t-t^{'}|^{1-\beta}\nonumber\\
 && +  \,\|B^{'}\|_{L^{\infty}(I)}\,\left[\frac{\partial u^{\varepsilon}}{\partial x_{k}}\right]_{t,\frac{\beta}{2};\overline{\Omega_{T}}}.
\end{eqnarray}
Note that the quantities
\begin{eqnarray}\label{Regularitylinear25}
  \left[\frac{\partial u^{\varepsilon}}{\partial x_{k}}\right]_{x,\beta;\overline{\Omega_{T}}},\,\,\left[\frac{\partial u^{\varepsilon}}{\partial x_{k}}\right]_{t,\frac{\beta}{2};\overline{\Omega_{T}}}  
 \end{eqnarray}
are finite due to the fact that $u^{\varepsilon}\in C^{2+\beta,\frac{2+\beta}{2}}(\overline{\Omega_{T}})$, and also since it was extended with the same regularity to a bigger convex domain. In view of this observation, and the fact that $\mbox{diam}(\Omega)<\infty$, $B\in C^{3}(\mathbb{R})$, therefore RHS of \eqref{Regularitylinear15} and  \eqref{Regularitylinear18} are finite.

Let us now show that $f_{i}^{\prime}(u^{\varepsilon})\in C^{1+\beta,\frac{1+\beta}{2}}(\overline{\Omega_{T}}).$
Note that
\begin{eqnarray}\label{Regularitylinear7a}
 \left[\frac{\partial f_{i}^{\prime}(u^{\varepsilon})}{\partial x_{k}}\right]_{x,\beta,\overline{\Omega_{T}}} &=&\displaystyle\sup_{(x,t),(x,t^{'})\in\overline{\Omega_{T}},x\neq x^{'}} \frac{\Big|f_i^{\prime\prime}(u^{\varepsilon}(x,t))u^{\varepsilon}_{x_{k}}(x,t)-f_i^{\prime\prime}(u^{\varepsilon}(x^{'},t))u^{\varepsilon}_{x_{k}}(x^{'},t)\Big|}{|x-x'|^{\beta}}\nonumber\\
 &\leq&\displaystyle\sup_{(x,t),(x,t^{'})\in\overline{\Omega_{T}},x\neq x^{'}}\frac{\Big|\left(f_i^{\prime\prime}(u^{\varepsilon}(x,t))-f_i^{\prime\prime}(u^{\varepsilon}(x^{'},t))\right)u^{\varepsilon}_{x_{k}}(x,t)\Big|}{|x-x'|^{\beta}}\nonumber\\
 &&+ \displaystyle\sup_{(x,t),(x,t^{'})\in\overline{\Omega_{T}},x\neq x^{'}}\frac{\Big| f_i^{''}(u^{\varepsilon}(x^{'},t))\left(u^{\varepsilon}_{x_{k}}(x,t)-u^{\varepsilon}_{x_{k}}(x^{'},t)\right)\Big|}{|x-x'|^{\beta}}.
\end{eqnarray}
Using mean value theorem in \eqref{Regularitylinear7a}, we obtain
\begin{eqnarray}\label{Regularitylinear9a}
\left[\frac{\partial f_{i}^{\prime}(u^{\varepsilon})}{\partial x_{k}}\right]_{x,\beta,\overline{\Omega_{T}}} &\leq&\displaystyle\sup_{(x,t),(x,t^{'})\in\overline{\Omega_{T}},x\neq x^{'}} \Big|\nabla_{x}\left(f_i^{\prime\prime}(u^{\varepsilon})\right)(\xi_{(x,t),(x^{'},t)})\Big|\,\Big|u_{x_{k}}^{\varepsilon}\Big|\,|x-x^{'}|^{1-\beta}\nonumber\\
 && + \|f_i^{\prime\prime}\|_{L^{\infty}(I)}\displaystyle\sup_{(x,t),(x,t^{'})\in\overline{\Omega_{T}},x\neq x^{'}}\Big|\nabla_{x}\left(u^{\varepsilon}_{x_{k}}\right)(\eta_{(x,t),(x^{'},t})\Big|.|x-x^{'}|^{1-\beta}.\nonumber\\
{} 
\end{eqnarray}
 Since $u^{\varepsilon}\in C^{2,1}(\overline{\Omega_{T}}),\,\, f\in C^{3}(\mathbb{R})$ and $\mbox{diam}(\Omega)<\infty$, therefore the RHS of \eqref{Regularitylinear9a} is finite.

We will now show that $\left[\frac{\partial f_{i}^{\prime}(u^{\varepsilon})}{\partial x_{k}}\right]_{t,\frac{\beta}{2};\overline{\Omega_{T}}}<\infty$. Note that
\begin{eqnarray*}\label{Regularitylinear11a}
 \left[\frac{\partial f_{i}^{\prime}(u^{\varepsilon})}{\partial x_{k}}\right]_{t,\frac{\beta}{2};\overline{\Omega_{T}}} &=& \displaystyle\sup_{(x,t),(x,t^{'})\in\overline{\Omega_{T}},t\neq t^{'}} \frac{\Big|f_i^{\prime\prime}(u^{\varepsilon}(x,t))u^{\varepsilon}_{x_{k}}(x,t)-f_i^{\prime\prime}(u^{\varepsilon}(x,t^\prime))u^{\varepsilon}_{x_{k}}(x,t^{'})\Big|}{|t-t'|^{\frac{\beta}{2}}}. \nonumber \\
  &\leq&  \displaystyle\sup_{(x,t),(x,t^{'})\in\overline{\Omega_{T}},t\neq t^{'}}\frac{\Big|\left(f_i^{''}(u^{\varepsilon}(x,t))-f_i^{''}(u^{\varepsilon}(x,t^{'}))\right)\Big|}{|t-t^{'}|^{\frac{\beta}{2}}}\,|u^{\varepsilon}_{x_{k}}|\nonumber\\
 &&+ \|f_i^{''}\|_{L^{\infty}(I)}  \displaystyle\sup_{(x,t),(x,t^{'})\in\overline{\Omega_{T}},t\neq t^{'}}\frac{\Big|\left(u^{\varepsilon}_{x_{k}}(x,t)-u^{\varepsilon}_{x_{k}}(x,t^{'})\right)\Big|}{|t-t^{'}|^{\frac{\beta}{2}}}
\end{eqnarray*}

Applying mean value theorem, we get
\begin{eqnarray}\label{Regularitylinear12Ba}
 \left[\frac{\partial f_{i}^{\prime}(u^{\varepsilon})}{\partial x_{k}}\right]_{t,\frac{\beta}{2};\overline{\Omega_{T}}} &\leq& T^{1-\frac{\beta}{2}}\, \displaystyle\sup_{(x,t),(x,t^{'})\in\overline{\Omega_{T}},t\neq t^{'}}\Big|\frac{\partial}{\partial t}\left(f_i^{''}(u^{\varepsilon})\right)(\xi_{(x,t),(x,t^{'})})\Big|\,\displaystyle\sup_{(x,t)\in\overline{\Omega_{T}}}|\nabla u^{\varepsilon}|\nonumber\\
 &&\,\,+\, \|f_i^{''}\|_{L^{\infty}(I)}\,[u^{\varepsilon}_{x_{k}}]_{t,\frac{\beta}{2};\overline{\Omega_{T}}}.
\end{eqnarray}
Since $u^{\varepsilon}\in C^{2,1}(\overline\Omega_{T})$, the RHS of \eqref{Regularitylinear12Ba} is finite. 

Checking of $B^{\prime}(u^{\varepsilon})\,u^{\varepsilon}_{x_{i}}\in C^{1+\beta,\frac{1+\beta}{2}}(\overline{\Omega_{T}})$ follows on similar lines, and we shall omit the computations.

Since $u_0$ has compact support in $\Omega$, it satisfes the compatibility conditions \eqref{compatibility.condns}, an application of Theorem \ref{thm2chap10} asserts the existence of a unique solution $v^{\varepsilon}$ to the IBVP \eqref{Regularlineareqn2} in $C^{3+\beta,\frac{3+\beta}{2}}(\overline{\Omega_{T}})$. Since 
$C^{3+\beta,\frac{3+\beta}{2}}(\overline{\Omega_{T}})$ is a subset of $C^{2+\beta,\frac{2+\beta}{2}}(\overline{\Omega_{T}})$, it follows that $u^{\varepsilon}=v^{\varepsilon}$. Thus $u^{\varepsilon}\in C^{3+\beta,\frac{3+\beta}{2}}(\overline{\Omega_{T}})$.

{\bf Step 2:} By following the procedure outlined and implemented in Step 1, it can be shown that the coefficient functions belong to the space $C^{2+\beta,\frac{2+\beta}{2}}(\overline{\Omega_{T}})$, and by using the information which is coming from Steps 1, namely, $u^{\varepsilon}\in C^{3+\beta,\frac{3+\beta}{2}}(\overline{\Omega_{T}})$. Since $u_0$ has compact support in $\Omega$, it satisfies the compatibility conditions \eqref{compatibility.condns}, an application of Theorem \ref{thm2chap10} asserts the existence of a unique solution $z^{\varepsilon}$ to the IBVP \eqref{Regularlineareqn2} in $C^{4+\beta,\frac{4+\beta}{2}}(\overline{\Omega_{T}})$. Since 
$C^{4+\beta,\frac{4+\beta}{2}}(\overline{\Omega_{T}})$ is a subset of $C^{2+\beta,\frac{2+\beta}{2}}(\overline{\Omega_{T}})$, it follows that $u^{\varepsilon}=z^{\varepsilon}$. Thus $u^{\varepsilon}\in C^{4+\beta,\frac{4+\beta}{2}}(\overline{\Omega_{T}})$.

In order to show that $u^{\varepsilon}_{tt}\in C(\overline{\Omega_{T}})$, we write the generalized viscosity problem \eqref{ibvp.parab} in nondivergence form 
\begin{eqnarray}\label{Regularitylinear26}
u^{\varepsilon}_{t}=-\displaystyle\sum_{j=1}^{d}f^{\prime}_{j}(u^{\varepsilon})\frac{\partial u^{\varepsilon}}{\partial x_{j}} + \varepsilon\displaystyle\sum_{j=1}^{d}\left(B^{\prime}(u^{\varepsilon})\left(\frac{\partial u^{\varepsilon}}{\partial x_{j}}\right)^{2} + B(u^{\varepsilon})\frac{\partial^{2}u^{\varepsilon}}{\partial x_{j}^{2}}\right).
\end{eqnarray}
Differentiating the RHS of \eqref{Regularitylinear26} with respect to $t$, we get  
\begin{eqnarray}\label{Regularitylinear27AB}
 u^{\varepsilon}_{tt}=\varepsilon\displaystyle\sum_{j=1}^{d}\Big[B^{'}(u^{\varepsilon})\,\frac{\partial u^{\varepsilon}}{\partial t}\,\frac{\partial^{2}u^{\varepsilon}}{\partial x_{j}^{2}} + B(u^{\varepsilon})\frac{\partial^{3}u^{\varepsilon}}{\partial t\partial x_{j}^{2}}\Big]- \displaystyle\sum_{i=1}^{d}\left(f_{i}^{'}(u^{\varepsilon})-\varepsilon B^{'}(u^{\varepsilon})\frac{\partial u^{\varepsilon}}{\partial x_{i}}\right)\frac{\partial^{2}u^{\varepsilon}}{\partial t\partial x_{i}}\nonumber\\ - \displaystyle\sum_{i=1}^{d}\Big[f_{i}^{''}(u^{\varepsilon})\,\frac{\partial u^{\varepsilon}}{\partial t} -\varepsilon B^{''}(u^{\varepsilon})\,\frac{\partial u^{\varepsilon}}{\partial t}\,\frac{\partial u^{\varepsilon}}{\partial x_{i}}-\varepsilon B^{'}(u^{\varepsilon})\,\frac{\partial^{2}u^{\varepsilon}}{\partial t\partial x_{i}}\Big]\,\frac{\partial u^{\varepsilon}}{\partial x_{i}}.\hspace*{0.5in}
\end{eqnarray}
Since $u^{\varepsilon}\in C^{4+\beta,\frac{4+\beta}{2}}(\overline{\Omega_{T}})$, \eqref{Regularitylinear27AB} belongs to $C(\overline{\Omega_{T}})$. Thus we have $u_{tt}^{\varepsilon}\in C(\overline{\Omega_{T}})$. This completes the proof of Theorem \ref{chapHR85thm5}.

\section{BV estimates}\label{section.bvestimates}
In this section we prove the following result concerning BV estimates on the sequence of viscous approximations $(u^\varepsilon)$. For definition of BV functions, we refer the reader to \cite{MR1304494}.
\begin{theorem}[BV estimate]\label{BVEstimate.thm1}
 Let $f,\,\,B,\,\,u_{0}$ satisfy Hypothesis A. Let $(u^{\varepsilon})$ be the sequence of solutions to the IBVP \eqref{ibvp.parab}. Then there exists a $C > 0$ such that for all $\varepsilon > 0$, the following estimate holds:
 \begin{equation}\label{BVestimateA.eqn1}
  \left\|\frac{\partial u^{\varepsilon}}{\partial t}\right\|_{L^{1}(\Omega_{T})} +  \|\nabla u^{\varepsilon}\|_{\left(L^{1}(\Omega_{T})\right)^{d}}\leq C.
 \end{equation}
 Further there exists a subsequence $(u^{\varepsilon_{k}})$ of $(u^{\varepsilon})$, and a function $u\in L^{1}(\Omega_{T})$ such that 
\begin{eqnarray}\label{BVestimateeqn25}
u^{\varepsilon_{k}}\to u\,\,\mbox{in}\,L^1(\Omega_T),\\
 u^{\varepsilon_{k}}\to u\,\,{\it a.e.}\,(x,t)\in \Omega_{T}
\end{eqnarray}
as $k\to\infty$.
\end{theorem}
Bardos {\it et. al.} \cite{MR542510} established uniform $L^1$-estimates on first order derivatives of $u^\varepsilon$, when $B(u)\equiv 1$, using suitable multipliers. When $B$ is a non-constant function, we are unable to  estimate the first order spatial derivatives of $u^\varepsilon$ using their multipliers, and we use a different multiplier for this purpose.

Let $sg_{n}$ ($n\in\mathbb{N}$) be the sequence of functions which converges pointwise to the signum function $sg$ which are defined for $s\in \mathbb{R}$ by
$$
sg_{n}(s)=
       \begin{cases}
       1 & \,\mbox{if}\, s >\frac{1}{n},\\
       ns & \,\mbox{if}\, |s|\leq\frac{1}{n},\\
       -1 & \,\mbox{if}\, s < -\frac{1}{n},
       \end{cases},\,\,sg(s)=
       \begin{cases}
       1 & \,\mbox{if}\, s > 0,\\
       0 & \,\mbox{if}\, s=0,\\
       -1 & \,\mbox{if}\, s < 0.
       \end{cases}
$$
The following result will be used in proving the required BV estimate.
\begin{lemma}\cite[p.1020]{MR542510}\label{compactness12.lem1}
 Let $v\in C^{1}(\overline{\Omega})$. Then 
 $$\displaystyle\lim_{n\to\infty}\int_{\left\{x\in\Omega\,;\,|v(x)|<\frac{1}{n}\right\}}\,\left|\nabla v\right|\,dx=0.$$
\end{lemma}

\noindent{\bf Proof of Theorem~\ref{BVEstimate.thm1}:}
We prove Theorem \ref{BVEstimate.thm1} in three steps. The estimate \eqref{BVestimateA.eqn1} will be established in the first two steps, and in the third step we deduce \eqref{BVestimateeqn25}.

\noindent\textbf{Step 1:} We show the existence of constant $C_{1}> 0$ such that for every $\varepsilon>0$,
\begin{eqnarray}\label{B.BVestimate26}
 \left\|\frac{\partial u^{\varepsilon}}{\partial t}\right\|_{L^{1}(\Omega_{T})}\leq C_{1}.
\end{eqnarray}
Differentiating the equation  \eqref{ibvp.parab.a} with respect to $t$, multiplying by $sg_{n}\left(\frac{\partial u^{\varepsilon}}{\partial t}\right)$ and integrating over $\Omega$, we get
\begin{eqnarray}
 \int_{\Omega} u^\varepsilon_{tt}\,sg_{n}(u^\varepsilon_{t})\,dx + \displaystyle\sum_{j=1}^{d}\int_{\Omega}\Big[\frac{\partial}{\partial x_{j}}\left(f_{j}^{\prime}(u^{\varepsilon})\,\frac{\partial u^{\varepsilon}}{\partial t}\right)\Big]sg_{n}\left(\frac{\partial u^{\varepsilon}}{\partial t}\right)\,dx \hspace*{1in}\nonumber\\=\varepsilon\displaystyle\sum_{j=1}^{d}\int_{\Omega}\,sg_{n}\left(\frac{\partial u^{\varepsilon}}{\partial t}\right)\,\frac{\partial}{\partial x_{j}}\Big(B^{'}(u^{\varepsilon})\,\frac{\partial u^{\varepsilon}}{\partial t}\,\frac{\partial u^{\varepsilon}}{\partial x_{j}} + B(u^{\varepsilon})\,\frac{\partial}{\partial x_{j}}\left(\frac{\partial u^{\varepsilon}}{\partial t}\right)\Big)\,dx.\label{compactness12.eqn3}
 \end{eqnarray}
Using integration by parts in \eqref{compactness12.eqn3} and using $sg_{n}\left(\frac{\partial u^{\varepsilon}}{\partial t}\right)=0\,\,\mbox{on}\,\,\partial\Omega\times(0,T)$, we have
\begin{eqnarray}
  \int_{\Omega} u^\varepsilon_{tt}\,sg_{n}(u^\varepsilon_{t})\,dx = \displaystyle\sum_{j=1}^{d}\int_{\Omega}f_{j}^{\prime}(u^{\varepsilon})\,\frac{\partial u^{\varepsilon}}{\partial t}\,sg^{'}_{n}\left(\frac{\partial u^{\varepsilon}}{\partial t}\right)\,\frac{\partial}{\partial x_{j}}\left(\frac{\partial u^{\varepsilon}}{\partial t}\right)\,dx \hspace*{1in}\nonumber\\
-\varepsilon\displaystyle\sum_{j=1}^{d}\int_{\Omega}B^{'}(u^{\varepsilon})\,\frac{\partial u^{\varepsilon}}{\partial t}\,\frac{\partial u^{\varepsilon}}{\partial x_{j}}\,sg^{\prime}_{n}\left(\frac{\partial u^{\varepsilon}}{\partial t}\right)\,\frac{\partial}{\partial x_{j}}\left(\frac{\partial u^{\varepsilon}}{\partial t}\right)\,dx\nonumber\\
  -\varepsilon\displaystyle\sum_{j=1}^{d}\int_{\Omega}B(u^{\varepsilon})\,\left(\frac{\partial}{\partial x_{j}}\left(\frac{\partial u^{\varepsilon}}{\partial t}\right)\right)^{2}\,sg^{\prime}_{n}\left(\frac{\partial u^{\varepsilon}}{\partial t}\right)\,dx.  \label{compactness12.eqn11}
\end{eqnarray}
We now prove that first two terms on the RHS of \eqref{compactness12.eqn11} tend to zero as $\varepsilon\to 0$. That is,
\begin{eqnarray}
\displaystyle\lim_{n\to\infty}\displaystyle\sum_{j=1}^{d}\int_{\Omega}f_{j}^{\prime}(u^{\varepsilon})\,\frac{\partial u^{\varepsilon}}{\partial t}\,sg^{'}_{n}\left(\frac{\partial u^{\varepsilon}}{\partial t}\right)\,\frac{\partial}{\partial x_{j}}\left(\frac{\partial u^{\varepsilon}}{\partial t}\right)\,dx =0,\label{compactness12.eqn6}\\
\displaystyle\lim_{n\to\infty}\varepsilon\displaystyle\sum_{j=1}^{d}\int_{\Omega}B^{'}(u^{\varepsilon})\,\frac{\partial u^{\varepsilon}}{\partial t}\,\frac{\partial u^{\varepsilon}}{\partial x_{j}}\,sg^{\prime}_{n}\left(\frac{\partial u^{\varepsilon}}{\partial t}\right)\,\frac{\partial}{\partial x_{j}}\left(\frac{\partial u^{\varepsilon}}{\partial t}\right)\,dx=0.\label{compactness12.eqn9}
\end{eqnarray}
\noindent{\bf Proof of \eqref{compactness12.eqn6}:}  
Since $\Big|\frac{\partial u^{\varepsilon}}{\partial t}\Big|\,sg^{'}_{n}\left(\frac{\partial u^{\varepsilon}}{\partial t}\right) < 1$, note that
\begin{eqnarray}\label{compactness12.eqn5B}
\left|\displaystyle\sum_{j=1}^{d}\int_{\left\{x\in\Omega\,:\,|\frac{\partial u^{\varepsilon}}{\partial t}|<\frac{1}{n}\right\}}\frac{\partial u^{\varepsilon}}{\partial t}sg^{'}_{n}\left(\frac{\partial u^{\varepsilon}}{\partial t}\right)\,\left(f_{1}^{\prime}(u^{\varepsilon}),\cdots,f_{d}^{\prime}(u^{\varepsilon})\right).\nabla\left(\frac{\partial u^{\varepsilon}}{\partial t}\right)\,dx\right|\hspace*{0.5cm}\nonumber\\
 \leq \sqrt{d}\,\displaystyle\max_{1\leq j\leq d}\left(\displaystyle\sup_{y\in I}\left|f^{'}_{j}(y)\right|\right)\,\int_{\left\{x\in\Omega\,:\,|\frac{\partial u^{\varepsilon}}{\partial t}|<\frac{1}{n}\right\}}\left|\nabla\left(\frac{\partial u^{\varepsilon}}{\partial t}\right) \right|\,dx.\hspace*{1cm}
\end{eqnarray}
 
Applying Lemma \ref{compactness12.lem1} with $v=\frac{\partial u^{\varepsilon}}{\partial t}$, the inequality \eqref{compactness12.eqn5B} gives \eqref{compactness12.eqn6}. 

\noindent{\bf Proof of \eqref{compactness12.eqn9}:}
Observe that
\begin{eqnarray}\label{compactness12.eqn8B}
 \left|\varepsilon\displaystyle\sum_{j=1}^{d}\int_{\Omega}B^{'}(u^{\varepsilon})\,\frac{\partial u^{\varepsilon}}{\partial t}\,\frac{\partial u^{\varepsilon}}{\partial x_{j}}\,sg^{\prime}_{n}\left(\frac{\partial u^{\varepsilon}}{\partial t}\right)\,\frac{\partial}{\partial x_{j}}\left(\frac{\partial u^{\varepsilon}}{\partial t}\right)\,dx \right|\hspace*{3cm}\nonumber\\ \leq \varepsilon\,\sqrt{d}\|B^{'}\|_{L^{\infty}(I)}\,\displaystyle\max_{1\leq j\leq d}\left(\left\|\frac{\partial u^{\varepsilon}}{\partial x_{j}}\right\|_{L^{\infty}(\Omega_{T})}\right) \int_{\left\{ x\in\Omega\,:\,\left|\frac{\partial u^{\varepsilon}}{\partial t}\right|<\frac{1}{n}\right\}}\left|\nabla\left(\frac{\partial u^{\varepsilon}}{\partial t}\right)\right|\,dx.
\end{eqnarray}

Applying Lemma \ref{compactness12.lem1} with $v=\frac{\partial u^{\varepsilon}}{\partial t}$, the inequality \eqref{compactness12.eqn8B} yields \eqref{compactness12.eqn9}. 

\noindent Since the third term on RHS of \eqref{compactness12.eqn13} is non-positive for every $\varepsilon>0$, on taking limit supremum on both sides of \eqref{compactness12.eqn11} yields
\begin{eqnarray}\label{compactness12.eqn13}
 \displaystyle\limsup_{n\to\infty}\,\int_{\Omega} u^\varepsilon_{tt}\,\,sg_{n}(u^\varepsilon_{t})\,dx &\leq& 0
\end{eqnarray}
in view of \eqref{compactness12.eqn6} and \eqref{compactness12.eqn9}. Note that the limit supremum in \eqref{compactness12.eqn13} is actually a limit, and as a consequence we get 
\begin{eqnarray}\label{compactness12.eqn14}
 \displaystyle\,\int_{\Omega} \frac{\partial}{\partial t}\left|u^\varepsilon_{t}\right|\,dx &\leq& 0
\end{eqnarray}
Integrating w.r.t. $t$ on both sides of \eqref{compactness12.eqn14}, and applying Fubini's theorem yields
\begin{eqnarray}\label{compactness12.eqn141}
 \int_{\Omega}\int_0^t \frac{\partial}{\partial t}\left|u^\varepsilon_{t}\right|\,d\tau\,dx &\leq& 0
\end{eqnarray}
Thus we get
\begin{eqnarray}\label{compactness12.eqn142}
 \int_{\Omega} \left(\left|u^\varepsilon_{t}(x,t)\right|-\left|u^\varepsilon_{t}(x,0)\right|\right)\,dx &\leq& 0
\end{eqnarray}
From the equation \eqref{ibvp.parab.a}, we get 
\begin{eqnarray}\label{compactness12.eqn16}
 \frac{\partial u^{\varepsilon}}{\partial t}(x,0) &=& \varepsilon\displaystyle\sum_{j=1}^{d}\Big[B(u_{0})\,\frac{\partial^{2}u^{\varepsilon}}{\partial x_{j}^{2}}(x,0) + B^{'}(u_{0})\left(\frac{\partial u^{\varepsilon}}{\partial x_{j}}\right)^{2}(x,0)\Big]-\displaystyle\sum_{j=1}^{d} f_{j}^{'}(u_{0})\frac{\partial u^{\varepsilon}}{\partial x_{j}}(x,0)\nonumber\\
 &=&\varepsilon\displaystyle\sum_{j=1}^{d}\Big[B(u_{0})\,\frac{\partial^{2}u_{0}}{\partial x_{j}^{2}} + B^{'}(u_{0})\left(\frac{\partial u_{0}}{\partial x_{j}}\right)^{2}\Big]-\displaystyle\sum_{j=1}^{d} f_{j}^{'}(u_{0})\frac{\partial u_{0}}{\partial x_{j}}.
\end{eqnarray}
Since $u_{0}\in C^{4+\beta}(\overline{\Omega})$, taking 
\begin{eqnarray}\label{compactness12.eqn18}
 \frac{C_{1}}{T\,d\,\mbox{Vol}(\Omega)}= \|B\|_{L^{\infty}(I)}\,\displaystyle\max_{1\leq j\leq d}\left(\displaystyle\sup_{x\in\overline{\Omega}}\Big|\frac{\partial^{2}u_{0}}{\partial x_{j}^{2}}\Big|\right) + \|B^{'}\|_{L^{\infty}(I)}\,\displaystyle\max_{1\leq j\leq d}\left(\displaystyle\sup_{x\in\overline{\Omega}}\Big|\frac{\partial u_{0}}{\partial x_{j}}\Big|^{2}\right)\nonumber \\+ \|f^{'}\|_{\left(L^\infty(I)\right)^d}\|\nabla u_0\|_{\left(L^\infty(\Omega)\right)^d}\,,
\end{eqnarray}
we get \eqref{B.BVestimate26}. 
\vspace{0.2cm}\\
\textbf{Step 2:} In Step 2, we show that
\begin{eqnarray}\label{B.BVestimate27}
 \left\|\nabla u^{\varepsilon}\right\|_{\left(L^{1}(\Omega_{T})\right)^{d}}= \|\nabla u_{0}\|_{\left(L^{1}(\Omega)\right)^{d}}.
\end{eqnarray}
Generalized viscosity problem \eqref{ibvp.parab} can be written in the following non-divergence form
\begin{subequations}\label{eqnchap21A}
 \begin{eqnarray}
   u^{\varepsilon}_{t} + \nabla\cdot f(u^{\varepsilon}) = \varepsilon\,\displaystyle\sum_{j=1}^{d}\left(B^{\prime}(u^{\varepsilon})\left(\frac{\partial u^{\varepsilon}}{\partial x_{j}}\right)^{2}+ B(u^{\varepsilon})\,\frac{\partial^{2}u^{\varepsilon}}{\partial x_{j}^{2}}\right) &\mbox{in }\Omega_{T},\label{eqnchap21A.a} \\
    u^{\varepsilon}(x,t)= 0&\,\,\,\,\mbox{on}\,\, \partial \Omega\times(0,T),\\
u^{\varepsilon}(x,0) = u_0(x)& x\in \Omega,
 \end{eqnarray}

\end{subequations}

For $t\in(0,T)$ and $i\in\left\{1,2,\cdots,d\right\}$, differentiating the equation  \eqref{eqnchap21A.a} with respect to $x_{i}$, multiplying by $sg_{n}\left(\frac{\partial u^{\varepsilon}}{\partial x_{i}}\right)$ and using $\frac{\partial u^{\varepsilon}}{\partial x_{i}}\in H^{1}(\Omega\times (0,T))$, we get
\begin{eqnarray}\label{BVestimateeqn2}
 \frac{\partial}{\partial t}\left(\frac{\partial u^{\varepsilon}}{\partial x_{i}}\right)\,sg_{n}\left(\frac{\partial u^{\varepsilon}}{\partial x_{i}}\right) + \displaystyle\sum_{j=1}^{d}sg_{n}\left(\frac{\partial u^{\varepsilon}}{\partial x_{i}}\right)\frac{\partial^{2}}{\partial x_{i}\partial x_{j}}f_{j}(u^{\varepsilon})\hspace{2in} \nonumber\\
 =\varepsilon\,\displaystyle\sum_{j=1}^{d}sg_{n}\left(\frac{\partial u^{\varepsilon}}{\partial x_{i}}\right)\frac{\partial}{\partial x_{i}}\left(B^{\prime}(u^{\varepsilon})\left(\frac{\partial u^{\varepsilon}}{\partial x_{j}}\right)^{2}
 + B(u^{\varepsilon})\,\frac{\partial^{2}u^{\varepsilon}}{\partial x_{j}^{2}}\right).
\end{eqnarray}
Summing \eqref{BVestimateeqn2}  over $i=1,2,\cdots,d$ and integrating over $\Omega$, we obtain
\begin{eqnarray}\label{BVestimateeqn4}
 \displaystyle\sum_{i=1}^{d}\int_{\Omega}\frac{\partial}{\partial t}\left(\frac{\partial u^{\varepsilon}}{\partial x_{i}}\right)\,sg_{n}\left(\frac{\partial u^{\varepsilon}}{\partial x_{i}}\right)\, dx + \displaystyle\sum_{i,j=1}^{d}\int_{\Omega}sg_{n}\left(\frac{\partial u^{\varepsilon}}{\partial x_{i}}\right)\frac{\partial}{\partial x_{i}}\left(f_{j}^{\prime}(u^{\varepsilon})\,\frac{\partial u^{\varepsilon}}{\partial x_{j}}\right)\,dx\nonumber\\
 =\varepsilon\,\displaystyle\sum_{i,j=1}^{d}\int_{\Omega}sg_{n}\left(\frac{\partial u^{\varepsilon}}{\partial x_{i}}\right)
 \frac{\partial}{\partial x_{i}}\Big(B^{\prime}(u^{\varepsilon})\left(\frac{\partial u^{\varepsilon}}{\partial x_{j}}\right)^{2}
 + B(u^{\varepsilon})\,\frac{\partial^{2}u^{\varepsilon}}{\partial x_{j}^{2}}\Big)\,dx
\end{eqnarray}
We pass to limit as $n\to\infty$ in each of the terms in \eqref{BVestimateeqn4} below.
\begin{enumerate}
\item[(i)]{Passage to  limit on RHS of \eqref{BVestimateeqn4} as $n\to\infty$ :}\\
Using integration by parts on RHS of \eqref{BVestimateeqn4}, we get
\begin{eqnarray}\label{BVestimateeqn5}
\varepsilon\displaystyle\sum_{i,j=1}^{d}\int_{\Omega}\,sg_{n}\left(\frac{\partial u^{\varepsilon}}{\partial x_{i}}\right)\frac{\partial}{\partial x_{i}}\Big(B^{\prime}(u^{\varepsilon})\left(\frac{\partial u^{\varepsilon}}{\partial x_{j}}\right)^{2}
 + B(u^{\varepsilon})\,\frac{\partial^{2}u^{\varepsilon}}{\partial x_{j}^{2}}\Big)\,dx\hspace*{0.7in}\nonumber\\
 = -\varepsilon\displaystyle\sum_{i,j=1}^{d}\int_{\Omega}\,sg_{n}^{\prime}\left(\frac{\partial u^{\varepsilon}}{\partial x_{i}}\right)\frac{\partial^{2} u^{\varepsilon}}{\partial x_{i}^{2}}
 \left(B^{\prime}(u^{\varepsilon})\left(\frac{\partial u^{\varepsilon}}{\partial x_{j}}\right)^{2}
+ B(u^{\varepsilon})\,\frac{\partial^{2}u^{\varepsilon}}{\partial x_{j}^{2}}\right)\,dx \nonumber\\  
+\varepsilon\displaystyle\sum_{i,j=1}^{d}\int_{\partial\Omega}\,sg_{n}\left(\frac{\partial u^{\varepsilon}}{\partial x_{i}}\right)
  \left(B(0)\,\frac{\partial^{2}u^{\varepsilon}}{\partial x_{j}^{2}}+B^\prime(0)\left(\frac{\partial u^{\varepsilon}}{\partial x_{j}}\right)^{2}\right)\,\sigma_{i}\,d\sigma.
\end{eqnarray}

Firstly, we want to show that
\begin{eqnarray}\label{BVestimateeqn6}
\displaystyle\lim_{n\to\infty}\displaystyle\sum_{i,j=1}^{d}\int_{\Omega}\,sg_{n}^{\prime}\left(\frac{\partial u^{\varepsilon}}{\partial x_{i}}\right)\frac{\partial^{2} u^{\varepsilon}}{\partial x_{i}^{2}}
 \left(B^{\prime}(u^{\varepsilon})\left(\frac{\partial u^{\varepsilon}}{\partial x_{j}}\right)^{2}
+ B(u^{\varepsilon})\,\frac{\partial^{2}u^{\varepsilon}}{\partial x_{j}^{2}}\right)\,dx=0.
\end{eqnarray}
For $i\in\left\{1,2,\cdots,d\right\}$, denote 
$$A^{\varepsilon}_{i}:=\left\{x\in\Omega\,\,:\,\,\frac{\partial u^{\varepsilon}}{\partial x_{i}}=0\right\}.$$
Since $\frac{\partial u^{\varepsilon}}{\partial x_{i}}\in C^{1}(\overline{\Omega})$, Stampacchia's theorem (see \cite{kesavan}) gives  $\nabla \left(\frac{\partial u^{\varepsilon}}{\partial x_{i}}\right)=0\,\,{\it a.e.}\,\,x\in A^{\varepsilon}_{i}$. In particular, $\frac{\partial^{2}u^{\varepsilon}}{\partial x_{i}^{2}}=0\,\,{\it a.e.}\,\,x\in A^{\varepsilon}_{i}$. 
For $i\in\left\{1,2,\cdots,d\right\}$, if $\Omega\smallsetminus A_{i}^{\varepsilon}=\varnothing$, then \eqref{BVestimateeqn6} follows trivially. Assuming that $\Omega\smallsetminus A_{i}^{\varepsilon}\neq\varnothing$, for each $i,j\in\{1,2,\cdots,d\}$ and for each $x\in\Omega\smallsetminus A^{\varepsilon}_{i}$, we have
\begin{eqnarray}\label{mollifier.bv.eqn10}
sg_{n}^{\prime}\left(\frac{\partial u^{\varepsilon}}{\partial x_{i}}\right)\frac{\partial^{2} u^{\varepsilon}}{\partial x_{i}^{2}}
 \Big(B^{\prime}(u^{\varepsilon})\left(\frac{\partial u^{\varepsilon}}{\partial x_{j}}\right)^{2}
+ B(u^{\varepsilon})\,\frac{\partial^{2}u^{\varepsilon}}{\partial x_{j}^{2}}\Big)\to 0\,\,\mbox{as}\,\,n\to\infty.\hspace*{0.2in}
\end{eqnarray}
Note that 
\begin{eqnarray}\label{mollifier.bv.eqn1}
 \left|-\displaystyle\sum_{i,j=1}^{d}\int_{\Omega\smallsetminus A_{i}^{\varepsilon}}\,sg_{n}^{\prime}\left(\frac{\partial u^{\varepsilon}}{\partial x_{i}}\right)\frac{\partial^{2} u^{\varepsilon}}{\partial x_{i}^{2}}
 \Big(B^{\prime}(u^{\varepsilon})\left(\frac{\partial u^{\varepsilon}}{\partial x_{j}}\right)^{2}
+ B(u^{\varepsilon})\,\frac{\partial^{2}u^{\varepsilon}}{\partial x_{j}^{2}}\Big)\,dx\right|\hspace*{0.5in}\nonumber\\ \leq\displaystyle\sum_{i,j=1}^{d}\int_{\Omega\smallsetminus A_{i}^{\varepsilon}}\,sg_{n}^{\prime}\left(\frac{\partial u^{\varepsilon}}{\partial x_{i}}\right)\left|\frac{\partial^{2} u^{\varepsilon}}{\partial x_{i}^{2}}
 \Big(B^{\prime}(u^{\varepsilon})\left(\frac{\partial u^{\varepsilon}}{\partial x_{j}}\right)^{2}
+ B(u^{\varepsilon})\,\frac{\partial^{2}u^{\varepsilon}}{\partial x_{j}^{2}}\Big)\right|\,dx.\nonumber\\
{}
\end{eqnarray}
Let $t\in (0,T)$ be fixed. For $i\in\left\{1,2,\cdots,d\right\}$, let $x_{i0}^{\varepsilon}\in A_{i}^{\varepsilon}$. Since $\Omega$ is bounded, there exists $R^{\prime}> 0$ such that $\Omega\subset B(x_{i0}^{\varepsilon},\,R^{\prime})$, where $B(x_{i0}^{\varepsilon},\,R^{\prime})$ denotes the open ball with center at $x_{i0}^{\varepsilon}$ and having radius $R'$. Consequently, $\Omega\subset B(x_{i0}^{\varepsilon},\,2R^{\prime})$. 

For each $n\in\N$ and $i\in\{1,2,\cdots,d\}$, denote 
\begin{eqnarray}\label{mollifier.bv.eqn2}
 C_{n,i}^{\varepsilon}:=\left\{x\in\Omega\smallsetminus A_{i}^{\varepsilon}:\,\,0<\left|\frac{\partial u^{\varepsilon}}{\partial x_{i}}(x,t)\right|\leq\frac{1}{n}\right\}.
\end{eqnarray}
Note that for each $n\in\mathbb{N}$,   we have $C_{(n+1),i}^{\varepsilon}\subseteq C_{n,i}^{\varepsilon}$.
Let $n_{0}\in\mathbb{N}$ be such that for all $n\geq n_{0}$, the following inclusion holds:
$$C_{n,i}^{\varepsilon}\subset B(x_{i0}^{\varepsilon},\frac{n}{2}).$$ 
Define a function $\rho\in C^{\infty}(\mathbb{R}^{d})$   by
\begin{eqnarray}\label{mpllifier.bv.eqn11}
\rho(x) := \left\{\def\arraystretch{1.2}%
\begin{array}{@{}c@{\quad}l@{}}
  k_\varepsilon\,\exp\left(-\frac{1}{1-|x-x_{i0}^{\varepsilon}|^{2}}\right), & \text{if $x\in B(x_{i0}^{\varepsilon},\,1)$}\\
  0,\,\, & \text{if\,\,$x\notin B(x_{i0}^{\varepsilon},\,1)$},
\end{array}\right.
\end{eqnarray}
where the constant $k_\varepsilon$ is chosen so that
\begin{eqnarray}\label{mollifier.bv.eqn4}
\int_{\mathbb{R}^{d}}\rho(x)\,\,dx=1.
\end{eqnarray}
Denote the sequence of mollifiers $\rho_{n,i}:\mathbb{R}^{d}\to\mathbb{R}$ by 
\begin{eqnarray}\label{mollifier.bv.eqn5}
\rho_{n,i}(x) := \left\{\def\arraystretch{1.2}%
\begin{array}{@{}c@{\quad}l@{}}
  k_\varepsilon\,n^{d}\,\exp\left(-\frac{n^{2}}{n^{2}-|x-x_{i0}^{\varepsilon}|^{2}}\right), & \text{if $x\in B(x_{i0}^{\varepsilon},\,n)$}\\
  0,\,\, & \text{if\,\,$x\notin B(x_{i0}^{\varepsilon},\,n)$}
\end{array}\right.
\end{eqnarray}
Since for each $n\geq n_{0}$, 
\begin{eqnarray}\label{mollifier.bv.eqn6}
 sg_{n}^{\prime}\left(\frac{\partial u^{\varepsilon}}{\partial x_{i}}\right):= \left\{\def\arraystretch{1.2}%
  \begin{array}{@{}c@{\quad}l@{}}
   n & \text{if $0\leq\left|\frac{\partial u^{\varepsilon}}{\partial x_{i}}\right|\leq\frac{1}{n}$},\\
    0\,\, & \text{if\,\,$\left|\frac{\partial u^{\varepsilon}}{\partial x_{i}}\right|>\frac{1}{n},$}
    \end{array}\right.
\end{eqnarray}
the RHS of \eqref{mollifier.bv.eqn1} can be rewritten as
\begin{eqnarray}\label{mollifier.bv.eqn7}
 \displaystyle\sum_{i,j=1}^{d}\int_{\Omega\smallsetminus A_{i}^{\varepsilon}}\,sg_{n}^{\prime}\left(\frac{\partial u^{\varepsilon}}{\partial x_{i}}\right)\left|\frac{\partial^{2} u^{\varepsilon}}{\partial x_{i}^{2}}
 \Big(B^{\prime}(u^{\varepsilon})\left(\frac{\partial u^{\varepsilon}}{\partial x_{j}}\right)^{2}
+ B(u^{\varepsilon})\,\frac{\partial^{2}u^{\varepsilon}}{\partial x_{j}^{2}}\Big)\right|\,dx\hspace*{1.7in}\nonumber\\
= \displaystyle\sum_{i,j=1}^{d}\int_{B(x_{i0}^{\varepsilon},\frac{n}{2})}\chi_{\Omega\smallsetminus A_{i}^{\varepsilon}}(x)\,sg_{n}^{\prime}\left(\frac{\partial u^{\varepsilon}}{\partial x_{i}}\right)\left|\frac{\partial^{2} u^{\varepsilon}}{\partial x_{i}^{2}}
 \Big(B^{\prime}(u^{\varepsilon})\left(\frac{\partial u^{\varepsilon}}{\partial x_{j}}\right)^{2}
+ B(u^{\varepsilon})\,\frac{\partial^{2}u^{\varepsilon}}{\partial x_{j}^{2}}\Big)\right|\,dx\hspace*{0.4in}\nonumber\\
= \displaystyle\sum_{i,j=1}^{d}\int_{B(x_{i0}^{\varepsilon},\frac{n}{2})}\chi_{\Omega\smallsetminus A_{i}^{\varepsilon}}(x)\,\frac{sg_{n}^{\prime}\left(\frac{\partial u^{\varepsilon}}{\partial x_{i}}\right)}{\rho_{n,i}(x)}\rho_{n,i}(x)\left|\frac{\partial^{2} u^{\varepsilon}}{\partial x_{i}^{2}}
 \Big(B^{\prime}(u^{\varepsilon})\left(\frac{\partial u^{\varepsilon}}{\partial x_{j}}\right)^{2}
+ B(u^{\varepsilon})\,\frac{\partial^{2}u^{\varepsilon}}{\partial x_{j}^{2}}\Big)\right|\,dx.\nonumber\\
{}
\end{eqnarray}
For all $n\geq n_{0}$ and for $x\in B(x_{i0}^{\varepsilon},\,\frac{n}{2})$, we have
\begin{eqnarray}\label{mollifier.bv.eqn8}
 \frac{sg_{n}^{\prime}\left(\frac{\partial u^{\varepsilon}}{\partial x_{i}}\right)}{\rho_{n,i}(x)}:= \left\{\def\arraystretch{1.2}%
  \begin{array}{@{}c@{\quad}l@{}}
   \frac{1}{k_\varepsilon\,n^{d-1}}\,e^{\frac{n^{2}}{n^{2}-|x-x_{i0}^{\varepsilon}|^{2}}} & \text{if $x\in B(x_{i0}^{\varepsilon},\,\frac{n}{2})\cap C_{n,i}^{\varepsilon}$},\\
    0\,\, & \text{if\,\,$x\notin B(x_{i0}^{\varepsilon},\,\frac{n}{2})\cap C_{n,i}^{\varepsilon}.$}
    \end{array}\right.
\end{eqnarray}
For $x\in  B(x_{i0}^{\varepsilon},\,\frac{n}{2})$, we have
$$\left| \frac{sg_{n}^{\prime}\left(\frac{\partial u^{\varepsilon}}{\partial x_{i}}\right)}{\rho_{n,i}(x)}\right|\leq\frac{1}{k_\varepsilon\,n^{d-1}}\,e^{\frac{4}{3}}.$$
\vspace{0.1cm}\\
Since $n\in\mathbb{N}$, the integrand on the last line of  \eqref{mollifier.bv.eqn7} is dominated by 
$$\frac{1}{k_\varepsilon}\,e^{\frac{4}{3}}\,\rho_{n,i}\,\left|\frac{\partial^{2} u^{\varepsilon}}{\partial x_{i}^{2}}
 \Big(B^{\prime}(u^{\varepsilon})\left(\frac{\partial u^{\varepsilon}}{\partial x_{j}}\right)^{2}
+ B(u^{\varepsilon})\,\frac{\partial^{2}u^{\varepsilon}}{\partial x_{j}^{2}}\Big)\right|,$$
which is integrable as $u^{\varepsilon}\in C^{4+\beta,\frac{4+\beta}{2}}(\overline{\Omega_{T}})$ and 
$$\int_{\mathbb{R}}\rho_{n,i}(x)\,dx=1.$$ 
Applying dominated convergence theorem on RHS of \eqref{mollifier.bv.eqn7}, we conclude \eqref{BVestimateeqn6}. 
\vspace{0.2cm}\\

Next, we want to show that 
\begin{eqnarray}\label{BVestimateeqn8C}
 \displaystyle\lim_{n\to\infty}\displaystyle\sum_{i,j=1}^{d}\int_{\partial\Omega}\,sg_{n}\left(\frac{\partial u^{\varepsilon}}{\partial x_{i}}\right)
  \left(B(0)\,\frac{\partial^{2}u^{\varepsilon}}{\partial x_{j}^{2}}+B^\prime(0)\left(\frac{\partial u^{\varepsilon}}{\partial x_{j}}\right)^{2}\right)\,\sigma_{i}\,d\sigma\nonumber\\
  =  \sum_{i,j=1}^{d}\int_{\partial\Omega}\,sg\left(\frac{\partial u^{\varepsilon}}{\partial x_{i}}\right)
  \left(B(0)\,\frac{\partial^{2}u^{\varepsilon}}{\partial x_{j}^{2}}+B^\prime(0)\left(\frac{\partial u^{\varepsilon}}{\partial x_{j}}\right)^{2}\right)\,\sigma_{i}\,d\sigma.
\end{eqnarray}
Note that \eqref{BVestimateeqn8C} follows by applying dominated convergence theorem, on noting that the sequence of integrands converges to the required limit pointwise, and is pointwise bounded by 
$$B(0)\left|\frac{\partial^{2}u^{\varepsilon}}{\partial x_{j}^{2}}\right|+|B^\prime(0)|\left(\frac{\partial u^{\varepsilon}}{\partial x_{j}}\right)^{2},$$
which is an integrable function as $u^{\varepsilon}\in C^{4+\beta,\frac{4+\beta}{2}}(\overline{\Omega_{T}})$ and $\mbox{surface measure}(\partial\Omega)<\infty$. \\

Using \eqref{BVestimateeqn6},  \eqref{BVestimateeqn8C} in \eqref{BVestimateeqn5}, we get 
\begin{eqnarray}\label{BVestimateeqn10}
 \varepsilon\displaystyle\lim_{n\to\infty}\displaystyle\sum_{i,j=1}^{d}\int_{\Omega}\,sg_{n}\left(\frac{\partial u^{\varepsilon}}{\partial x_{i}}\right)\frac{\partial}{\partial x_{i}}\Big(B^{\prime}(u^{\varepsilon})\left(\frac{\partial u^{\varepsilon}}{\partial x_{j}}\right)^{2}
 + B(u^{\varepsilon})\,\frac{\partial^{2}u^{\varepsilon}}{\partial x_{j}^{2}}\Big)\,dx\nonumber\\
 = \varepsilon\sum_{i,j=1}^{d}\int_{\partial\Omega}\,sg\left(\frac{\partial u^{\varepsilon}}{\partial x_{i}}\right)
  \left(B(0)\,\frac{\partial^{2}u^{\varepsilon}}{\partial x_{j}^{2}}+B^\prime(0)\left(\frac{\partial u^{\varepsilon}}{\partial x_{j}}\right)^{2}\right)\,\sigma_{i}\,d\sigma.
\end{eqnarray}
From the equation \eqref{ibvp.parab.a}, we have
\begin{eqnarray}\label{BVestimateeqn11}
 \varepsilon\displaystyle\sum_{j=1}^{d}\Big(B^{\prime}(u^{\varepsilon})\left(\frac{\partial u^{\varepsilon}}{\partial x_{j}}\right)^{2} + B(u^{\varepsilon})\,\frac{\partial^{2}u^{\varepsilon}}{\partial x_{j}^{2}}\Big) = u_{t}^{\varepsilon} + \displaystyle\sum_{j=1}^{d}f_{j}^{\prime}(u^{\varepsilon})\,u^{\varepsilon}_{x_{j}}\,\,\mbox{in}\,\,\overline{\Omega_{T}}.
\end{eqnarray}
Using \eqref{BVestimateeqn11} on RHS of \eqref{BVestimateeqn10}, we get
\begin{eqnarray}\label{BVestimateeqn12}
 \varepsilon\displaystyle\lim_{n\to\infty}\displaystyle\sum_{i,j=1}^{d}\int_{\Omega}\,sg_{n}\left(\frac{\partial u^{\varepsilon}}{\partial x_{i}}\right)\frac{\partial}{\partial x_{i}}\Big(B^{\prime}(u^{\varepsilon})\left(\frac{\partial u^{\varepsilon}}{\partial x_{j}}\right)^{2}
 + B(u^{\varepsilon})\,\frac{\partial^{2}u^{\varepsilon}}{\partial x_{j}^{2}}\Big)\,dx\nonumber\\
 = \sum_{i,j=1}^{d}\int_{\partial\Omega}\,sg\left(\frac{\partial u^{\varepsilon}}{\partial x_{i}}\right)
  f_j^\prime(0)\frac{\partial u^{\varepsilon}}{\partial x_{j}}\,\sigma_{i}\,d\sigma.
\end{eqnarray}

\item[(ii)]{Passage to  limit in the second term on LHS of \eqref{BVestimateeqn4} as $n\to\infty$ :}\\
Using integration by parts, we have
\begin{eqnarray}\label{BVestimateeqn14}
  \displaystyle\sum_{i,j=1}^{d}\int_{\Omega}sg_{n}\left(\frac{\partial u^{\varepsilon}}{\partial x_{i}}\right)\,\,\frac{\partial}{\partial x_{i}}\left(f_{j}^{\prime}(u^{\varepsilon})\frac{\partial u^{\varepsilon}}{\partial x_{j}}\right)\,dx\hspace*{3in}\nonumber\\
  = -\displaystyle\sum_{i,j=1}^{d}\int_{\Omega}sg_{n}^{\prime}\left(\frac{\partial u^{\varepsilon}}{\partial x_{i}}\right)\,\frac{\partial^{2} u^{\varepsilon}}{\partial x_{i}^{2}}\,\,f_{j}^{\prime}(u^{\varepsilon})\frac{\partial u^{\varepsilon}}{\partial x_{j}}\,dx + \displaystyle\sum_{i,j=1}^{d}\int_{\partial\Omega}f_{j}^{\prime}(0)\,\frac{\partial u^{\varepsilon}}{\partial x_{j}}\,sg_{n}\left(\frac{\partial u^{
\varepsilon}}{\partial x_{i}}\right)\,\,\sigma_
{i}\,d\sigma\nonumber\\
  {}
\end{eqnarray}
Using the arguments given to prove  assertions \eqref{BVestimateeqn6} and \eqref{BVestimateeqn8C}, we obtain
\begin{eqnarray} 
 \hspace*{0.2in}\displaystyle\lim_{n\to\infty}\displaystyle\sum_{i,j=1}^{d}\int_{\Omega}sg_{n}^{\prime}\left(\frac{\partial u^{\varepsilon}}{\partial x_{i}}\right)\,\frac{\partial^{2} u^{\varepsilon}}{\partial x_{i}^{2}}\,\,f_{j}^{\prime}(u^{\varepsilon})\frac{\partial u^{\varepsilon}}{\partial x_{j}}\,dx =0,\hspace*{2in}\label{BVestimateeqn15}\\
  \displaystyle\lim_{n\to\infty}\displaystyle\sum_{i,j=1}^{d}\int_{\partial\Omega}f_{j}^{\prime}(u^{\varepsilon})\,\frac{\partial u^{\varepsilon}}{\partial x_{j}}\,sg_{n}\left(\frac{\partial u^{\varepsilon}}{\partial x_{i}}\right)\,\,\sigma_{i}\,d\sigma =\displaystyle\sum_{i,j=1}^{d}\int_{\partial\Omega}f_{j}^{\prime}(0)\,\frac{\partial u^{\varepsilon}}{\partial x_{j}}\,sg\left(\frac{\partial u^{\varepsilon}}{\partial x_{i}}\right)\,\,\sigma_{i}\,d\sigma.\nonumber\\
  {}\label{BVestimateeqn16A}
\end{eqnarray}
 
\end{enumerate}
On using equations \eqref{BVestimateeqn12}, \eqref{BVestimateeqn15}, and \eqref{BVestimateeqn16A}, the equation \eqref{BVestimateeqn4} yields
\begin{eqnarray}\label{BVestimateeqn21}
  \displaystyle\sum_{i=1}^{d}\int_{\Omega}\frac{\partial}{\partial t}\left(\frac{\partial u^{\varepsilon}}{\partial x_{i}}\right)\,sg\left(\frac{\partial u^{\varepsilon}}{\partial x_{i}}\right)\, dx= \frac{\partial}{\partial t}\left(\displaystyle\sum_{i=1}^{d}\int_{\Omega}\Big|\frac{\partial u^{\varepsilon}}{\partial x_{i}}\Big|\,\,dx\right) =0.
\end{eqnarray}
The last equation implies that 
\begin{eqnarray}\label{BVestimateeqn24}
 \|\nabla u^{\varepsilon}\|_{\left(L^{1}(\Omega_{T})\right)^{d}}= \|\nabla u_{0}\|_{\left(L^{1}(\Omega)\right)^{d}},\nonumber  
\end{eqnarray}
which is nothing but \eqref{B.BVestimate27}. Equations \eqref{B.BVestimate26} and \eqref{B.BVestimate27} together give \eqref{BVestimateA.eqn1}.\\
\vspace{0.2cm}\\
\textbf{Step 3:} Denote the total variation of $u^{\varepsilon}$ by $TV_{\Omega_{T}}(u^{\varepsilon})$. Since for each $\varepsilon > 0$, $u^{\varepsilon}\in H^{1}(\Omega_{T})$, the total variation of $u^{\varepsilon}$ is given by \eqref{BVestimateA.eqn1},{\it i.e.,}
\begin{eqnarray}\label{B.BVestimate28}
 TV_{\Omega_{T}}(u^{\varepsilon})= \left\|\frac{\partial u^{\varepsilon}}{\partial t}\right\|_{L^{1}(\Omega_{T})} +  \|\nabla u^{\varepsilon}\|_{\left(L^{1}(\Omega_{T})\right)^{d}}.
\end{eqnarray}
Equation \eqref{BVestimateA.eqn1} shows that $\left\{TV_{\Omega_{T}}(u^{\varepsilon})\right\}_{\varepsilon\geq 0}$ is bounded.
Since $BV(\Omega_{T})\cap L^{1}(\Omega_{T})$ is compactly imbedded in $L^{1}(\Omega_{T})$ \cite{MR1304494}, therefore there exists a subsequence $(u^{\varepsilon_{k}})$ and a function $u\in L^{1}(\Omega_{T})$ such that $u^{\varepsilon_{k}}\to u$ in $L^{1}(\Omega_{T})$ as $k\to\infty$. We still denote the subsequence by $(u^{\varepsilon})$ and since $u^{\varepsilon}\to u$ in $L^{1}(\Omega_{T})$ as $\varepsilon\to 0$, there exists a further subsequence $(u^{\varepsilon_{k}})$ such that we have \eqref{BVestimateeqn25}.

\begin{theorem}\label{Compactness.lemma.1}
Let $f,\,\,B,\,\,u_{0}$ satisfy Hypothesis A. Let $u^{\varepsilon}$ be the unique solution to generalized viscosity problem \eqref{ibvp.parab}. Then 
 \begin{eqnarray}\label{uniformnot.compactness.eqn1a}
 \displaystyle\sum_{j=1}^{d} \hspace{0.1cm}\left(\sqrt{\varepsilon}\Big\| \frac{\partial u^{\varepsilon}}{\partial x_{j}}\Big\|_{L^{2}(\Omega_{T})}\right)^{2} \leq\frac{1}{2r}\|u_{0}\|^{2}_{L^{2}(\Omega)}  
 \end{eqnarray}
\end{theorem}
\begin{proof}
 Substituing $v=u^{\varepsilon}$ in the equation \eqref{ibvp.parab.weak.a} of generalized viscosity problem \eqref{ibvp.parab} and using chain rule and then integration by parts formula, for {\it a.e.} $t\in(0,T)$, we obtain
\begin{equation}\label{chap9eqn7}
 \begin{split}
 \langle u^{\varepsilon}_{t},u^{\varepsilon} \rangle + \varepsilon\hspace{0.1cm}\displaystyle\sum_{j=1}^{d}\int_{\Omega}\hspace{0.1cm}B(u^{\varepsilon})\hspace{0.1cm}\frac{\partial u^{\varepsilon}}{\partial x_{j}}\hspace{0.1cm}\frac{\partial u^{\varepsilon}}{\partial x_{j}}\hspace{0.1cm}dx - \displaystyle\sum_{j=1}^{d}\int_{\Omega}\hspace{0.1cm}f_{j}(u^{\varepsilon})\hspace{0.1cm}\frac{\partial u^{\varepsilon}}{\partial x_{j}}\hspace{0.1cm}dx =0.
 \end{split}
 \end{equation}
 Integrating \eqref{chap9eqn7} with respect to $t$ over the interval $(0,T)$ and using the formula
$$<u^{\varepsilon}_{t}, u^{\varepsilon}> =\frac{1}{2}\hspace{0.1cm}\frac{d}{dt}\|u^{\varepsilon}(t)\|^{2}_{L^{2}(\Omega)},$$ we get
\begin{equation}\label{chap9eqn8}
 \frac{1}{2}\|u(T)\|^{2}_{L^{2}(\Omega)} + \varepsilon r \displaystyle\sum_{j=1}^{d} \int_{0}^{T}\int_{\Omega}\hspace{0.1cm}\left(\frac{\partial u^{\varepsilon}}{\partial x_{j}}\right)^{2}\hspace{0.1cm}dx\hspace{0.1cm}dt = \frac{1}{2}\|u_{0}\|^{2}_{L^{2}(\Omega)} + \displaystyle\sum_{j=1}^{d}\int_{0}^{T}\int_{\Omega}f_{j}(u^{\varepsilon})\hspace{0.1cm}\frac{\partial u^{\varepsilon}}{\partial x_{j}}\hspace{0.1cm}dx\hspace{0.1cm}dt.
\end{equation}
From equation \eqref{chap9eqn8}, we obtain 
\begin{equation}\label{compactnessequation13}
 \varepsilon r \displaystyle\sum_{j=1}^{d} \int_{0}^{T}\int_{\Omega}\hspace{0.1cm}\left(\frac{\partial u^{\varepsilon}}{\partial x_{j}}\right)^{2}\hspace{0.1cm}dx\hspace{0.1cm}dt \leq \frac{1}{2}\|u_{0}\|^{2}_{L^{2}(\Omega)} + \displaystyle\sum_{j=1}^{d}\int_{0}^{T}\int_{\Omega}f_{j}(u^{\varepsilon})\hspace{0.1cm}\frac{\partial u^{\varepsilon}}{\partial x_{j}}\hspace{0.1cm}dx\hspace{0.1cm}dt.
\end{equation}
For $j=1,2\cdots,d$, denote
\begin{equation}\label{chap9eqn20}
 g_{j}(s):= \int_{0}^{s} f_{j}(\lambda)\hspace{0.1cm}d\lambda.\nonumber
\end{equation}
Since for $j=1,2\cdots,d$, each $f_{j}:\mathbb{R}\to\mathbb{R}$ is continuous, in view of fundamental theorem of calculus, we have
\begin{equation}\label{chap9eqn22}
 g_{j}^{\prime}(u^{\varepsilon})= f_{j}(u^{\varepsilon}),
\end{equation}
for {\it a.e.} $(x,t)\in\Omega_{T}$. Substituting $ g_{j}^{\prime}(u^{\varepsilon})= f_{j}(u^{\varepsilon})$ and using chain rule in inequality \eqref{compactnessequation13}, we obtain
\begin{equation}\label{chap9eqn24}
 \varepsilon r \displaystyle\sum_{j=1}^{d} \int_{0}^{T}\int_{\Omega}\hspace{0.1cm}\left(\frac{\partial u^{\varepsilon}}{\partial x_{j}}\right)^{2}\hspace{0.1cm}dx\hspace{0.1cm}dt \leq\frac{1}{2}\|u_{0}\|^{2}_{L^{2}(\Omega)} + \displaystyle\sum_{j=1}^{d}\int_{0}^{T}\int_{\Omega}\frac{\partial g_{j}(u^{\varepsilon})}{\partial x_{j}}\hspace{0.1cm}dx\hspace{0.1cm}dt.
\end{equation}
Using integration by parts formula and $u^{\varepsilon}=0$ for {\it a.e.} $(x,t)\in\partial\Omega\times (0,T)$ in inequality \eqref{chap9eqn24}, we get
\begin{equation}\label{compactnessequation8}
 \varepsilon r \displaystyle\sum_{j=1}^{d} \int_{0}^{T}\int_{\Omega}\hspace{0.1cm}\left(\frac{\partial u^{\varepsilon}}{\partial x_{j}}\right)^{2}\hspace{0.1cm}dx\hspace{0.1cm}dt \leq \frac{1}{2}\|u_{0}\|^{2}_{L^{2}(\Omega)} + \displaystyle\sum_{j=1}^{d}\int_{0}^{T}\int_{\partial\Omega}g_{j}(0)\nu_{j}\hspace{0.1cm}d\nu.
\end{equation}
Since $g_j(0)=0$, the inequality \eqref{compactnessequation8} reduces to \eqref{uniformnot.compactness.eqn1a}.
\end{proof}

\section{Proof of Theorem~\ref{theorem1}}\label{section.prooftheorem1}

In this section, we want to show that the {\it a.e.} limit of the subsequence of solutions to generalized viscosity problem \eqref{ibvp.parab} obtained in Theorem~\ref{BVEstimate.thm1} is an entropy solution to scalar conservation 
laws \eqref{ivp.cl}. We do not prove the uniqueness of the entropy solution as it was already established in \cite{MR542510}. The notion of entropy employed here was introduced by Bardos-Leroux-Nedelec \cite{MR542510} for IBVPs for conservation laws, which is defined below.  
\begin{definition}
Let $u\in BV(\Omega\times(0,T))$ and  $\gamma(u)$ be the trace of $u$ on  $\partial\Omega$. The function $u$ is said to be an entropy solution of IBVP \eqref{ivp.cl} if for all $k\in\mathbb{R}$ and for all nonnegative $\phi\in C^{2}(\overline{\Omega}\times(0,T))$ with compact support in $\overline{\Omega}\times(0,T),$ the inequality
  \begin{eqnarray}\label{B.entropysolution.eqn1}
   \int_{0}^{T}\int_{\Omega}\left\{\left|u-k\right|\,\frac{\partial\phi}{\partial t} + sg(u-k)\,\left(f(u)-f(k)\right)\cdot \nabla\phi \right\}\,dx\,dt \nonumber\\+  \int_{0}^{T}\int_{\partial\Omega}sg(k)\left(f(\gamma(u))-f(k)\right)\cdot\sigma\,\phi\,d\sigma\,dt\geq 0
  \end{eqnarray}
holds, and $u$ satisfies the initial condition \eqref{ivp.cl.c} almost everywhere in $\Omega$.
\end{definition}
We once again drop the subscript $k$ in the subsequence $u^{\varepsilon_k}$. Since uniqueness of an entropy solution is already established by Bardos {\it et.al.} in \cite{MR542510}, we conclude that the full sequence $(u^\varepsilon)$ converges to the unique entropy solution to \eqref{ivp.cl}.

\noindent{\bf Proof of Theorem~\ref{theorem1}:}

We prove Theorem~\ref{theorem1} in two steps. In Step 1, we show that $u$ is a weak solution to IBVP \eqref{ivp.cl} and in Step 2, we show that weak solution $u$ is an entropy solution to IBVP \eqref{ivp.cl}

\noindent \textbf{Step 1:} Let $\phi \in \mathcal{D}(\Omega\times [0, T))$. Multiplying the equation \eqref{ibvp.parab.a} by $\phi$, integrating over $\Omega_{T}$, and using integration by parts formula, we get
 \begin{equation}\label{eqnchap502}
\begin{split}
 \int_{0}^{T}\int_{\Omega}u^{\varepsilon}\hspace{0.1cm}\frac{\partial \phi}{\partial t}\hspace{0.1cm}dx\hspace{0.1cm}dt + \varepsilon\hspace{0.1cm}\displaystyle\sum_{j=1}^{d}\int_{0}^{T}\hspace{0.1cm}\int_{\Omega}\hspace{0.1cm}B(u^{\varepsilon})\hspace{0.1cm}\frac{\partial u^{\varepsilon}}{\partial x_{j}}\hspace{0.1cm}\frac{\partial \phi}{\partial x_{j}}\hspace{0.1cm}dx\hspace{0.1cm}dt + \displaystyle\sum_{j=1}^{d}\int_{0}^{T}\hspace{0.1cm}\int_{\Omega}\hspace{0.1cm}f_{j}(u^{\varepsilon})\hspace{0.1cm}\frac{\partial \phi}{\partial x_{j}}\hspace{0.1cm}dx\hspace{0.1cm}dt\\ =\int_{\Omega}u^{\varepsilon}(x,0)\hspace{0.1cm}\phi(x,0)\hspace{0.1cm}dx\hspace{0.1cm}dt.
 \end{split}
\end{equation}
We pass to the limit as $\varepsilon\to 0$ in \eqref{eqnchap502} in four steps.\\
\vspace{0.1cm}\\
\textbf{Step 1A:}$\left(\mbox{Passage to limit in  the first term on LHS of \eqref{eqnchap502} } \right)$\\ 
Since for {\it a.e.} $(x,t)\in\Omega_{T}$, 
$$u^{\varepsilon}\,\frac{\partial \phi}{\partial t}\to u\,\frac{\partial \phi}{\partial t}\,\,\mbox{as}\,\,\varepsilon\to 0,$$
and the quantity $u^{\varepsilon}\,\frac{\partial \phi}{\partial t}$ is uniformly bounded by $\|u_{0}\|_{L^{\infty}(\Omega)}\Big|\frac{\partial \phi}{\partial t}\Big|$,  bounded convergence theorem yields
\begin{equation}\label{weaksolution.eqn1}
 \displaystyle\lim_{\varepsilon\to 0}\int_{0}^{T}\hspace{0.1cm}\int_{\Omega}u^{\varepsilon}\hspace{0.1cm}\frac{\partial \phi}{\partial t}\hspace{0.1cm}dx\hspace{0.1cm}dt =\int_{0}^{T}\hspace{0.1cm}\int_{\Omega}u\hspace{0.1cm}\frac{\partial \phi}{\partial t}\hspace{0.1cm}dx\hspace{0.1cm}dt.
\end{equation}
\textbf{Step 1B:}$\left(\mbox{Passage to limit in  the  second term on LHS of \eqref{eqnchap502} }\right)$\\
Note that for each $j\in\left\{1,2,\cdots,d\right\}$,
\begin{eqnarray}\label{chap5eqn13}
\Big|\varepsilon\hspace{0.1cm}\int_{0}^{T}\hspace{0.1cm}\int_{\Omega}\hspace{0.1cm}B(u^{\varepsilon})\hspace{0.1cm}\frac{\partial u^{\varepsilon}}{\partial x_{j}}\hspace{0.1cm}\frac{\partial \phi}{\partial x_{j}}\hspace{0.1cm}dx\hspace{0.1cm}dt\Big|&\leq& \|B\|_{\infty}\,\varepsilon\hspace{0.1cm}\int_{0}^{T}\int_{\Omega}\Big|\frac{\partial u^{\varepsilon}}{\partial x_{j}}\Big|\hspace{0.1cm}\Big|\frac{\partial\phi}{\partial x_{j}}\Big|\hspace{0.1cm}dx\hspace{0.1cm}dt\nonumber\\
&\leq& \|B\|_{\infty}\sqrt{\varepsilon}\left(\sqrt{\varepsilon}\Big\|\frac{\partial u^{\varepsilon}}{\partial x_{j}}\Big\|_{L^{2}(\Omega_{T})}\right)\Big\|\frac{\partial\phi}{\partial x_{j}}\Big\|_{L^{2}(\Omega_{T})}.
\end{eqnarray}
In view of \eqref{uniformnot.compactness.eqn1a}, the sequence of numbers $\left\{\sqrt{\varepsilon}\Big\|\frac{\partial u^{\varepsilon}}{\partial x_{j}}\Big\|_{L^{2}(\Omega_{T})}\right\}_{\varepsilon\geq 0}$ is bounded, and using sandwich theorem the inequality \eqref{chap5eqn13} gives
\begin{equation}\label{eqnchap504}
 \displaystyle\lim_{\varepsilon\to 0}\,\varepsilon\hspace{0.1cm}\displaystyle\sum_{j=1}^{d}\int_{0}^{T}\hspace{0.1cm}\int_{\Omega}\hspace{0.1cm}B(u^{\varepsilon})\hspace{0.1cm}\frac{\partial u^{\varepsilon}}{\partial x_{j}}\hspace{0.1cm}\frac{\partial \phi}{\partial x_{j}}\hspace{0.1cm}dx\hspace{0.1cm}dt = 0.
\end{equation}
\textbf{Step 1C:}$\left(\mbox{Passage to limit in  the  third term on LHS of \eqref{eqnchap502}}\right)$\\
In this step, we show that
\begin{equation}\label{weaksolution.eqn2}
  \displaystyle\lim_{\varepsilon\to 0}\displaystyle\sum_{j=1}^{d}\int_{0}^{T}\hspace{0.1cm}\int_{\Omega}\hspace{0.1cm}f_{j}(u^{\varepsilon})\hspace{0.1cm}\frac{\partial \phi}{\partial x_{j}}\hspace{0.1cm}dx\hspace{0.1cm}dt = \int_{0}^{T}\hspace{0.1cm}\int_{\Omega}\hspace{0.1cm} f_{j}(u)\hspace{0.1cm}\frac{\partial \phi}{\partial x_{j}}\hspace{0.1cm}dx\hspace{0.1cm}dt.
\end{equation}
For each $j\in\left\{1,2,\cdots,d\right\}$ and {\it a.e.} $(x,t)\in\Omega_{T}$, we have 
$$f_{j}(u^{\varepsilon})\hspace{0.1cm}\frac{\partial \phi}{\partial x_{j}}\to f_{j}(u)\hspace{0.1cm}\frac{\partial \phi}{\partial x_{j}}\,\,\mbox{as}\,\,\varepsilon\to 0.$$
For each  $j\in\left\{1,2,\cdots,d\right\}$, using continuity of $f_{j}:\mathbb{R}\to \mathbb{R}$ and maximum principle of sequence of solutions to generalized viscosity problem \eqref{ibvp.parab}, we observe that the integrand on LHS of  \eqref{weaksolution.eqn2} is pointwise bounded by 
$$\,\displaystyle\max_{1\leq j\leq d}\left\{\left|f_{j}(x)\right|\hspace{0.2cm}\big|\hspace{0.2cm} x\in I \right\} \left|\frac{\partial\phi}{\partial x_{j}}\right|$$
which is integrable as $\mbox{Vol}(\Omega_{T})< \infty$.\\
Applying dominated convergence theorem on LHS of \eqref{weaksolution.eqn2} , we have \eqref{weaksolution.eqn2}.\\
\vspace{0.2cm}\\
 \textbf{Step 1D:}$\left(\mbox{Passage to limit in  the RHS of \eqref{eqnchap502}}\right)$
 \vspace{0.2cm}\\
 In this step, we prove that
 \begin{equation}\label{weaksolution.eqn4}
 \displaystyle\lim_{\varepsilon\to 0}\int_{0}^{T}\int_{\Omega}u^{\varepsilon}(x,0)\hspace{0.1cm}\phi(x,0)\hspace{0.1cm}dx\hspace{0.1cm}dt= \int_{0}^{T}\int_{\Omega} u(x,0)\hspace{0.1cm}\phi(x,0)\hspace{0.1cm}dx\hspace{0.1cm}dt.
 \end{equation}
 For {\it a.e.} $x\in \Omega$, we have
 $$u^{\varepsilon}(x,0)\hspace{0.1cm}\phi(x,0)\to u^{\varepsilon}(x,0)\hspace{0.1cm}\phi(x,0)\,\,\mbox{as}\,\,\varepsilon\to 0.$$
 The integrand on LHS of \eqref{weaksolution.eqn4} is {\it a.e.} bounded by $\|u_{0}\|_{L^{\infty}(\Omega)}\left|\phi(x,0)\right|$ which is integrable as $\mbox{Vol}(\Omega_{T})<\infty$. \\
 Applying dominated convergence theorem on LHS of \eqref{weaksolution.eqn4}, we have \eqref{weaksolution.eqn4}.\\
 Using equations \eqref{weaksolution.eqn1},\eqref{eqnchap504},\eqref{weaksolution.eqn2},\eqref{weaksolution.eqn4} in \eqref{eqnchap502}, we have 
 \begin{equation}\label{eqnchap512}
 \int_{0}^{T}\hspace{0.1cm}\int_{\Omega}u\hspace{0.1cm}\frac{\partial \phi}{\partial t}\hspace{0.1cm}dx\hspace{0.1cm}dt + \displaystyle\sum_{j=1}^{d}\int_{0}^{T}\hspace{0.1cm}\int_{\Omega} f_{j}(u)\hspace{0.1cm}\frac{\partial \phi}{\partial x_{j}}\hspace{0.1cm}dx \hspace{0.1cm}dt =\int_{\Omega}u(x,0)\hspace{0.1cm}\phi(x,0)\hspace{0.1cm}dx\hspace{0.1cm}dt.
\end{equation}
In view of a result \cite[p.1026]{MR542510} that $u(x,t)=u_{0}(x)$ {\it a.e.} $x\in\Omega$, from \eqref{eqnchap512}, we have
\begin{equation}\nonumber\\
 \int_{0}^{T}\hspace{0.1cm}\int_{\Omega}u\hspace{0.1cm}\frac{\partial \phi}{\partial t}\hspace{0.1cm}dx\hspace{0.1cm}dt + \displaystyle\sum_{j=1}^{d}\int_{0}^{T}\hspace{0.1cm}\int_{\Omega} f_{j}(u)\hspace{0.1cm}\frac{\partial \phi}{\partial x_{j}}\hspace{0.1cm}dx \hspace{0.1cm}dt =\int_{\Omega}u_{0}(x)\hspace{0.1cm}\phi(x,0)\hspace{0.1cm}dx\hspace{0.1cm}dt.\nonumber\\
\end{equation}
Therefore $u$ is a weak solution of IBVP for conservation law \eqref{ivp.cl}.\\
\vspace{0.1cm}\\
\textbf{Step 2:} We want to show that weak solution $u$ of IBVP \eqref{ivp.cl} belongs to $BV(\Omega_T)$, and satisfies inequality \eqref{B.entropysolution.eqn1}. Since $u$ is the $L^1(\Omega_T)$ limit of a sequence of BV functions, it follows from \cite[p.1021]{MR542510} that $u\in BV(\Omega_T)$. Let $k\in\mathbb{R}$ and $\phi\in C^{2}(\overline{\Omega}\times(0,T))$ such that $\phi\geq 0$ and has compact support in $\overline{\Omega}\times (0,T)$. Multiplying the first equation of IBVP \eqref{ibvp.parab} by $sg_{n}(u^{\varepsilon}-k)\,\phi$ and integrating over $\Omega\times(0,T)$, we get
 \begin{eqnarray}\label{B.entropysolution.eqn2}
  \int_{0}^{T}\int_{\Omega}\frac{\partial u^{\varepsilon}}{\partial t}\,sg_{n}(u^{\varepsilon}-k)\,\phi\,dx\,dt + \displaystyle\sum_{j=1}^{d}\int_{0}^{T}\int_{\Omega}\,\frac{\partial}{\partial x_{j}}\left(f_{j}(u^{\varepsilon})\right)\,sg_{n}(u^{\varepsilon}-k)\,\phi\,dx\,dt \nonumber\\=\varepsilon\displaystyle\sum_{j=1}^{d}\int_{0}^{T}\int_{\Omega}\frac{\partial}{\partial x_{j}}\left(B(u^{\varepsilon})\,\frac{\partial u^{\varepsilon}}{\partial x_{j}}\right)\,sg_{n}(u^{\varepsilon}-k)\,\phi\,dx\,dt.
 \end{eqnarray}
Using integration by parts in \eqref{B.entropysolution.eqn2}, we arrive at
\begin{eqnarray}\label{B.entropysolution.eqn8}
 \int_{0}^{T}\int_{\Omega}\left\{\int_{k}^{u^{\varepsilon}}sg_{n}(y-k)\,dy\right\}\,\frac{\partial\phi}{\partial t}\,dx\,dt +\int_{0}^{T}\int_{\Omega}\,\left(f(u^{\varepsilon})-f(k)\right)\cdot\nabla\phi\,\,sg_{n}(u^{\varepsilon}-k)\,dx\,dt \nonumber\\ + \int_{0}^{T}\int_{\Omega}\,\left(f(u^{\varepsilon})-f(k)\right)\cdot\nabla u^{\varepsilon}\,sg_{n}^{'}(u^{\varepsilon}-k)\,\phi\,dx\,dt = \varepsilon\int_{0}^{T}\int_{\Omega}B(u^{\varepsilon})\,\left(\nabla u^{\varepsilon}\cdot\nabla\phi\right)\,sg_{n}(u^{\varepsilon}-k)\,dx\,dt\nonumber\\+ \varepsilon\int_{0}^{T}\int_{\Omega}B(u^{\varepsilon})\,\left|\nabla u^{\varepsilon}\right|^{2}\,sg_{n}^{'}(u^{\varepsilon}-k)\,\phi\,dx\,dt + \varepsilon\int_{0}^{T}\int_{\partial\Omega}\,B(0)\,\nabla u^{\varepsilon}\cdot\sigma\, sg_{n}(k)\phi\,d\sigma\,dt\nonumber\\ -\int_{0}^{T}\int_{\partial\Omega}\,\left(f(0)-f(k)\right)\cdot\sigma\,sg_{n}(k)\,\phi\,d\sigma\,dt.\nonumber\\
 {}
\end{eqnarray}
We want to prove the entropy inequality \eqref{B.entropysolution.eqn1} by passing to the limit in \eqref{B.entropysolution.eqn8} in two steps. In Step 2A, we pass to the limit in \eqref{B.entropysolution.eqn8} as $n\to\infty$ and then in Step 2B, we pass to the limit  as $\varepsilon\to 0$ in the resultant from Step 2A.
\vspace{0.2cm}\\
\textbf{Step 2A:}
\begin{enumerate}
 \item[(i)]{(Passage to the limit in the first term on LHS of \eqref{B.entropysolution.eqn8} as $n\to\infty$):}\\ 
 Note that, for {\it a.e.} $(x,t)\in\Omega_{T}$, we have 
 $$\left(\int_{k}^{u^{\varepsilon}}sg_{n}(y-k)\,dy\right)\frac{\partial\phi}{\partial t}\to \left(\int_{k}^{u^{\varepsilon}}sg(y-k)\,dy\right)\,\frac{\partial\phi}{\partial t}\,\,\mbox{as}\,\,n\to\infty.$$
 Observe that
 \begin{eqnarray}\label{B.entropysolution.eqn7a}
 \left|\left(\int_{k}^{u^{\varepsilon}}sg_{n}(y-k)\,dy\right)\frac{\partial\phi}{\partial t}\right|\leq |u^{\varepsilon}-k|\,\left|\frac{\partial\phi}{\partial t}\right|.
 \end{eqnarray}
 Since $\mbox{Vol}(\Omega_{T})<\infty$, the RHS of \eqref{B.entropysolution.eqn7a} is integrable. An application of dominated convergence theorem gives
 \begin{eqnarray}\label{B.entropysolution.eqn9}
  \displaystyle\lim_{n\to\infty}\int_{0}^{T}\int_{\Omega}\left(\int_{k}^{u^{\varepsilon}}sg_{n}(y-k)\,dy\right)\frac{\partial\phi}{\partial t}dx\,dt= \int_{0}^{T}\int_{\Omega}\left(\int_{k}^{u^{\varepsilon}}sg(y-k)\,dy\right)\,\frac{\partial\phi}{\partial t}\,\,dx\,dt.\nonumber\\
  {}
 \end{eqnarray}
 \item[(ii)]{(Passage to the limit in the second term on LHS of \eqref{B.entropysolution.eqn8} as $n\to\infty$):}
For {\it a.e.} $(x,t)\in\Omega_{T}$, we have 
$$\left(f(u^{\varepsilon})-f(k)\right)\cdot\nabla\phi\,\,sg_{n}(u^{\varepsilon}-k)\to \left(f(u^{\varepsilon})-f(k)\right)\cdot\nabla\phi\,\,sg(u^{\varepsilon}-k)\,\,\mbox{as}\,\,n\to\infty.$$
Note that 
\begin{eqnarray}\label{B.entropysolution.eqn10}
\left|\left(f(u^{\varepsilon})-f(k)\right)\cdot\nabla\phi\,\,sg_{n}(u^{\varepsilon}-k)\right|&\leq& d\Big[\displaystyle\max_{1\leq j\leq d}\left(\displaystyle\sup_{y\in I}\left|f_{j}(y)\right|\right) \nonumber\\
&&+ \displaystyle\max_{1\leq j\leq d}\left(\left|f_{j}(k)\right|\right)\Big]\displaystyle\max_{1\leq j\leq d}\left(\displaystyle\sup_{(x,t)\in\overline{\Omega_{T}}}\left|\frac{\partial\phi}{\partial x_{j}}\right|\right)\nonumber\\
{}
\end{eqnarray}
Since $\mbox{Vol}(\Omega_{T})<\infty$, we have the RHS of \eqref{B.entropysolution.eqn10} is integrable. An application of dominated convergence theorem gives
\begin{eqnarray}\label{B.entropysolution.eqn11}
 \displaystyle\lim_{n\to\infty}\int_{0}^{T}\left(f(u^{\varepsilon})-f(k)\right)\cdot\nabla\phi\,\,sg_{n}(u^{\varepsilon}-k)\,dx\,dt= \int_{0}^{T}\left(f(u^{\varepsilon})-f(k)\right)\cdot\nabla\phi\,\,sg(u^{\varepsilon}-k)\,dx\,dt.\nonumber\\
 {}
\end{eqnarray}
\item[(iii)]{(Passage to the limit in the third term on LHS of \eqref{B.entropysolution.eqn8} as $n\to\infty$):} For $j=1,2,\cdots,d$, applying mean value theorem to $f_{j}:\mathbb{R}\to\mathbb{R}$, we get
\begin{eqnarray}\label{B.entropysolution.eqn12}
 f_{j}(u^{\varepsilon})-f_{j}(k)= f_{j}^{'}(\xi^{j}_{u^{\varepsilon},k})(u^{\varepsilon}-k).
\end{eqnarray}
Denote
$$f^{'}(\xi_{u^{\varepsilon},k}):=\left(f_{1}^{'}(\xi^{1}_{u^{\varepsilon},k}),f_{2}^{'}(\xi^{2}_{u^{\varepsilon},k}),\cdots,f_{d}^{'}(\xi^{d}_{u^{\varepsilon},k})\right).$$
Applying \eqref{B.entropysolution.eqn12} and  $\left|(u^{\varepsilon}-k)sg_{n}^{'}(u^{\varepsilon}-k)\right|\leq 1$, we have 
\begin{eqnarray}\label{B.entropysolution.eqn13}
 \left|\int_{0}^{T}\int_{\Omega}\left(f(u^{\varepsilon})-f(k)\right)\cdot\nabla u^{\varepsilon}\,sg_{n}^{'}(u^{\varepsilon}-k)\,\phi\,dx\,dt\right|\nonumber\\
 \leq \int_{0}^{T}\int_{\left\{x\in\Omega\,:\,|u^{\varepsilon}-k|<\frac{1}{n}\right\}}\left|\left(f^{'}(\xi_{u^{\varepsilon},k})\right)\right|\,\left|\nabla \left(u^{\varepsilon}-k\right)\right|\,\left|\phi\right|dx\,dt,\nonumber\\
 \leq \|\phi\|_{L^{\infty}(\Omega_{T})}\,\sqrt{d}\displaystyle\max_{1\leq j\leq d}\left(\displaystyle\sup_{y\in I}\left|f_{j}^{'}(y)\right|\right)\,\int_{0}^{T}\int_{\left\{x\in\Omega\,:\,|u^{\varepsilon}-k|<\frac{1}{n}\right\}}\left|\nabla \left(u^{\varepsilon}-k\right)\right|\,dx\,dt.
\end{eqnarray}
Denote
$$g_{n}(t):=\int_{\left\{x\in\Omega\,:\,|u^{\varepsilon}-k|<\frac{1}{n}\right\}}\left|\nabla \left(u^{\varepsilon}-k\right)\right|\,dx.$$
Applying Lemma \ref{compactness12.lem1} with $v=\left(u^{\varepsilon}-k\right)$, we get
\begin{eqnarray}\label{B.entropysolution.eqn13A}
 \displaystyle\lim_{n\to\infty}g_{n}(t)=0.
\end{eqnarray}
For all $t\in (0,T)$ and $n\in\mathbb{N}$, we have $g_{n}(t)\leq\int_{\Omega}\left|\nabla \left(u^{\varepsilon}-k\right)\right|\,dx $. Since $u^{\varepsilon}\in C^{4+\beta,\frac{4+\beta}{2}}(\overline{\Omega_{T}})$, 
$$\int_{0}^{T}\int_{\Omega}\left|\nabla \left(u^{\varepsilon}-k\right)\right|\,dx\,dt<\infty.$$ 
An application of dominated convergence theorem yields
\begin{eqnarray}\label{B.entropysolution.eqn14}
 \displaystyle\lim_{n\to\infty}\left(\int_{0}^{T}\int_{\Omega}\left(f(u^{\varepsilon})-f(k)\right)\cdot\nabla u^{\varepsilon}\,sg_{n}^{'}(u^{\varepsilon}-k)\,\phi\,dx\,dt\right)&=&\int_{0}^{T}\displaystyle\lim_{n\to\infty}g_{n}(t)dt\nonumber\\
 &&=0
\end{eqnarray}
\vspace{0.1cm}\\
\item[(iv)]{(Passage to the limit in the first term on RHS of \eqref{B.entropysolution.eqn8} as $n\to\infty$):}
For {\it a.e.} $(x,t)\in \Omega_{T}$, 
$$B(u^{\varepsilon})\left(\nabla u^{\varepsilon}\cdot\nabla\phi\right)\,sg_{n}(u^{\varepsilon}-k)\to B(u^{\varepsilon})\left(\nabla u^{\varepsilon}\cdot\nabla\phi\right)\,sg(u^{\varepsilon}-k)\,\,\mbox{as}\,n\to\infty.$$
Note that 
\begin{eqnarray}\label{B.entropysolution.eqn16}
 \left|B(u^{\varepsilon})\left(\nabla u^{\varepsilon}\cdot\nabla\phi\right)\,sg_{n}(u^{\varepsilon}-k)\right|\leq \|B\|_{L^{\infty}(I)}\,\|\nabla u^{\varepsilon}\|_{\left(L^{\infty}(\Omega_{T})\right)^{d}}\,\|\nabla\phi\|_{\left(L^{\infty}(\Omega_{T})\right)^{d}}.
\end{eqnarray}
Since $\mbox{Vol}(\Omega_{T})<\infty$, the RHS of \eqref{B.entropysolution.eqn16} is integrable and 
applying dominated convergence theorem, we conclude
\begin{eqnarray}\label{B.entropysolution.eqn17}
 \displaystyle\lim_{n\to\infty}\int_{0}^{T}\int_{\Omega}B(u^{\varepsilon})\left(\nabla u^{\varepsilon}\cdot\nabla\phi\right)\,sg_{n}(u^{\varepsilon}-k)\,dx\,dt=\int_{0}^{T}\int_{\Omega}B(u^{\varepsilon})\left(\nabla u^{\varepsilon}\cdot\nabla\phi\right)\,sg(u^{\varepsilon}-k)\,dx\,dt.\nonumber\\
 {}
\end{eqnarray}
\item[(v)]{(Passage to the limit in the third term on RHS of \eqref{B.entropysolution.eqn8} as $n\to\infty$):} 
For {\it a.e.} $(x,t)\in\partial\Omega\times(0,T)$, we have 
$$B(0)\frac{\partial u^{\varepsilon}}{\partial\sigma}\,sg_{n}(k)\phi\to B(0)\frac{\partial u^{\varepsilon}}{\partial\sigma}\,sg(k)\phi\,\,\mbox{as}\,\,n\to\infty.$$
Since $\left|sg_{n}(k)\right|\leq 1$, we have
\begin{eqnarray}\label{eqnarray}\label{B.entropysolution.eqn18}
\left|B(0)\frac{\partial u^{\varepsilon}}{\partial\sigma}\,sg_{n}(k)\phi\right|\leq B(0)\displaystyle\max_{1\leq j\leq d}\left(\displaystyle\sup_{(x,t)\in\overline{\Omega_{T}}}\left|\frac{\partial u^{\varepsilon}}{\partial x_{j}}\right|\right)\left|\phi\right|,
\end{eqnarray}
the RHS of \eqref{B.entropysolution.eqn18} is integrable as $\mbox{surface measure}(\partial\Omega\times (0,T))<\infty.$ An application of dominated convergence theorem gives
\begin{eqnarray}\label{B.entropysolution.eqn19}
 \displaystyle\lim_{n\to\infty}\int_{0}^{T}\int_{\partial\Omega}B(0)\frac{\partial u^{\varepsilon}}{\partial\sigma}\,sg_{n}(k)\phi\,dx\,dt =\int_{0}^{T}\int_{\partial\Omega}B(0)\frac{\partial u^{\varepsilon}}{\partial\sigma}\,sg(k)\phi\,dx\,dt.
\end{eqnarray}
\item[(vi)] {(Passage to the limit in the fourth term on RHS of \eqref{B.entropysolution.eqn8} as $n\to\infty$):}
The proof of 
\begin{eqnarray}\label{B.entropysolution.eqn20}
 \displaystyle\lim_{n\to\infty}\int_{0}^{T}\int_{\partial\Omega}sg_{n}(k)\left(f(0)-f(k)\right)\cdot\sigma\,\phi\,d\sigma\,dt=\int_{0}^{T}\int_{\partial\Omega}sg(k)\left(f(0)-f(k)\right)\cdot\sigma\,\phi\,d\sigma\,dt.\nonumber\\
 {}
\end{eqnarray}
is similar to the proof of \eqref{B.entropysolution.eqn19}. 
\end{enumerate}
Using \eqref{B.entropysolution.eqn9},\eqref{B.entropysolution.eqn11},\eqref{B.entropysolution.eqn14},\eqref{B.entropysolution.eqn17},\eqref{B.entropysolution.eqn19},\eqref{B.entropysolution.eqn20} and 
$$\varepsilon\int_{0}^{T}\int_{\Omega}B(u^{\varepsilon})\,\left|\nabla u^{\varepsilon}\right|^{2}\,sg_{n}^{'}(u^{\varepsilon}-k)\,\phi\,dx\,dt\geq 0,$$
in \eqref{B.entropysolution.eqn8}, we get
\begin{eqnarray}\label{B.entropysolution.eqn22}
 \int_{0}^{T}\int_{\Omega}\left\{\int_{k}^{u^{\varepsilon}}sg(y-k)\,dy\right\}\,\frac{\partial\phi}{\partial t}\,dx\,dt +\int_{0}^{T}\int_{\Omega}\,\left(f(u^{\varepsilon})-f(k)\right)\cdot\nabla\phi\,\,sg(u^{\varepsilon}-k)\,dx\,dt \nonumber\\ \geq \varepsilon\int_{0}^{T}\int_{\Omega}B(u^{\varepsilon})\,\left(\nabla u^{\varepsilon}\cdot\nabla\phi\right)\,sg(u^{\varepsilon}-k)\,dx\,dt + \varepsilon\int_{0}^{T}\int_{\partial\Omega}\,B(0)\,\nabla u^{\varepsilon}\cdot\sigma\, sg^{'}(k)\phi\,d\sigma\,dt \nonumber\\-\int_{0}^{T}\int_{\partial\Omega}\,\left(f(0)-f(k)\right)\cdot\sigma\,sg(k)\,\phi\,dx\,dt.\nonumber\\
 {}
\end{eqnarray}
\vspace{0.2cm}\\
\textbf{Step 2B:} 
 \begin{enumerate}
 \item[(i)]{(Passage to the limit in the first term on LHS of \eqref{B.entropysolution.eqn22} as $\varepsilon\to 0$:)}
 Denote
 $$g(z):=\int_{0}^{z}sg(y-k)\,dy.$$
 Since $g$ is an absolutely continuous function, for {\it a.e.} $(x,t)\in\Omega_{T}$, we have 
 \begin{eqnarray}\label{B.entropysolution.eqn23}
  \displaystyle\lim_{\varepsilon\to 0}g(u^{\varepsilon})= g(u).\nonumber
 \end{eqnarray}
Therefore, as $\varepsilon\to 0$, for {\it a.e.} $(x,t)\in\Omega_{T}$, we obtain
\begin{eqnarray}\label{B.entropysolution.eqn24}
 g(u^{\varepsilon})\,\frac{\partial\phi}{\partial t}\to g(u)\,\frac{\partial\phi}{\partial t} \nonumber
\end{eqnarray}
Observe that
\begin{eqnarray}\label{B.entropysolution.eqn25}
 \left|\left(\int_{k}^{u^{\varepsilon}}sg(y-k)\,dy\right)\frac{\partial\phi}{\partial t}\right|\leq \left|u^{\varepsilon}-k\right|\,\left|\frac{\partial\phi}{\partial t}\right|\leq \left(\|u_{0}\|_{L^{\infty}(\Omega)}+|k|\right)\,\left|\frac{\partial\phi}{\partial t}\right|.
\end{eqnarray}
Since $\mbox{Vol}(\Omega_{T})<\infty$, the RHS of \eqref{B.entropysolution.eqn25} is integrable. Applying dominated convergence theorem, we get
\begin{eqnarray}\label{B.entropysolution.eqn26}
 \displaystyle\lim_{\varepsilon\to 0}\int_{0}^{T}\int_{\Omega}\left\{\int_{k}^{u^{\varepsilon}}sg(y-k)\,dy\right\}\,\frac{\partial\phi}{\partial t}\,dx\,dt &=& \int_{0}^{T}\int_{\Omega}\left|u-k\right|\frac{\partial\phi}{\partial t}\,dx\,dt\nonumber\\
 {}
\end{eqnarray}
\item[(ii)]{(Passage to the limit in the second term on LHS of \eqref{B.entropysolution.eqn22} as $\varepsilon\to 0$:)} We know that for {\it a.e.} $(x,t)\in\Omega_{T}$, 
$$sg(u^{\varepsilon}-k)\to sg(u-k)\,\,\mbox{as}\,\varepsilon\to 0.$$
Therefore for {\it a.e.} $(x,t)\in\Omega_{T}$, we have 
$$\left(\left(f(u^{\varepsilon})-f(k)\right)\cdot\nabla\phi\right)\,sg(u^{\varepsilon}-k)\to\left(\left(f(u)-f(k)\right)\cdot\nabla\phi\right)\,sg(u-k)\,\,\mbox{as}\,\varepsilon\to 0.$$
Observe that 
\begin{eqnarray}\label{B.entropysolution.eqn27}
 \left|\left(\left(f(u^{\varepsilon})-f(k)\right)\cdot\nabla\phi\right)\,sg(u^{\varepsilon}-k)\right| &\leq& d\left[\displaystyle\max_{1\leq j\leq d}\left(\displaystyle\sup_{y\in I}\left|f_{j}(y)\right|\right) + \displaystyle\max_{1\leq j\leq d}\left|f_{j}(k)\right|\right]\times\nonumber\\ &&\displaystyle\max_{1\leq j\leq d}\left(\displaystyle\sup_{(x,t)\in\overline{\Omega_{T}}}\left|\frac{\partial\phi}{\partial x_{j}}\right|\right).
\end{eqnarray}
Since $\mbox{Vol}(\Omega_{T})<\infty$, therefore the RHS of \eqref{B.entropysolution.eqn27} is integrable. An application of dominated convergence theorem yields
\begin{eqnarray}\label{B.entropysolution.eqn28}
 \displaystyle\lim_{\varepsilon\to 0}\int_{0}^{T}\int_{\Omega}\left(f(u^{\varepsilon})-f(k)\right)\cdot\nabla\phi\,\,sg(u^{\varepsilon}-k)\,dx\,dt=\int_{0}^{T}\int_{\Omega}\left(f(u)-f(k)\right)\cdot\nabla\phi\,\,sg(u-k)\,dx\,dt.\nonumber\\
 {}
\end{eqnarray}
\item[(iii)]{(Passage to the limit in the first term on RHS of \eqref{B.entropysolution.eqn22} as $\varepsilon\to 0$:)} Using $$\|\nabla u^{\varepsilon}\|_{\left(L^{1}(\Omega_{T})\right)^{d}}\leq \|\nabla u_{0}\|_{\left(L^{1}(\Omega_{T})\right)^{d}},$$ we observe that 
\begin{eqnarray}\label{B.entropysolution.eqn29}
 \left|\varepsilon\,\int_{0}^{T}\int_{\Omega}B(u^{\varepsilon})\,\nabla u^{\varepsilon}\cdot\nabla\phi\,\,sg(u^{\varepsilon}-k)\,dx\,dt\right|\hspace{2in}\nonumber\\
 \leq \varepsilon\,\|B\|_{L^{\infty}(I)}\displaystyle\max_{1\leq j\leq d}\left(\displaystyle\sup_{(x,t)\in \overline{\Omega_{T}}}\left|\frac{\partial\phi}{\partial x_{i}}\right|\right)\|\nabla u_{0}\|_{\left(L^{1}(\Omega_{T})\right)^{d}}.\nonumber\\
 {}
\end{eqnarray}
Passing to the limit as $\varepsilon\to 0$ in \eqref{B.entropysolution.eqn29}, we conclude
\begin{eqnarray}\label{B.entropysolution.eqn30}
 \displaystyle\lim_{\varepsilon\to 0} \varepsilon\,\int_{0}^{T}\int_{\Omega}B(u^{\varepsilon})\,\nabla u^{\varepsilon}\cdot\nabla\phi\,\,sg(u^{\varepsilon}-k)\,dx\,dt =0.
\end{eqnarray}
\item[(iv)]{(Passage to the limit in the second term on RHS of \eqref{B.entropysolution.eqn22} as $\varepsilon\to 0$:)} We want to compute 
$$\displaystyle\lim_{\varepsilon\to 0}\left\{\varepsilon\int_{0}^{T}\int_{\partial\Omega}B(0)\phi\,\frac{\partial u^{\varepsilon}}{\partial\sigma}\,d\sigma\,dt\right\}.$$
For $\delta > 0$, let $\rho_{\delta}\in C^{2}(\overline{\Omega})$ be functions introduced by Kruzhkov (see \cite{MR0255253}, \cite{MR542510}) having the properties 
 \begin{equation}\label{B.entropysolution.eqn31}\nonumber
\left.
\begin{aligned}
    \rho_{\delta} &= 1\,\,\mbox{on}\,\,\partial\Omega \quad\\ 
    \rho_{\delta} &=0\,\,\mbox{on}\,\,\left\{x\in\Omega\,\,:\,\,\mbox{dist}(x,\partial\Omega)\geq\delta\right\}\quad\\
    0 &\leq \rho_{\delta}\leq 1\,\,\mbox{on}\,\,\Omega\quad\\
    \left|\nabla\rho_{\delta}\right|&\leq\frac{C}{\delta},
\end{aligned}
\right\}
\end{equation}
where $C$ is independent of $\delta$.
Observe that 
\begin{eqnarray}\label{B.entropysolution.eqn31AB}
 \varepsilon\int_{0}^{T}\int_{\partial\Omega}B(0)\phi\,\frac{\partial u^{\varepsilon}}{\partial\sigma}\,d\sigma\,dt &=& \varepsilon\int_{0}^{T}\int_{\Omega}\nabla\left(B(u^{\varepsilon})\nabla u^{\varepsilon}\right)\,\,\phi\,\rho_{\delta}\,dx\,dt\nonumber\\
 && +\varepsilon\int_{0}^{T}\int_{\Omega}B(u^{\varepsilon})\nabla u^{\varepsilon}\cdot\nabla\left(\phi\,\rho_{\delta}\right)\,dx\,dt.
\end{eqnarray}
Using the equation  \eqref{ibvp.parab.a} and integration by parts in \eqref{B.entropysolution.eqn31AB}, we get
\begin{eqnarray}\label{B.entropysolution.eqn32}
 \varepsilon\int_{0}^{T}\int_{\partial\Omega}B(0)\phi\,\frac{\partial u^{\varepsilon}}{\partial\sigma}\,d\sigma\,dt\hspace{3in} \nonumber\\
 = -\int_{0}^{T}\int_{\Omega}\left\{u^{\varepsilon}\,\frac{\partial\phi}{\partial t} + f(u^{\varepsilon})\cdot\nabla\phi\right\}\,\rho_{\delta}\,dx\,dt -\int_{0}^{T}\int_{\Omega}\phi\,f(u^{\varepsilon})\cdot\nabla\rho_{\delta}\,dx\,dt\nonumber\\
 +\varepsilon\int_{0}^{T}\int_{\Omega}\left[B(u^{\varepsilon})\left\{\left(\nabla u^{\varepsilon}\cdot\nabla\phi\right)\rho_{\delta} +\left(\nabla u^{\varepsilon}\cdot\nabla\rho_{\delta}\right)\phi\right\}\right]dx\,dt\nonumber\\
 +\int_{0}^{T}\int_{\partial\Omega}\phi\,f(0)\cdot\sigma\,d\sigma\,dt.\nonumber
\end{eqnarray}
Using property $\left|\nabla\rho_{\delta}\right|\leq\frac{C}{\delta}$ and $\|\nabla u^{\varepsilon}\|_{\left(L^{1}(\Omega_{T})\right)^{d}}\leq \|\nabla u_{0}\|_{\left(L^{1}(\Omega_{T})\right)^{d}}$, we observe that
\begin{eqnarray}\label{B.entropysolution.eqn32A}
 \left|\varepsilon B(u^{\varepsilon})\left\{\left(\nabla u^{\varepsilon}\cdot\nabla\phi\right)\rho_{\delta} +\left(\nabla u^{\varepsilon}\cdot\nabla\rho_{\delta}\right)\phi\right\}\right|\hspace{2.2in}\nonumber\\ \leq\varepsilon\left(\displaystyle\max_{1\leq j\leq d}\left(\displaystyle\sup_{(x,t)\in\Omega_{T}}\left|\frac{\partial\phi}{\partial x_{j}}\right|\right) + \frac{C}{\delta}\|\phi\|_{L^{\infty}(\Omega_{T})}\right) \|B\|_{L^{\infty}(I)}\,\|\nabla u^{\varepsilon}\|_{\left(L^{1}(\Omega_{T})\right)^{d}}\nonumber\\
 \leq \varepsilon\left(\displaystyle\max_{1\leq j\leq d}\left(\displaystyle\sup_{(x,t)\in\Omega_{T}}\left|\frac{\partial\phi}{\partial x_{j}}\right|\right) + \frac{C}{\delta}\|\phi\|_{L^{\infty}(\Omega_{T})}\right)
 \|B\|_{L^{\infty}(I)}\, \|\nabla u_{0}\|_{\left(L^{1}(\Omega_{T})\right)^{d}}.
\end{eqnarray}
Passing to the limit as $\varepsilon\to 0$ in \eqref{B.entropysolution.eqn32A}, we get
\begin{eqnarray}\label{B.entropysolution.eqn33}
 \displaystyle\lim_{\varepsilon\to 0}\varepsilon\int_{0}^{T}\int_{\Omega} B(u^{\varepsilon})\left\{\left(\nabla u^{\varepsilon}\cdot\nabla\phi\right)\rho_{\delta} +\left(\nabla u^{\varepsilon}\cdot\nabla\rho_{\delta}\right)\phi\right\}\,dx\,dt=0.
\end{eqnarray}
Applying integration by parts in the second term on RHS of \eqref{B.entropysolution.eqn32}, we have
\begin{eqnarray}\label{B.entropysolution.eqn34}
 \displaystyle\lim_{\varepsilon\to 0}\left\{\varepsilon\int_{0}^{T}\int_{\partial\Omega}B(0)\phi\,\frac{\partial u^{\varepsilon}}{\partial\sigma}\,d\sigma\,dt\right\} &=& -\int_{0}^{T}\int_{\Omega}\left\{u\,\frac{\partial\phi}{\partial t} + f(u)\cdot\nabla\phi\right\}\,\rho_{\delta}\,dx\,dt\nonumber\\ &&-\int_{0}^{T}\int_{\Omega}\nabla\cdot\left(\phi\,f(u)\right)\rho_{\delta}\,dx\,dt \nonumber\\
&&+\int_{0}^{T}\int_{\partial\Omega}\phi\,\left(f(0)-f(\gamma(u))\right)\cdot\sigma\,d\sigma\,dt.\nonumber\\
{}
\end{eqnarray}
Passing to the limit as $\delta\to 0$ on RHS of \eqref{B.entropysolution.eqn34}, we get
\begin{eqnarray}\label{B.entropysolution.eqn35}
 \displaystyle\lim_{\varepsilon\to 0}\left\{\varepsilon\int_{0}^{T}\int_{\partial\Omega}B(0)\phi\,\frac{\partial u^{\varepsilon}}{\partial\sigma}\,d\sigma\,dt\right\}=\int_{0}^{T}\int_{\partial\Omega}\phi\,\left(f(0)-f(\gamma(u))\right)\cdot\sigma\,d\sigma\,dt.
\end{eqnarray}
Using \eqref{B.entropysolution.eqn26},\eqref{B.entropysolution.eqn28},\eqref{B.entropysolution.eqn30},\eqref{B.entropysolution.eqn35} in \eqref{B.entropysolution.eqn22}, we get the required entropy inequality \eqref{B.entropysolution.eqn1}.\\
\end{enumerate}

\section{Proof of Theorem~\ref{theorem2}}\label{section.prooftheorem2}

Denote by $f_{\varepsilon}$ and $u_{0\varepsilon}$, the regularizations of the flux function $f=(f_{1},f_{2},\cdots,f_{d})$ and initial condition $u_{0}$ of IBVP \eqref{ibvp.parab}, using the standard sequence of mollifiers $\tilde\rho_{\varepsilon}$ defined on $\R$, and $\rho_{\varepsilon}$ defined on $\R^d$, respectively. They are given by $f_{\varepsilon} := \left(f_{1\varepsilon},f_{2\varepsilon},\cdots,f_{d\varepsilon}\right)$ where
\begin{eqnarray*}\label{regularized.eqn2}
 f_{j\varepsilon}:= f_{j}\ast\tilde\rho_{\varepsilon} \, (j=1,2,\cdots,d),\,\,\,\mbox{and}\,\,
  u_{0\varepsilon} := u_{0}\ast\rho_{\varepsilon}.\nonumber
\end{eqnarray*}

Consider the IBVP for regularized generalized viscosity problem 
\begin{subequations}\label{regularized.IBVP}
\begin{eqnarray}
 u^\varepsilon_{t} + \nabla \cdot f_{\varepsilon}(u^{\varepsilon}) = \varepsilon\,\nabla\cdot\left(B(u^\varepsilon)\,\nabla u^\varepsilon\right)&\mbox{in }\Omega_{T},\label{regularized.IBVP.a} \\
    u^\varepsilon(x,t)= 0&\,\,\,\,\mbox{on}\,\, \partial \Omega\times(0,T),\label{regularized.IBVP.b}\\
u^{\varepsilon}(x,0) = u_{0\varepsilon}(x)& x\in \Omega,\label{regularized.IBVP.c}
\end{eqnarray}
\end{subequations}

\begin{lemma}\label{regularized.thm2}
Let $f, B, u_{0}$ satisfy Hypothesis B. Then there exists a unique solution of \eqref{regularized.IBVP} in $C^{4+\beta,\frac{4+\beta}{2}}(\overline{\Omega_{T}})$, for every $0<\beta<1$ and the following estimates hold:
  \begin{eqnarray}
  \|u^{\varepsilon}\|_{L^{\infty}(\Omega)}&\leq& \|u_{0}\|_{L^{\infty}(\Omega)},  \label{regularized.eqn19}\\[2mm]
   TV_{\Omega_T}(u^{\varepsilon})&\leq& T\, TV_\Omega(u_{0}). \label{regularized.eqn19aa}
  \end{eqnarray}
Also, there exists a constant $C> 0$ such that
\begin{eqnarray}\label{regularized.eqn26}
 \left\|\frac{\partial u^{\varepsilon}}{\partial t}\right\|_{L^{1}(\Omega_{T})}\leq CT
 \|B\|_{L^{\infty}(I)}\mbox{Vol}\left(\Omega\right)\, TV_{\Omega}(u_{0}) +T \mbox{Vol}\left(\Omega\right) \displaystyle\sum_{j=1}^{d}\|B^{\prime}\|_{L^{\infty}(I)}
 \left\|\frac{\partial u_{0}}{\partial x_{j}}\right\|_{L^{\infty}(\Omega)}^{2} \nonumber\\ +T\|f^{\prime}\|_{L^{\infty}(I)}\displaystyle\sum_{j=1}^{d} 
 \left\|\frac{\partial u_{0}}{\partial x_{j}}\right\|_{L^{\infty}(\Omega)}.\nonumber\\
 {}
\end{eqnarray}
  
Furthermore, there exists a subsequence $\left\{u^{\varepsilon_{k}}\right\}_{k=1}^{\infty}$ of $u^{\varepsilon}$ and a function $u$ in $L^{1}(\Omega_T)$
such that $ u^{\varepsilon_{k}}\to u$ {\it a.e.} in $\Omega_T$, and also in $L^{1}(\Omega_T)$ as $k\to\infty$.
\end{lemma}
The following result follows from \cite[p.67]{MR1304494} which is useful in proving Theorem \ref{regularized.thm2}, and we omit its proof. 
\begin{lemma}\label{regularized.lem1}
  Let $u_{0}\in W^{1,\infty}_{c}(\Omega)$. Then $u_{0\varepsilon}$ satisfies the following bounds
  \begin{eqnarray}
    \|u_{0\varepsilon}\|_{L^{\infty}(\Omega)}\leq \|u_{0}\|_{L^{\infty}(\Omega)}\label{regularized.max.eqn1}\\[2mm]
     \|\nabla u_{0\varepsilon}\|_{\left(L^{1}(\Omega)\right)^{d}}\leq TV_{\Omega}(u_{0}) \label{regularized.max.eqn1a}
  \end{eqnarray}
There exists a constant $C> 0$ such that for all $\varepsilon > 0$, $u_{0\varepsilon}$ satisfies 
  \begin{eqnarray}\label{regularized.max.eqn1b}
   \|\Delta u_{0\varepsilon}\|_{L^{1}(\Omega)}\leq \frac{C}{\varepsilon}TV_{\Omega}(u_{0}).
  \end{eqnarray}
\end{lemma}
\noindent{\bf Proof of Lemma \ref{regularized.thm2}}: Note that $f_{\varepsilon}\in\left(C^{\infty}(\mathbb{R})\right)^{d}$ and $\|f^{\prime}_{\varepsilon}\|_{\left(L^{\infty}(\Omega)\right)^{d}}\leq\|f^{\prime}\|_{\left(L^{\infty}(\Omega)\right)^{d}}<\infty.$
Since $u_{0}\in W^{1,\infty}_{c}(\Omega)$, the function $u_{0\varepsilon}$ belongs to the space $C^{\infty}(\overline{\Omega})$ and also has compact support in $\Omega$ for sufficiently small $\varepsilon$. As a consequence the initial-boundary data of the regularized generalized viscosity problem \eqref{regularized.IBVP} satisfies compatibility conditions of orders $0,1,2$ which are required to apply Theorem  \ref{chapHR85thm5}.  Applying Theorem  \ref{chapHR85thm5}, we get the existence of a unique solution $u^{\varepsilon}$ in $C^{4+\beta,\frac{4+\beta}{2}}(\overline{\Omega_{T}})$   for regularized generalized viscosity problem \eqref{regularized.IBVP}, and  $u^{\varepsilon}_{tt}\in C(\overline{\Omega_{T}})$. \\

By maximum principle (Theorem \ref{chap3thm1}), we conclude that $u^{\varepsilon}$ satisfies 
\begin{eqnarray}\label{regularized.eqn20}
 \|u^{\varepsilon}\|_{L^{\infty}(\Omega)}\leq \|u_{0\varepsilon}\|_{L^{\infty}(\Omega)}\,\,{\it a.e.}\,\,t\in (0,T).
\end{eqnarray}
Combining \eqref{regularized.max.eqn1} with \eqref{regularized.eqn20}, we get \eqref{regularized.eqn19}.\\
\vspace{0.2cm}\\
 Using equations \eqref{compactness12.eqn14} and \eqref{compactness12.eqn16} (from  {Step 1} in the proof of Theorem \ref{BVEstimate.thm1}) with $f=f_{\varepsilon}$ and $u_{0}=u_{0\varepsilon}$, we get
\begin{eqnarray}\label{regularized.eqn21aa}
 \int_{\Omega}\,\Big|\frac{\partial u^{\varepsilon}}{\partial t}(x,t)\Big|\,dx \leq \varepsilon
 \int_{\Omega}B(u_{0\varepsilon})\,\left|\Delta u_{0\varepsilon}\right|\,dx+ \int_{\Omega}\left|B^{'}(u_{0\varepsilon})\right|
 \left(\frac{\partial u_{0\varepsilon}}{\partial x_{j}}\right)^{2}\,dx +\displaystyle\sum_{j=1}^{d}\int_{\Omega} \left|f_{j\varepsilon}^{'}(u_{0\varepsilon})\right|\,\left|\frac{\partial u_{0\varepsilon}}{\partial x_{j}}\right|\,dx.\nonumber\\
 {}
\end{eqnarray}
Since $u_{0}\in W^{1,\infty}_{c}(\Omega)$, we have
\begin{subequations}\label{regularized.eqn22a}
 \begin{eqnarray}
\|u_{0\varepsilon}\|_{L^{\infty}(\Omega)}\leq \|u_{0}\|_{L^{\infty}(\Omega)},\hspace{1in}\\[2mm]
 \left\|\frac{\partial u_{0\varepsilon}}{\partial x_{j}}\right\|_{L^{\infty}(\Omega)}=\left\|\frac{\partial u_{0}}{\partial x_{j}}\ast\rho_{\varepsilon}\right\|_{L^{\infty}(\Omega)}
 \leq \left\|\frac{\partial u_{0}}{\partial x_{j}}\right\|_{L^{\infty}(\Omega)}.
\end{eqnarray}
\end{subequations}
In view of \eqref{regularized.eqn21aa} and \eqref{regularized.eqn22a}, we obtain 
\begin{eqnarray}\label{regularized.eqn21}
 \int_{\Omega}\,\Big|\frac{\partial u^{\varepsilon}}{\partial t}(x,t)\Big|\,dx \leq \varepsilon
 \|B\|_{L^{\infty}(I)}\int_{\Omega}\,\left|\Delta u_{0\varepsilon}\right|\,dx+\varepsilon\mbox{Vol}(\Omega) \displaystyle\sum_{j=1}^{d}\|B^{\prime}\|_{L^{\infty}(I)}
 \left\|\frac{\partial u_{0}}{\partial x_{j}}\right\|_{L^{\infty}(\Omega)}^{2}+ \nonumber\\ \|f^{\prime}_{\varepsilon}\|_{L^{\infty}(I)}\displaystyle\sum_{j=1}^{d} 
 \left\|\frac{\partial u_{0}}{\partial x_{j}}\right\|_{L^{\infty}(\Omega)}.\nonumber\\
 {}
\end{eqnarray}
We may assume that $\varepsilon <1$. Using \eqref{regularized.max.eqn1b} in \eqref{regularized.eqn21} gives
\begin{eqnarray}\label{regularized.eqn22}
 \int_{\Omega}\,\Big|\frac{\partial u^{\varepsilon}}{\partial t}(x,t)\Big|\,dx \leq C
 \|B\|_{L^{\infty}(I)}\mbox{Vol}\left(\Omega\right)\, TV_{\Omega}(u_{0}) + \mbox{Vol}\left(\Omega\right) \displaystyle\sum_{j=1}^{d}\|B^{\prime}\|_{L^{\infty}(I)}
 \left\|\frac{\partial u_{0}}{\partial x_{j}}\right\|_{L^{\infty}(\Omega)}^{2}+ \nonumber\\ \|f^{\prime}\|_{L^{\infty}(I)}\displaystyle\sum_{j=1}^{d} 
 \left\|\frac{\partial u_{0}}{\partial x_{j}}\right\|_{L^{\infty}(\Omega)}.\nonumber\\
 {}
\end{eqnarray}
Integrating on both sides of the last inequality w.r.t. $t$ on the interval $[0,T]$ yields \eqref{regularized.eqn26}.\\

Applying the conclusion of  {Step 2} in the proof of Theorem \ref{BVEstimate.thm1}, namely \eqref{BVestimateeqn24}, with $f=f_{\varepsilon}$ and $u_{0}=u_{0\varepsilon}$ yields
\begin{eqnarray}\label{regularized.eqn23}
 TV_{\Omega}(u^{\varepsilon})\leq TV_{\Omega}(u_{0\varepsilon}).\nonumber
\end{eqnarray}
Applying Lemma \ref{regularized.lem1}, we get
\begin{eqnarray}\label{regularized.eqn24}
 TV_{\Omega}(u^{\varepsilon})\leq TV_{\Omega}(u_{0})
\end{eqnarray}
Once again integrating w.r.t. $t$ over the interval $[0,T]$ on both sides of the inequality \eqref{regularized.eqn24} yields \eqref{regularized.eqn19aa}.\\

Using the compact embedding of $BV(\Omega_T)\cap L^1(\Omega_T)$ in $L^1(\Omega_T)$, we conclude that there
exists a subsequence $\left\{u^{\varepsilon_{k}}\right\}_{k=1}^{\infty}$ of $u^{\varepsilon}$ and a function $u$ in $L^{1}(\Omega_T)$
such that we have $u^{\varepsilon_{k}}\to u$ in $L^{1}(\Omega_T)$ as well as pointwise {\it a.e.} $\Omega_T$, as $k\to\infty$. This completes the proof of Lemma \ref{regularized.thm2}.

We still denote the subsequence $\left\{u^{\varepsilon_{k}}\right\}_{k=1}^{\infty}$ by $\left\{u^{\varepsilon}\right\}$.

We now show that the a.e. limit of a sequence of solutions to the regularized generalized viscosity problem is a weak solution for IBVP \eqref{ivp.cl}. \\

\noindent{\bf Proof of Theorem~\ref{theorem2}:} We will show that the function $u$ whose existence is asserted by Lemma \ref{regularized.thm2} is indeed an entropy solution to \eqref{ivp.cl}. For notational convenience, we still denote the subsequence $\left\{u^{\varepsilon_{k}}\right\}_{k=1}^{\infty}$ (as asserted by Lemma \ref{regularized.thm2}) by $\left\{u^{\varepsilon}\right\}$. The proof of Theorem~\ref{theorem2} is divided into two steps. 
In the first step, we show that $u$ is a weak solution for the IBVP \eqref{ivp.cl}. In the second step, we show that $u$ satisfies an entropy inequality for the IBVP \eqref{ivp.cl} in the sense of \eqref{B.entropysolution.eqn1}. Note that $u\in BV(\Omega_T)$ as it the $L^1(\Omega_T)$ limit of a sequence of BV functions  \cite[p.1021]{MR542510}. Thus it then follows that $u$ is an entropy solution.\\
 
 \textbf{Step 1:}  Let $\phi \in \mathcal{D}(\Omega\times [0, T))$. Multiplying the first equation of  \eqref{regularized.IBVP} by $\phi$, integrating over $\Omega_{T}$ and using integration by parts, we get
 \begin{equation}\label{regularized.eqn27}
\begin{split}
 \int_{0}^{T}\hspace{0.1cm}\int_{\Omega}u^{\varepsilon}\hspace{0.1cm}\frac{\partial \phi}{\partial t}\hspace{0.1cm}dx\hspace{0.1cm}dt + \varepsilon\hspace{0.1cm}\displaystyle\sum_{j=1}^{d}\int_{0}^{T}\hspace{0.1cm}\int_{\Omega}\hspace{0.1cm}B(u^{\varepsilon})\hspace{0.1cm}\frac{\partial u^{\varepsilon}}{\partial x_{j}}\hspace{0.1cm}\frac{\partial \phi}{\partial x_{j}}\hspace{0.1cm}dx\hspace{0.1cm}dt + \displaystyle\sum_{j=1}^{d}\int_{0}^{T}\hspace{0.1cm}\int_{\Omega}\hspace{0.1cm}f_{j\varepsilon}(u^{\varepsilon})\hspace{0.1cm}\frac{\partial \phi}{\partial x_{j}}\hspace{0.1cm}dx\hspace{0.1cm}dt\\ =\int_{\Omega}u^{\varepsilon}(x,0)\hspace{0.1cm}\phi(x,0)\hspace{0.1cm}dx\hspace{0.1cm}dt.
 \end{split}
\end{equation}
We would like to pass to the limit as $\varepsilon\to 0$ in the equation \eqref{regularized.eqn27}, and obtain
\begin{equation}\label{regularized.eqn34}
\begin{split}
 \int_{0}^{T}\hspace{0.1cm}\int_{\Omega}u\hspace{0.1cm}\frac{\partial \phi}{\partial t}\hspace{0.1cm}dx\hspace{0.1cm}dt +
 \displaystyle\sum_{j=1}^{d}\int_{0}^{T}\hspace{0.1cm}\int_{\Omega}\hspace{0.1cm}f_{j}(u)\hspace{0.1cm}\frac{\partial \phi}{\partial x_{j}}\hspace{0.1cm}dx\hspace{0.1cm}dt =\int_{\Omega}u_{0}(x)\hspace{0.1cm}\phi(x,0)\hspace{0.1cm}dx\hspace{0.1cm}dt,
 \end{split}
\end{equation}
thereby we conclude that $u$ is a weak solution for the IBVP  \eqref{ivp.cl}.
Note that we have
\begin{eqnarray}
  \displaystyle\lim_{\varepsilon\to 0}\int_{0}^{T}\hspace{0.1cm}\int_{\Omega}u^{\varepsilon}\hspace{0.1cm}\frac{\partial \phi}{\partial t}\hspace{0.1cm}dx\hspace{0.1cm}dt =\int_{0}^{T}\hspace{0.1cm}\int_{\Omega}u\hspace{0.1cm}\frac{\partial \phi}{\partial t}\hspace{0.1cm}dx\hspace{0.1cm}dt,\label{regularized.eqn28}\\
  \displaystyle\lim_{\varepsilon\to 0}\varepsilon\hspace{0.1cm}\displaystyle\sum_{j=1}^{d}\int_{0}^{T}\hspace{0.1cm}\int_{\Omega}\hspace{0.1cm}B(u^{\varepsilon})\hspace{0.1cm}\frac{\partial u^{\varepsilon}}{\partial x_{j}}\hspace{0.1cm}\frac{\partial \phi}{\partial x_{j}}\hspace{0.1cm}dx\hspace{0.1cm}dt = 0,\label{regularized.eqn28a}\\
  \displaystyle\lim_{\varepsilon\to 0}\int_{0}^{T}\int_{\Omega}u^{\varepsilon}(x,0)\hspace{0.1cm}\phi(x,0)\hspace{0.1cm}dx\hspace{0.1cm}dt= \int_{0}^{T}\int_{\Omega} u_{0}(x)\hspace{0.1cm}\phi(x,0)\hspace{0.1cm}dx\hspace{0.1cm}dt,\label{regularized.eqn29}
\end{eqnarray}
whose proofs follow on similar lines as those of \eqref{weaksolution.eqn1}, \eqref{eqnchap504} and \eqref{weaksolution.eqn4} respectively.\\

 For $j=1,2,\cdots,d$, we want to show that
\begin{eqnarray}\label{regularized.eqn30}
 \displaystyle\lim_{\varepsilon\to 0}\int_{0}^{T}\int_{\Omega}f_{j\varepsilon}(u^{\varepsilon})\,\frac{\partial\phi}{\partial x_{j}}\,dx= \int_{0}^{T}\int_{\Omega}f_{j}(u)\,\frac{\partial\phi}{\partial x_{j}}\,dx.
\end{eqnarray}
Note that
\begin{eqnarray}\label{regularized.eqn31a}
 \left|\int_{0}^{T}\int_{\Omega}\left(f_{j\varepsilon}(u^{\varepsilon})-f_{j}(u)\right)\,\frac{\partial\phi}{\partial x_{j}}\,dx\right|&\leq& \int_{0}^{T}\int_{\Omega}\left|f_{j\varepsilon}(u^{\varepsilon})-f_{j}(u)\right|\,\left|\frac{\partial\phi}{\partial x_{j}}\right|\,dx,\nonumber\\
 &\leq& \int_{0}^{T}\int_{\Omega}\left|f_{j\varepsilon}(u^{\varepsilon})-f_{j}(u^{\varepsilon})\right|\,\left|\frac{\partial\phi}{\partial x_{j}}\right|\,dx\,dt \nonumber\\
 &&+ \int_{0}^{T}\int_{\Omega}\left|f_{j}(u^{\varepsilon})-f_{j}(u)\right|\,\left|\frac{\partial\phi}{\partial x_{j}}\right|\,dx\,dt.\nonumber\\
 {}
\end{eqnarray}
For $j=1,2,\cdots,d$, $f_{j}$ is continuous, therefore $f_{j\varepsilon}\to f_{j}$ uniformly on compact sets as $\varepsilon\to 0$.
Observe that
\begin{eqnarray}\label{regularized.eqn31}
  \int_{0}^{T}\int_{\Omega}\left|f_{j\varepsilon}(u^{\varepsilon})-f_{j}(u^{\varepsilon})\right|\,\left|\frac{\partial\phi}{\partial x_{j}}\right|\,dx\leq \left\|f_{j\varepsilon}-f_{j}\right\|_{L^{\infty}(I)}\,\int_{0}^{T}\int_{\Omega}\left|\frac{\partial\phi}{\partial x_{j}}\right|\,dx\,dt.
\end{eqnarray}
Using sandwich theorem in \eqref{regularized.eqn31}, and dominated convergence theorem in the second term on RHS of the inequality 
\eqref{regularized.eqn31a}, we get \eqref{regularized.eqn30}. This completes the proof of \eqref{regularized.eqn34}.\\

\noindent\textbf{Step 2:} We want to show that  $u$ satisfies entropy inequality \eqref{B.entropysolution.eqn1}. Let $k\in\mathbb{R}$ and $\phi\in C^{2}(\overline{\Omega}\times(0,T))$ such that $\phi\geq 0$ and has compact support in $\overline{\Omega}\times (0,T)$. Multiplying the equation  \eqref{regularized.IBVP.a} by $sg_{n}(u^{\varepsilon}-k)\,\phi$ and integrating over $\Omega\times(0,T)$, we get
 \begin{eqnarray}\label{regularized.eqn35}
  \int_{0}^{T}\int_{\Omega}\frac{\partial u^{\varepsilon}}{\partial t}\,sg_{n}(u^{\varepsilon}-k)\,\phi\,dx\,dt + \displaystyle\sum_{j=1}^{d}\int_{0}^{T}\int_{\Omega}\,\frac{\partial}{\partial x_{j}}\left(f_{j\varepsilon}(u^{\varepsilon})\right)\,sg_{n}(u^{\varepsilon}-k)\,\phi\,dx\,dt \nonumber\\=\varepsilon\displaystyle\sum_{j=1}^{d}\int_{0}^{T}\int_{\Omega}\frac{\partial}{\partial x_{j}}\left(B(u^{\varepsilon})\,\frac{\partial u^{\varepsilon}}{\partial x_{j}}\right)\,sg_{n}(u^{\varepsilon}-k)\,\phi\,dx\,dt.
 \end{eqnarray}
 Integrating by parts in \eqref{regularized.eqn35} yields
 \begin{eqnarray}\label{regularized.eqn36}
 \int_{0}^{T}\int_{\Omega}\left\{\int_{k}^{u^{\varepsilon}}sg_{n}(y-k)\,dy\right\}\,\frac{\partial\phi}{\partial t}\,dx\,dt +\int_{0}^{T}\int_{\Omega}\,\left(f_{\varepsilon}(u^{\varepsilon})-f_{\varepsilon}(k)\right)\cdot\nabla\phi\,\,sg_{n}(u^{\varepsilon}-k)\,dx\,dt \nonumber\\ + \int_{0}^{T}\int_{\Omega}\,\left(f_{\varepsilon}(u^{\varepsilon})-f_{\varepsilon}(k)\right)\cdot\nabla u^{\varepsilon}\,sg_{n}^{'}(u^{\varepsilon}-k)\,\phi\,dx\,dt = \varepsilon\int_{0}^{T}\int_{\Omega}B(u^{\varepsilon})\,\left(\nabla u^{\varepsilon}\cdot\nabla\phi\right)\,sg_{n}(u^{\varepsilon}-k)\,dx\,dt\nonumber\\+ \varepsilon\int_{0}^{T}\int_{\Omega}B(u^{\varepsilon})\,\left|\nabla u^{\varepsilon}\right|^{2}\,sg_{n}^{'}(u^{\varepsilon}-k)\,\phi\,dx\,dt + \varepsilon\int_{0}^{T}\int_{\partial\Omega}\,B(0)\,\nabla u^{\varepsilon}\cdot\sigma\, sg_{n}(k)\phi\,d\sigma\,dt\nonumber\\ -\int_{0}^{T}\int_{\partial\Omega}\,\left(f_{\varepsilon}(0)-f_{\varepsilon}(k)\right)\cdot\sigma\,sg_{n}(k)\,\phi\,d\sigma\,dt.\nonumber\\
 {}
\end{eqnarray}
We prove the entropy inequality \eqref{B.entropysolution.eqn1} by passing to the limit in \eqref{regularized.eqn36}, first as $n\to\infty$ (in Step 2A), and then pass to the limit as $\varepsilon\to 0$ in the equation resulting from Step 2A (in Step 2B). Here we follow the arguments of \cite{MR542510}.\\

\noindent\textbf{Step 2A:} Using the conclusion of {Step 1} of the proof of Theorem \ref{theorem1} with $f=f_{\varepsilon}$ , we pass to the limit in \eqref{regularized.eqn36} as $n\to\infty$ and obtain the following inequality
\begin{eqnarray}\label{regularized.eqn37}
 \int_{0}^{T}\int_{\Omega}\left\{\int_{k}^{u^{\varepsilon}}sg(y-k)\,dy\right\}\,\frac{\partial\phi}{\partial t}\,dx\,dt +\int_{0}^{T}\int_{\Omega}\,\left(f_{\varepsilon}(u^{\varepsilon})-f_{\varepsilon}(k)\right)\cdot\nabla\phi\,\,sg(u^{\varepsilon}-k)\,dx\,dt \nonumber\\ \geq \varepsilon\int_{0}^{T}\int_{\Omega}B(u^{\varepsilon})\,\left(\nabla u^{\varepsilon}\cdot\nabla\phi\right)\,sg(u^{\varepsilon}-k)\,dx\,dt + \varepsilon\int_{0}^{T}\int_{\partial\Omega}\,B(0)\,\nabla u^{\varepsilon}\cdot\sigma\, sg^{'}(k)\phi\,d\sigma\,dt \nonumber\\-\int_{0}^{T}\int_{\partial\Omega}\,\left(f_{\varepsilon}(0)-f_{\varepsilon}(k)\right)\cdot\sigma\,sg(k)\,\phi\,dx\,dt.\nonumber\\
 {}
\end{eqnarray}
\noindent\textbf{Step 2B:} Following the proofs of \eqref{B.entropysolution.eqn26},\eqref{B.entropysolution.eqn30}, we get
 \begin{eqnarray}\label{regularized.eqn38}
 \displaystyle\lim_{\varepsilon\to 0}\int_{0}^{T}\int_{\Omega}\left\{\int_{k}^{u^{\varepsilon}}sg(y-k)\,dy\right\}\,\frac{\partial\phi}{\partial t}\,dx\,dt &=& \int_{0}^{T}\int_{\Omega}\left|u-k\right|\frac{\partial\phi}{\partial t}\,dx\,dt,
\end{eqnarray}
\begin{eqnarray}\label{regularized.eqn46}
 \displaystyle\lim_{\varepsilon\to 0} \varepsilon\,\int_{0}^{T}\int_{\Omega}B(u^{\varepsilon})\,\nabla u^{\varepsilon}\cdot\nabla\phi\,\,sg(u^{\varepsilon}-k)\,dx\,dt =0.
\end{eqnarray}

\noindent We pass to limit as $\varepsilon\to 0$ in the second term on LHS of \eqref{regularized.eqn37}, and prove
\begin{eqnarray}\label{regularized.eqn45}
 \displaystyle\lim_{\varepsilon\to 0}\int_{0}^{T}\int_{\Omega}\,\left(f_{\varepsilon}(u^{\varepsilon})-f_{\varepsilon}(k)\right)\cdot\nabla\phi\,\,sg(u^{\varepsilon}-k)\,dx\,dt= \int_{0}^{T}\int_{\Omega}\,\left(f(u)-f(k)\right)\cdot\nabla\phi\,\,sg(u-k)\,dx\,dt,\nonumber\\
 {}
\end{eqnarray}
by an application of dominated convergence theorem. Firstly, we show that integrands on LHS converge to the integrand on the RHS of the equation \eqref{regularized.eqn45}.
Observe that
\begin{eqnarray}\label{regularized.eqn40}
 \left|f_{j\varepsilon}(u^{\varepsilon})-f_{j}(u)\right|&\leq& \left|f_{j\varepsilon}(u^{\varepsilon})-f_{j}(u^{\varepsilon})\right| + \left|f_{j}(u^{\varepsilon})-f_{j}(u)\right|,\nonumber\\
 &\leq& \left\|f_{j\varepsilon}-f_{j}\right\|_{L^{\infty}(I)} + \left|f_{j}(u^{\varepsilon})-f_{j}(u)\right|.
\end{eqnarray}
For each $j\in\left\{1,2,\cdots,d\right\}$, $f_{j}\in C(\mathbb{R})$. Therefore $f_{j\varepsilon}\to f_{j}$ uniformly on compact sets of $\mathbb{R}$. Since for {\it a.e.} $(x,t)\in\Omega_{T}$, $u^{\varepsilon}(x,t)\in I$, we have 
\begin{eqnarray}\label{regularized.eqn39}
 \|f_{j\varepsilon}-f_{j}\|_{L^{\infty}(I)}\to 0\,\,\mbox{as}\,\,\varepsilon\to 0.
\end{eqnarray}
For each $j\in\left\{1,2,\cdots,d\right\}$, since for {\it a.e.} $(x,t)\in\Omega_{T}$, $u^{\varepsilon}\to u$ as $\varepsilon\to 0$, we have
\begin{eqnarray}\label{regularized.eqn41}
 f_{j}(u^{\varepsilon})\to f_{j}(u)\,\,\mbox{as}\,\,\varepsilon\to 0.
\end{eqnarray}
Using equations \eqref{regularized.eqn39} and \eqref{regularized.eqn41} in \eqref{regularized.eqn40}, we get
\begin{eqnarray}\label{regularized.eqn42}
 f_{j\varepsilon}(u^{\varepsilon})\to f_{j}(u)\,\,{\it a.e.}\,\,(x,t)\in\Omega_{T},\,\,\mbox{as}\,\,\varepsilon\to 0.
\end{eqnarray}
As $\varepsilon\to 0$, it is easy to observe that
\begin{eqnarray}
 sg(u^{\varepsilon}-k)\to sg(u-k)\,\,\mbox{{\it a.e.} in }\,\Omega_T,\label{regularized.eqn44}\\
 f_{j\varepsilon}(k)\to f_{j}(k). \label{regularized.eqn43}
\end{eqnarray}
Using the information from  \eqref{regularized.eqn42}, \eqref{regularized.eqn44},  and \eqref{regularized.eqn43}, we conclude that for {\it a.e.} $(x,t)\in\Omega_{T}$,
\begin{eqnarray}\label{regularized.eqn38A}
 \displaystyle\lim_{\varepsilon\to 0}\left(\left(f_{\varepsilon}(u^{\varepsilon})-f_{\varepsilon}(k)\right)\cdot\nabla\phi\,\,sg(u^{\varepsilon}-k)\right)= \left(f(u)-f(k)\right)\cdot\nabla\phi\,\,sg(u-k).\nonumber
\end{eqnarray}
Since $|f_{j\varepsilon}(u^{\varepsilon})|\leq \|f_{j}\|_{L^{\infty}(I)}$ and $|f_{j\varepsilon}(k)|\leq |f_{j}(k)|$, therefore the integrand on LHS of the equation \eqref{regularized.eqn45} is bounded by
$$\displaystyle\sum_{j=1}^{d}\left(\|f_{j}\|_{L^{\infty}(I)} + |f_{j}(k)|\right)\left|\frac{\partial\phi}{\partial x_{j}}\right|,$$
which is integrable over $\Omega_{T}$ as $\phi$ has compact support in $\overline{\Omega}\times(0,T)$. Applying dominated convergence theorem, we get \eqref{regularized.eqn45}.\\
 
Following the proof of \eqref{B.entropysolution.eqn35}, we get
\begin{eqnarray}\label{regularized.eqn47}
 \displaystyle\lim_{\varepsilon\to 0}\left\{\varepsilon\int_{0}^{T}\int_{\partial\Omega}B(0)\phi\,\frac{\partial u^{\varepsilon}}{\partial\sigma}\,d\sigma\,dt\right\}=\int_{0}^{T}\int_{\partial\Omega}\phi\,\left(f(0)-f(\gamma(u))\right)\cdot\sigma\,d\sigma\,dt.
\end{eqnarray}
An application of dominated convergence theorem gives
\begin{eqnarray}\label{regularized.eqn49}
 \displaystyle\lim_{\varepsilon\to 0}\int_{0}^{T}\int_{\partial\Omega}\,\left(f_{\varepsilon}(0)-f_{\varepsilon}(k)\right)\cdot\sigma\,sg(k)\,\phi\,dx\,dt=\int_{0}^{T}\int_{\partial\Omega}\,\left(f(0)-f(k)\right)\cdot\sigma\,sg(k)\,\phi\,dx\,dt.\nonumber\\
 {}
\end{eqnarray}
Using the information from \eqref{regularized.eqn38},\eqref{regularized.eqn46},\eqref{regularized.eqn45}, 
\eqref{regularized.eqn47}, and \eqref{regularized.eqn49} in \eqref{regularized.eqn37}, we get the required entropy 
inequality \eqref{B.entropysolution.eqn1}. We do not prove the uniqueness of the entropy solution as it was already 
established in \cite{MR542510}.
\section{Proof of Theorem \ref{theorem3}:}
Let $f_{\varepsilon}$ be as defined in the regularized generalized viscosity problem \eqref{regularized.IBVP} and 
$\left(u_{0\varepsilon}\right)$ be as given in Hypotheses C. We intoduce 
the IBVP for viscosity problem
\begin{subequations}\label{regularized.IBVP.Compact}
\begin{eqnarray}
 u^\varepsilon_{t} + \nabla \cdot f_{\varepsilon}(u^{\varepsilon}) = \varepsilon\,\nabla\cdot\left(B(u^\varepsilon)\,\nabla u^\varepsilon\right)
 &\mbox{in }\Omega_{T},\label{regularized.IBVP.Compact.a} \\
    u^\varepsilon(x,t)= 0&\,\,\,\,\mbox{on}\,\, \partial \Omega\times(0,T),\label{regularized.IBVP.Compact.b}\\
u^{\varepsilon}(x,0) = u_{0\varepsilon}(x)& x\in \Omega,\label{regularized.IBVP.Compact.c}
\end{eqnarray}
\end{subequations}
The next result deals with the existence of a uniformly bounded sequence $\left(u_{0\varepsilon}\right)$ in $\mathcal{D}(\Omega)$ with all the properties as mentioned in Hypothesis C. 
\begin{lemma}\label{paper2.initialdata.lemma1}
 Let $u_{0}\in H^{1}_{0}(\Omega)\cap C(\overline{\Omega})$. Then there exists a sequuence $\left(u_{0\varepsilon}\right)$ 
 in $\mathcal{D}(\Omega)$ such that the following properties hold.
 \begin{enumerate}
  \item \begin{eqnarray}\label{initialdata.h1.eqn1}
   u_{0\varepsilon}\to u_{0}\,\,\mbox{in}\,\,H^{1}(\Omega)\,\,\mbox{as}\,\,\varepsilon\to 0.
  \end{eqnarray}
\item For all $\varepsilon >0$, there exists a constant $A> 0$ such that 
\begin{eqnarray}\label{initialdata.h1.eqn2}
 \|u_{0\varepsilon}\|_{L^{\infty}(\Omega)}\leq A.
\end{eqnarray}
\item For all $\varepsilon > 0$ small enough, there exists a constant $C> 0$ such that 
\begin{eqnarray}\label{initialdata.h1.eqn8C}
 \left\|\Delta u_{0\varepsilon}\right\|_{L^{1}(\Omega)}\leq \frac{C}{\varepsilon}.
\end{eqnarray}
 \end{enumerate}
\end{lemma}
Denote
$$Q:=\left\{y\in\mathbb{R}^{d}:\,\,\left|y_{i}\right|< 1,\,i=1,2,\cdots,d\right\},\, Q^{+}:=\left\{y\in Q:\,\,y_{d}> 0\right\},\,\,
\Gamma:=\left\{y\in Q:\,\,y_{d}=0\right\}.$$
The following result is used to prove Lemma \ref{paper2.initialdata.lemma1}.
\begin{proposition}\cite[p.31]{MR2309679}\label{paper2.initialdata.proposition1}
 Let $u\in H^{1}(Q^{+})\cap C(\overline{Q^{+}})$ and $u=0$ near $\partial Q^{+}\setminus\Gamma$. If $u=0$
 on the bottom $\Gamma$, then $u\in H^{1}_{0}(Q^{+})$.
\end{proposition}
\textbf{Proof of Lemma \ref{paper2.initialdata.lemma1}:}\\
 \textbf{Step 1: } The proof of \eqref{initialdata.h1.eqn1}. \\Since $\partial\Omega$ is smooth, for a given point $x_{0}\in\partial\Omega$, there exists a neighbourhood 
 $U_{0}$ of $x_{0}$ and a smooth invertible mapping $\Psi_{0}: Q\to U_{0}$ such that 
 $$\Psi_{0}\left(Q^{+}\right)= U_{0}\cap\Omega,\,\, \Psi_{0}\left(\Gamma\right)= U_{0}\cap\partial\Omega,$$
 where $Q^{+},\,\Gamma$ as defined above.
 Let $\eta_{0}\in\mathcal{D}(U_{0})$. We consider $\eta_{0} u_{0}$ instead of $u_{0}$ in the neighbourhood $U_{0}$ of $x_{0}$ . 
 Define $v:Q^{+}\to\mathbb{R}$ by
 $$v(y):=\left(\eta u_{0}\right)\left(\Psi(y)\right).$$ 
 Since $\eta_{0}\in\mathcal{D}(U_{0})$, therefore $v\equiv 0$ in a neighbourhood of the upper boundary and the lateral boundary of
 $Q^{+}$. Since $u_{0}=0$ on $\partial\Omega$, for $y\in\Gamma$, we have $v(y)=0$. Therefore an application of Proposition 
 \ref{paper2.initialdata.proposition1}, we have $v\in H^{1}_{0}(Q^{+})$. Let $\tilde{v}$ be the extension 
 of $v$ to the whole $Q$ by setting $\tilde{v}=0$ in $Q\setminus Q^{+}$. Let $\tilde{\rho}_{\varepsilon}:\mathbb{R}\to\mathbb{R}$
 be the standard sequence of mollifiers.  Define $J_{\varepsilon}\tilde{v}(y): Q^{+}\to\mathbb{R}$ by
 \begin{eqnarray}\label{initialdata.h1.eqn3AA}
  J_{\varepsilon}\tilde{v}(y) :=\int_{Q}\tilde{\rho}_{\varepsilon}(y_{1}-z_{1})\,\tilde{\rho}_{\varepsilon}(y_{2}-z_{2})\,\cdots
 \tilde{\rho}_{\varepsilon}(y_{d-1}-z_{d-1})\tilde{\rho}_{\varepsilon}(y_{d}-z_{d}-2\varepsilon)\,\tilde{v}(z)\,dz.\nonumber\\
 {}
 \end{eqnarray}
 Observe that $J_{\varepsilon}\tilde{v}\to v$ in $H^{1}_{0}(Q^{+})$. For clarity, we only show that 
 $\left(J_{\varepsilon}\tilde{v}\right)$ has compact support. Since $\tilde{v}$ is zero in a neighbourhood of the upper 
 boundary and the lateral boundary of $Q^{+}$, we only show that $J_{\varepsilon}\tilde{v}$ is zero in $\left\{y\in Q^{+};\,\, 0<y_{d}<\varepsilon\right\}$. We know that 
 $\rho_{\varepsilon}(y_{d}-z_{d}-2\varepsilon)=0$ whenever $\left|y_{d}-z_{d}-2\varepsilon\right|\geq\varepsilon$. Let us 
 compute
\begin{eqnarray}\label{initialdata.h1.eqn3}
 \left|y_{d}-z_{d}-2\varepsilon\right|&=& \left|z_{d}+\varepsilon +\left(\varepsilon-y_{d}\right)\right|,\nonumber\\
 &\geq& z_{d} +\varepsilon >\varepsilon.
\end{eqnarray}
Therefore $\left(J_{\varepsilon}\tilde{v}\right)$ has compact support in $Q^{+}$. As a result, the function 
$\left(J_{\varepsilon}\tilde{v}(\Psi^{-1}_{0}(x))\right)$ belongs to $C^{1}_{0}(\Omega\cap U_{x_{0}})$ and 
$$J_{\varepsilon}\tilde{v}(\Psi^{-1}_{0}(x))\to \eta u_{0}\,\,\,\mbox{in}\,\,\, H^{1}(\Omega).$$
Since $\partial\Omega$ is compact, there exists $x_{1}, x_{2},\cdots,x_{N}\in\partial\Omega$ and $U_{1},U_{2},\cdots,U_{N}$
such that $\partial\Omega\subset\displaystyle\cup_{i=1}^{N}U_{i}$. Choose $U_{N+1}\subset\subset\Omega$ such that 
$$\overline{\Omega}\subset\displaystyle\cup_{i=1}^{N+1}U_{i}.$$
For $i=1,2,\cdots,N,N+1$, let $\eta_{i}$ be a partition of unity associated to $U_{i}$. For $i=1,2,\cdots,N$, let 
$\left(u^{\varepsilon}_{0i}\right)$ be the sequences corresponding to $U_{i},\,\eta_{i}$ obtained as above manner,{\it i.e.,}
$u_{0i}^{\varepsilon}=J_{\varepsilon}\tilde{v_{i}}(\Psi^{-1}_{i}(x))$. Let $\rho_{\varepsilon}:\mathbb{R}^{d}\to\mathbb{R}$ be 
sequence of mollifiers. Then $u_{0N+1}^{\varepsilon}:=\left(\eta_{N+1}u_{0}\right)\ast\rho_{\varepsilon}\to \eta_{N+1}u_{0}$ uniformly on 
$\overline{U_{N+1}}$ as $\varepsilon\to 0$ and $\left(\eta_{N+1}u_{0}\right)\ast\rho_{\varepsilon}\to \eta_{N+1}u_{0}$ in 
$H^{1}(\Omega)$.\\ Denote
$$u_{0\varepsilon}(x):=\displaystyle\sum_{i=1}^{N+1}u_{0i}^{\varepsilon}(x).$$
It is clear that $u_{0\varepsilon}\to u_{0}$ in $H^{1}(\Omega)$ and we obtain \eqref{initialdata.h1.eqn1}.\\
\textbf{Step 2: } The proof of \eqref{initialdata.h1.eqn2}.\\ Applying change of variable $\frac{y-z}{\varepsilon}=p$ in 
\eqref{initialdata.h1.eqn3AA}, we get
\begin{eqnarray}\label{initialdata.h1.eqn4}
 J_{\varepsilon}^{-}\tilde{v}(y) &:=&\int_{y_{1}-\varepsilon}^{y_{1}+\varepsilon}\int_{y_{2}-\varepsilon}^{y_{2}+\varepsilon}\cdots
 \int_{y_{d}-\varepsilon}^{y_{d}+\varepsilon}\frac{1}{\varepsilon^{d}}\tilde{\rho}(\frac{y_{1}-z_{1}}{\varepsilon})\cdots
 \tilde{\rho}(\frac{y_{d-1}-z_{d-1}}{\varepsilon})\tilde{\rho}(\frac{y_{d}-z_{d}-2\varepsilon}{\varepsilon})\,\tilde{v}(z)\,dz,\nonumber\\
 &=& \int_{-1}^{1}\int_{-1}^{1}\cdots\int_{-1}^{1}\int_{3}^{1}\tilde{\rho}(p_{1})\tilde{\rho}(p_{2})\cdots\tilde{\rho}(p_{d-1})\tilde{\rho}(p_{d}-2)\,\tilde{v}(x-\varepsilon p)\,(-1)^{d}\,dp
\end{eqnarray}
Taking modulus on both sides of \eqref{initialdata.h1.eqn4}, for $i\in\left\{1,2,\cdots,N\right\}$, we get
\begin{eqnarray}\label{initialdata.h1.eqn5}
 \left\|J_{\varepsilon}\tilde{v_{i}}(\Psi^{-1}_{i})\right\|_{L^{\infty}(\Omega)}\leq 2^{d}\left(\left\|\rho\right\|_{L^{\infty}(\mathbb{R})}\right)^{d}\,
 \|\eta_{i}\,u_{0}\|_{L^{\infty}(\Omega)}.
\end{eqnarray}
In view of \eqref{initialdata.h1.eqn5}, for $i=1,2,\cdots,N$, there exist constants $C_{i}>0$ such that 
\begin{eqnarray}\label{initialdata.h1.eqn6}
 \|u_{0i}^{\varepsilon}\|_{L^{\infty}(\Omega)}\leq C_{i}
\end{eqnarray}
Again, since $u_{0}\in C(\overline{\Omega})$, then $\left\|\eta_{0} u_{0}\ast\rho_{\epsilon}\right\|_{L^{\infty}(\Omega)}\leq
\left\|\eta_{0} u_{0}\right\|_{L^{\infty}(\Omega)}$ on $U_{N+1}$.\\
Taking \\$A=\max\left\{C_{1}, C_{2},\cdots,C_{d}, \left\|\eta_{0} u_{0}\right\|_{L^{\infty}(\Omega)}\right\}$, we get \eqref{initialdata.h1.eqn2}.\\
\textbf{Step 3: } The proof \eqref{initialdata.h1.eqn8C}. \\
For $i\in\left\{1,2,\cdots,N\right\}$, denote
$$\Psi_{i}^{-1}(x):=\left(\Psi_{i1}^{-1}(x),\Psi_{i2}^{-1}(x),\cdots,\Psi_{id}^{-1}(x)\right).$$
We recall $\tilde{v_{i}}$ as given in Step 1, {\it i.e.}, for $i\in\left\{1,2,\cdots,N\right\}$, $\tilde{v_{i}}$ are given by
$\left(\eta_{i}u_{0}\right)\left(\Psi_{i}(y)\right)$ on $Q^{+}$ and zero in $Q\setminus Q^{+}$.  Then 
$u_{0i}^{\varepsilon}(x)= J_{\varepsilon}\tilde{v_{i}}\left(\Psi_{i}^{-1}(x)\right)$.  Denote 
$h_{\varepsilon}:\mathbb{R}\to\mathbb{R}$ by
$$h_{\varepsilon}(p):=\tilde{\rho}(p_{1})\tilde{\rho}(p_{2})\cdots\tilde{\rho}(p_{d-1})\tilde{\rho}(p_{d}-2\varepsilon).$$
Denote $Q_{\varepsilon}^{+}:=\left\{y\in Q^{+};\,\,\mbox{dist}\left(y,\,\partial Q^{+}\right)>\varepsilon\right\}$. Then for 
$y\in Q_{\varepsilon}^{+}$, observe that
\begin{eqnarray}\label{initialdata.h1.eqn9}
 J_{\varepsilon}\tilde{v_{i}}(y)&:=&\left(h_{\varepsilon}\ast\tilde{v_{i}}\right)(y).
\end{eqnarray}
Since $\mbox{supp}(J_{\varepsilon}\tilde{v_{i}})$ is contained in $Q_{\varepsilon}^{+}$, therefore we consider 
$J_{\varepsilon}\tilde{v_{i}}(y)$ only on $Q_{\varepsilon}^{+}$. In $Q_{\varepsilon}^{+}$, we have
\begin{eqnarray}\label{initialdata.h1.eqn10}
 \frac{\partial}{\partial x_{l}}J_{\varepsilon}\tilde{v_{i}}(y) &=& \frac{\partial}{\partial x_{l}}\left[\left(h_{\varepsilon}\ast
 \tilde{v_{i}}\right)\left(\Psi^{-1}_{i}(x)\right)\right],\nonumber\\
 &=&\displaystyle\sum_{j=1}^{d}\left(h_{\varepsilon}\ast\frac{\partial\tilde{v_{i}}}{\partial y_{j}}\right)\left(\Psi^{-1}_{i}(x)\right)\,\,\frac{\partial\Psi^{-1}_{ij}(x)}{\partial x_{l}}.
 \end{eqnarray}
Differentiating \eqref{initialdata.h1.eqn10} with respect to $x_{l}$, we get
\begin{eqnarray}\label{initialdata.h1.eqn11}
 \frac{\partial^{2}}{\partial x_{l}^{2}}J_{\varepsilon}\tilde{v_{i}}(\Psi^{-1}_{i}(x))&=&\displaystyle\sum_{j=1}^{d}\frac{\partial}{\partial x_{l}}
 \left[\left(h_{\varepsilon}\ast\frac{\partial\tilde{v_{i}}}{\partial y_{j}}(\Psi^{-1}_{i}(x))\right)\,\frac{\partial\Psi^{-1}_{ij}(x)}{\partial x_{l}}\right],\nonumber\\
 &=&\displaystyle\sum_{j,k=1}^{d}\frac{\partial}{\partial y_{k}}\left(h_{\varepsilon}\ast\frac{\partial\tilde{v_{i}}}{\partial y_{j}}\right)(\Psi^{-1}_{i}(x))
 \,\frac{\partial\Psi_{ik}^{-1}(x)}{\partial x_{l}}\,\frac{\partial\Psi_{ij}^{-1}(x)}{\partial x_{l}}\nonumber\\
 &&+\displaystyle\sum_{j=1}^{d}\left(h_{\varepsilon}\ast\frac{\partial\tilde{v_{i}}}{\partial y_{j}}\right)(\Psi^{-1}_{i}(x))\,\frac{\partial^{2}}{\partial x_{l}^{2}}\left(\Psi_{ij}^{-1}(x)\right),\nonumber\\
&=&\displaystyle\sum_{j,k=1}^{d}\left(\frac{\partial h_{\varepsilon}}{\partial y_{k}}\ast\frac{\partial\tilde{v_{i}}}{\partial y_{j}}\right)(\Psi^{-1}_{i}(x))
 \,\frac{\partial\Psi_{ik}^{-1}(x)}{\partial x_{l}}\,\frac{\partial\Psi_{ij}^{-1}(x)}{\partial x_{l}}\nonumber\\
 &&+\displaystyle\sum_{j=1}^{d}\left(h_{\varepsilon}\ast\frac{\partial\tilde{v_{i}}}{\partial y_{j}}\right)(\Psi^{-1}_{i}(x))\,\frac{\partial^{2}}{\partial x_{l}^{2}}\left(\Psi_{ij}^{-1}(x)\right).\nonumber\\
 {}
 \end{eqnarray}
Summing over $l=1,2,\cdots,d$, we get
\begin{eqnarray}\label{initialdata.h1.eqn12}
 \Delta_{x}J_{\varepsilon}\tilde{v_{i}}(\Psi^{-1}_{i}(x))&=&\displaystyle\sum_{j,k,l=1}^{d}\left(\frac{\partial h_{\varepsilon}}{\partial y_{k}}\ast\frac{\partial\tilde{v_{i}}}{\partial y_{j}}\right)(\Psi^{-1}_{i}(x))
 \,\frac{\partial\Psi_{ik}^{-1}(x)}{\partial x_{l}}\,\frac{\partial\Psi_{ij}^{-1}(x)}{\partial x_{l}}\nonumber\\
 &&+\displaystyle\sum_{j,l=1}^{d}\left(h_{\varepsilon}\ast\frac{\partial\tilde{v_{i}}}{\partial y_{j}}\right)(y)\,
 \frac{\partial^{2}}{\partial x_{l}^{2}}\left(\Psi_{ij}^{-1}(x)\right).\nonumber\\
 {}
\end{eqnarray}
For $y\in Q^{+}$, We know 
\begin{eqnarray}\label{initialdata.h1.eqn13}
 \left(h_{\varepsilon}\ast\frac{\partial\tilde{v_{i}}}{\partial y_{j}}\right)(y)=\int_{Q}h_{\varepsilon}(y-z)\,
 \frac{\partial\tilde{v_{i}}}{\partial y_{j}}(z)\,dz.
\end{eqnarray}
Using the change of variable $\frac{y-z}{\varepsilon}=r$ in \eqref{initialdata.h1.eqn13}, we get
\begin{eqnarray}\label{initialdata.h1.eqn14}
 \left(h_{\varepsilon}\ast\frac{\partial\tilde{v_{i}}}{\partial y_{j}}\right)(y)=\int_{-1}^{1}\int_{-1}^{1}\cdots\int_{-1}^{1}
 \int_{1}^{3}\tilde{\rho}(r_{1})\tilde{\rho}(r_{2})\cdots\tilde{\rho_{1}}(r_{d-1})\tilde{\rho}(r_{d}-2)\,(-1)^{d}\,
 \frac{\partial\tilde{v_{i}}}{\partial y_{j}}(y-\varepsilon r)\,dr.\nonumber\\
 {}
\end{eqnarray}
Taking modulus on both sides of \eqref{initialdata.h1.eqn14}, we have
\begin{eqnarray}\label{initialdata.h1.eqn15}
 \left|\left(h_{\varepsilon}\ast\frac{\partial\tilde{v_{i}}}{\partial y_{j}}\right)(y)\right|&\leq& \|\tilde{\rho}\|_{L^{\infty}(\mathbb{R})}^{d}
 \int_{-1}^{1}\int_{-1}^{1}\cdots\int_{-1}^{1}\int_{1}^{3}\left|\frac{\partial\tilde{v_{i}}}{\partial y_{j}}(y-\varepsilon r)\right|\,dr,\nonumber\\
 \int_{Q^{+}}\left|\left(h_{\varepsilon}\ast\frac{\partial\tilde{v_{i}}}{\partial y_{j}}\right)(y)\right|&\leq&\|\tilde{\rho}\|_{L^{\infty}(\mathbb{R})}^{d}
 \int_{-1}^{1}\int_{-1}^{1}\cdots\int_{-1}^{1}\left(\int_{1}^{3}\int_{Q^{+}}\left|\frac{\partial\tilde{v_{i}}}{\partial y_{j}}(y-\varepsilon r)\right|\,dy\right)dr,\nonumber\\
 &\leq& \|\tilde{\rho}\|_{L^{\infty}(\mathbb{R})}^{d}
 \int_{-1}^{1}\int_{-1}^{1}\cdots\int_{-1}^{1}\int_{1}^{3}\left(\int_{Q^{+}}\left|\frac{\partial\tilde{v_{i}}}{\partial y_{j}}(y)\right|\,dy\right)dr.
\end{eqnarray}
Since $\mbox{Vol}(Q)<\infty,\,\tilde{v_{i}}\in H^{1}(Q)$, therefore the RHS of \eqref{initialdata.h1.eqn15} is bounded by 
$2^{d}\|\tilde{\rho}\|_{L^{\infty}(\mathbb{R})}^{d}\int_{Q}\left|\frac{\partial\tilde{v_{i}}(y)}{\partial y_{j}}\right|\,dy$, which
is independent of $\varepsilon$. Since $\left|\left(h_{\varepsilon}\ast\frac{\partial\tilde{v_{i}}}{\partial y_{j}}\right)(y)\right|$
is bounded on $Q^{+}$, therefore $\left|\left(h_{\varepsilon}\ast\frac{\partial\tilde{v_{i}}}{\partial y_{j}}\right)(\Psi_{i}^{-1}(x))\right|$ is 
bounded on $U_{i}\cap\Omega$ by the same constant.\\
Next for $y\in Q^{+}$, we compute 
\begin{eqnarray}\label{initialdata.h1.eqn16}
 \left(\frac{\partial h_{\varepsilon}}{\partial y_{k}}\ast\frac{\partial\tilde{v_{i}}}{\partial y_{j}}\right)(y) &=&
 \int_{Q^{+}}\frac{\partial h_{\varepsilon}}{\partial y_{k}}(y-k)\,\frac{\partial\tilde{v_{i}}}{\partial y_{j}}(z)\,dz,\nonumber\\
 &=& \int_{Q}\tilde{\rho_{\varepsilon}}(y_{1}-z_{1})\tilde{\rho_{\varepsilon}}(y_{2}-z_{2})\cdots\frac{\partial}{\partial y_{k}}\left(\tilde{\rho_{\varepsilon}}(y_{k}-z_{k})\right)
\cdots\tilde{\rho_{\varepsilon}}(y_{d-1}-z_{d-1})\nonumber\\
&&\tilde{\rho_{\varepsilon}}(y_{d}-z_{d}-2\varepsilon)\,\frac{\partial\tilde{v_{i}}}{\partial y_{j}}(z)\,dz.
\end{eqnarray}
Using change of variable $\frac{y-z}{\varepsilon}=r$ and taking modulus on both sides of \eqref{initialdata.h1.eqn16}, we get
\begin{eqnarray}\label{initialdata.h1.eqn17}
 \left|\left(\frac{\partial h_{\varepsilon}}{\partial y_{k}}\ast\frac{\partial\tilde{v_{i}}}{\partial y_{j}}\right)(y)\right|&\leq&
 \frac{1}{\varepsilon}\|\tilde{\rho}\|_{L^{\infty}(\mathbb{R})}^{d-1}\,\|\tilde{\rho}^{\prime}\|_{L^{\infty}(\mathbb{R})}\int_{-1}^{1}
 \int_{-1}^{1}\cdots\int_{-1}^{1}\int_{1}^{3}\left|\frac{\partial\tilde{v_{i}}}{\partial y_{j}}(y-\varepsilon r)\right|\,dr.\nonumber\\
 {}
\end{eqnarray}
Integrating over $Q^{+}$, we arrive at
\begin{eqnarray}\label{initialdata.h1.eqn18}
 \int_{Q^{+}}\left|\left(\frac{\partial h_{\varepsilon}}{\partial y_{k}}\ast\frac{\partial\tilde{v_{i}}}{\partial y_{j}}\right)(y)\right|\,dy
 &\leq&\frac{2^{d}}{\varepsilon}\|\tilde{\rho}\|_{L^{\infty}(\mathbb{R})}^{d-1}\,\|\tilde{\rho}^{\prime}\|_{L^{\infty}(\mathbb{R})}
\int_{Q}\left|\frac{\partial\tilde{v_{i}}}{\partial y_{j}}(y)\right|\,dy.
 \end{eqnarray}
Therefore we have 
\begin{eqnarray}\label{initialdata.h1.eqn18A}
 \int_{\Omega}\left|\Delta_{x} J_{\varepsilon}^{-}\tilde{v_{i}}(\Psi^{-1}_{i}(x))\right|\,dx &\leq& \displaystyle\sum_{j,k,l=1}^{d}
 \left\|\frac{\partial\Psi_{ik}^{-1}}{\partial x_{l}}\right\|_{L^{\infty}(\Omega\cap U_{i})}\,
 \left\|\frac{\partial\Psi^{-1}_{ij}}{\partial x_{l}}\right\|_{L^{\infty}(\Omega\cap U_{i})}\,
 \int_{\Omega}\left|\left(\frac{\partial h_{\varepsilon}}{\partial y_{k}}\ast\frac{\partial\tilde{v_{i}}}{\partial y_{j}}\right)(\Psi^{-1}_{i}(x))\right|\,dx
 \nonumber\\ && +\displaystyle\sum_{j,l=1}^{d}\left\|\frac{\partial^{2}\Psi_{ij}^{-1}}{\partial x_{l}^{2}}\right\|_{L^{\infty}(\Omega\cap U_{i})}
\int_{\Omega} \left|\left(\frac{\partial h_{\varepsilon}}{\partial y_{k}}\ast\frac{\partial\tilde{v_{i}}}{\partial y_{j}}\right)(\Psi^{-1}_{i}(x))\right|\,dx
,\nonumber\\
 &=& \displaystyle\sum_{j,k,l=1}^{d}
 \left\|\frac{\partial\Psi_{ik}^{-1}}{\partial x_{l}}\right\|_{L^{\infty}(\Omega\cap U_{i})}\,
 \left\|\frac{\partial\Psi^{-1}_{ij}}{\partial x_{l}}\right\|_{L^{\infty}(\Omega\cap U_{i})}\,
 \int_{\Omega\cap U_{i}}\left|\left(\frac{\partial h_{\varepsilon}}{\partial y_{k}}\ast\frac{\partial\tilde{v_{i}}}{\partial y_{j}}\right)(\Psi^{-1}_{i}(x))\right|\,dx
 \nonumber\\ && +\displaystyle\sum_{j,l=1}^{d}\left\|\frac{\partial^{2}\Psi_{ij}^{-1}}{\partial x_{l}^{2}}\right\|_{L^{\infty}(\Omega\cap U_{i})}
\int_{\Omega\cap U_{i}} \left|\left(\frac{\partial h_{\varepsilon}}{\partial y_{k}}\ast\frac{\partial\tilde{v_{i}}}{\partial y_{j}}\right)(\Psi^{-1}_{i}(x))\right|\,dx
 ,\nonumber\\
 &=& \displaystyle\sum_{j,k,l=1}^{d}
\left\|\frac{\partial\Psi_{ik}^{-1}}{\partial x_{l}}\right\|_{L^{\infty}(\Omega\cap U_{i})}\,
 \left\|\frac{\partial\Psi^{-1}_{ij}}{\partial x_{l}}\right\|_{L^{\infty}(\Omega\cap U_{i})}\,
 \int_{Q^{+}}\left|\left(\frac{\partial h_{\varepsilon}}{\partial y_{k}}\ast\frac{\partial\tilde{v_{i}}}{\partial y_{j}}\right)(y)\right|\,dy
 \nonumber\\ && +\displaystyle\sum_{j,l=1}^{d}\left\|\frac{\partial^{2}\Psi_{ij}^{-1}}{\partial x_{l}^{2}}\right\|_{L^{\infty}(\Omega\cap U_{i})}
\int_{Q^{+}} \left|\left(\frac{\partial h_{\varepsilon}}{\partial y_{k}}\ast\frac{\partial\tilde{v_{i}}}{\partial y_{j}}\right)(y)\right|\,dy
 ,\nonumber\\
&\leq& \displaystyle\sum_{j,k,l=1}^{d}\Big(2^{d}\|\tilde{\rho}\|_{L^{\infty}(\mathbb{R})}^{d}
\left\|\frac{\partial\Psi_{ik}^{-1}}{\partial x_{l}}\right\|_{L^{\infty}(\Omega\cap U_{i})}\,
 \left\|\frac{\partial\Psi^{-1}_{ij}}{\partial x_{l}}\right\|_{L^{\infty}(\Omega\cap U_{i})}
 \nonumber\\ && + \displaystyle\sum_{j,l=1}^{d}\frac{2^{d}}{\varepsilon}\|\tilde{\rho}\|_{L^{\infty}(\mathbb{R})}^{d-1}\,
 \|\tilde{\rho}^{\prime}\|_{L^{\infty}(\mathbb{R})}\left\|\frac{\partial^{2}\Psi_{ij}^{-1}}{\partial x_{l}^{2}}\right\|_{L^{\infty}(\Omega\cap U_{i})}
 \Big)\int_{Q}\left|\frac{\partial\tilde{v_{i}}}{\partial y_{j}}(y)\right|\,dy.
\end{eqnarray}
For $i\in\left\{1,2,\cdots,N\right\}$, denote
\begin{eqnarray}\label{initialdata.h1.eqn19}\nonumber\\
C_{i}:= \displaystyle\sum_{j,k,l=1}^{d}\Big(2^{d}\|\tilde{\rho}\|_{L^{\infty}(\mathbb{R})}^{d}\left\|\frac{\partial\Psi_{ik}^{-1}}{\partial x_{l}}\right\|_{L^{\infty}(\Omega\cap U_{i})}\,
 \left\|\frac{\partial\Psi^{-1}_{ij}}{\partial x_{l}}\right\|_{L^{\infty}(\Omega\cap U_{i})} +\displaystyle\sum_{j,l=1}^{d} 2^{d}\|\tilde{\rho}\|_{L^{\infty}(\mathbb{R})}^{d-1}\,
 \|\tilde{\rho}^{\prime}\|_{L^{\infty}(\mathbb{R})}\times
 \nonumber\\ \left\|\frac{\partial^{2}\Psi_{ij}^{-1}}{\partial x_{l}^{2}}\right\|_{L^{\infty}(\Omega\cap U_{i})}
 \Big)\int_{Q}\left|\frac{\partial\tilde{v_{i}}}{\partial y_{j}}(y)\right|\,dy\nonumber
\end{eqnarray}
Since $\eta_{N+1}\in\mathcal{D}(U_{N+1})$, $\eta_{N+1}u_{0}\in H^{1}_{0}(\Omega)$. Let $0<\varepsilon <\frac{\mbox{dist}(\mbox{supp}(\eta_{N+1}u_{0}),\partial\Omega)}{2}$ 
be small enough. For $l\in\left\{1,2,\cdots,d\right\}$, on $\Omega_{\varepsilon}$, we compute
\begin{eqnarray}\label{initialdata.h1.eqn20}
 \Delta\left(\eta_{N+1}u_{0}\ast\rho_{\varepsilon}\right)(x)&:=& \displaystyle\sum_{l=1}^{d}\left(\frac{\partial}{\partial x_{l}}\left(\eta_{N+1}u_{0}\right)\ast
 \frac{\partial\rho_{\varepsilon}}{\partial x_{l}}\right)(x),\nonumber\\
 &=& \displaystyle\sum_{l=1}^{d}\int_{\Omega}\frac{\partial}{\partial x_{l}}\rho_{\varepsilon}(x-z)\,\frac{\partial}{\partial x_{l}}\left(\eta_{N+1}u_{0}\right)(z)\,dz.
\end{eqnarray}
Using change of variable $\frac{x-y}{\varepsilon}=r$ in \eqref{initialdata.h1.eqn20}, we get
\begin{eqnarray}\label{initialdata.h1.eqn20A}
 \Delta\left(\eta_{N+1}u_{0}\ast\rho_{\varepsilon}\right)(x)=\frac{1}{\varepsilon}\displaystyle\sum_{l=1}^{d}\int_{\Omega}\rho^{\prime}(r)
 \frac{\partial}{\partial x_{l}}\left(\eta_{N+1}u_{0}\right)(x-r\varepsilon)\,(-1)^{d}\,dr.
\end{eqnarray}
For $i\in\left\{1,2,\cdots,d\right\}$, since $\mbox{supp}(\frac{\partial}{\partial x_{l}}\left(\eta_{N+1}u_{0}\right))
\subset\mbox{supp}(\eta_{N+1}u_{0})$ and $\mbox{supp}(\frac{\partial}{\partial x_{l}}\left(\rho_{\varepsilon}\right))
\subset\mbox{supp}(\rho_{\varepsilon})$, $\Delta\left(\eta_{N+1}u_{0}\ast\rho_{0}\right)=0$ on $\left(\Omega_{\varepsilon}\right)^{c}$.
From equation \eqref{initialdata.h1.eqn20A}, we get
\begin{eqnarray}\label{initialdata.h1.eqn21}
 \int_{\Omega}\left|\Delta\left(\eta_{N+1}u_{0}\ast\rho_{\varepsilon}\right)\right|\,dx &\leq& \frac{1}{\varepsilon}
 \|\rho^{\prime}\|_{L^{\infty}(\mathbb{R}^{d})}\displaystyle\sum_{l=1}^{d}\int_{\Omega}\left|\frac{\partial}{\partial x_{l}}
 \left(\eta_{N+1}u_{0}\right)(x-\varepsilon r)\right|\,dr\nonumber\\
 &\leq& \frac{1}{\varepsilon}\|\rho^{\prime}\|_{L^{\infty}(\mathbb{R}^{d})}\displaystyle\sum_{l=1}^{d}\int_{\Omega}\left|\frac{\partial}{\partial x_{l}}
 \left(\eta_{N+1}u_{0}\right)(r)\right|\,dr.
\end{eqnarray}
Denote 
$$C_{N+1}:=\|\rho^{\prime}\|_{L^{\infty}(\mathbb{R}^{d})}\displaystyle\sum_{l=1}^{d}\int_{\Omega}\left|\frac{\partial}{\partial x_{l}}
 \left(\eta_{N+1}u_{0}\right)(r)\right|\,dr.$$
 Taking $C=\displaystyle\sum_{p=1}^{N+1}C_{p}$, we get \eqref{initialdata.h1.eqn8C}. This completes the proof of 
 Lemma \ref{paper2.initialdata.lemma1}.\\
\vspace{0.2cm}\\
In the next result, we show the existence of a unique solution of \eqref{regularized.IBVP.Compact} in 
$C^{4+\beta,\frac{4+\beta}{2}}(\overline{\Omega_{T}})$ and an {\it a.e.} convergent subsequence of the quasilinear viscous approximations 
$\left(u^{\varepsilon}\right)$.   
\begin{lemma}\label{regularized.removecompact.thm2}
Let $f, B, u_{0}$ satisfy Hypothesis C. Then there exists a unique solution of \eqref{regularized.IBVP.Compact} in $C^{4+\beta,\frac{4+\beta}{2}}(\overline{\Omega_{T}})$, for every $0<\beta<1$ and the following estimates hold:
  \begin{eqnarray}
  \|u^{\varepsilon}\|_{L^{\infty}(\Omega)}&\leq& A,  \label{regularized.removecompact.eqn19}\\[2mm]
   TV_{\Omega_T}(u^{\varepsilon})&\leq& T\,\left(\left(\mbox{Vol}(\Omega)\right)^{\frac{1}{2}}\displaystyle\sum_{j=1}^{d}\sqrt{C_{i}}\right). \label{regularized.removecompact.eqn19aa}
  \end{eqnarray}
Also, there exists a constant $C> 0$ such that
\begin{eqnarray}\label{regularized.removecompact.eqn26}
 \left\|\frac{\partial u^{\varepsilon}}{\partial t}\right\|_{L^{1}(\Omega_{T})}\leq T\left(C\,\|B\|_{L^{\infty}(I)}  +\|B^{\prime}\|_{L^{\infty}(I)}\,\displaystyle\sum_{j=1}^{d} C_{i} + 
 \left(\mbox{Vol}(\Omega)\right)^{\frac{1}{2}}\,\|f^{\prime}\|_{L^{\infty}(I)}\displaystyle\sum_{j=1}^{d}\sqrt{C_{i}}\right)\nonumber\\
 {}
\end{eqnarray}
  
Furthermore, there exists a subsequence $\left\{u^{\varepsilon_{k}}\right\}_{k=1}^{\infty}$ of $u^{\varepsilon}$ and a function $u$ in $L^{1}(\Omega_T)$
such that $ u^{\varepsilon_{k}}\to u$ {\it a.e.} in $\Omega_T$, and also in $L^{1}(\Omega_T)$ as $k\to\infty$.
\end{lemma}
\begin{proof}
 Note that $f_{\varepsilon}\in\left(C^{\infty}(\mathbb{R})\right)^{d}$ and 
 $\|f^{\prime}_{\varepsilon}\|_{\left(L^{\infty}(\Omega)\right)^{d}}\leq\|f^{\prime}\|_{\left(L^{\infty}(\Omega)\right)^{d}}
 <\infty.$ Since $u_{0\varepsilon}\in \mathcal{D}(\Omega)$, the initial-boundary data of the viscosity problem 
\eqref{regularized.IBVP.Compact} satisfies compatibility conditions of orders $0,1,2$ which are required to apply 
Theorem  \ref{chapHR85thm5}.  Applying Theorem  \ref{chapHR85thm5}, we get the existence of a unique solution $u^{\varepsilon}$
in $C^{4+\beta,\frac{4+\beta}{2}}(\overline{\Omega_{T}})$ for the viscosity problem \eqref{regularized.IBVP.Compact},
and  $u^{\varepsilon}_{tt}\in C(\overline{\Omega_{T}})$. \\
\vspace{0.3cm}\\
By maximum principle (Theorem \ref{chap3thm1}), we conclude that $u^{\varepsilon}$ satisfies 
\begin{eqnarray}\label{regularized.removecompact.eqn20}
 \|u^{\varepsilon}\|_{L^{\infty}(\Omega)}\leq \|u_{0\varepsilon}\|_{L^{\infty}(\Omega)}\,\,{\it a.e.}\,\,t\in (0,T).
\end{eqnarray}
Combining \eqref{initialdata.h1.eqn2} with \eqref{regularized.removecompact.eqn20}, we get \eqref{regularized.removecompact.eqn19}.\\
Using equations \eqref{compactness12.eqn14} and \eqref{compactness12.eqn16} (from  {Step 1} in the proof of Theorem \ref{BVEstimate.thm1}) with
$f=f_{\varepsilon}$ and $u_{0}=u_{0\varepsilon}$, we arrive at
\begin{eqnarray}\label{regularized.eqn21aa}
 \int_{\Omega}\,\Big|\frac{\partial u^{\varepsilon}}{\partial t}(x,t)\Big|\,dx \leq \varepsilon
 \int_{\Omega}B(u_{0\varepsilon})\,\left|\Delta u_{0\varepsilon}\right|\,dx+ \int_{\Omega}\left|B^{'}(u_{0\varepsilon})\right|
 \left(\frac{\partial u_{0\varepsilon}}{\partial x_{j}}\right)^{2}\,dx +\displaystyle\sum_{j=1}^{d}\int_{\Omega} \left|f_{j\varepsilon}^{'}(u_{0\varepsilon})\right|\,\left|\frac{\partial u_{0\varepsilon}}{\partial x_{j}}\right|\,dx.\nonumber\\
 {}
\end{eqnarray}
Since $u_{0\varepsilon}\to u$ in $H^{1}_{0}(\Omega)$, for $i\in\left\{1,2,\cdots,d\right\}$ and for all $\varepsilon>0$, 
there exist $C_{i}>0$ such that
\begin{eqnarray}\label{regularized.removecompact.eqn22a}
\left\|\frac{\partial u_{0\varepsilon}}{\partial x_{i}}\right\|_{L^{2}(\Omega)}^{2}\leq C_{i}. 
\end{eqnarray}
Using H\"{o}lder inequality and \eqref{regularized.removecompact.eqn22a}, we have
\begin{eqnarray}\label{regularized.removecompact.eqn22abc}
 \left\|\frac{\partial u_{0\varepsilon}}{\partial x_{i}}\right\|_{L^{1}(\Omega)}&\leq&\left(\mbox{Vol}(\Omega)\right)^{\frac{1}{2}}\,
 \left\|\frac{\partial u_{0\varepsilon}}{\partial x_{i}}\right\|_{L^{2}(\Omega)},\nonumber\\
 &\leq&\left(\mbox{Vol}(\Omega)\right)^{\frac{1}{2}}\,\sqrt{C_{i}}.
\end{eqnarray}
In view of \eqref{regularized.removecompact.eqn22a}, \eqref{regularized.removecompact.eqn22abc} and \eqref{initialdata.h1.eqn8C},
we obtain 
\begin{eqnarray}\label{regularized.removecompact.eqn21}
 \int_{\Omega}\,\Big|\frac{\partial u^{\varepsilon}}{\partial t}(x,t)\Big|\,dx \leq 
 C\,\|B\|_{L^{\infty}(I)}  +\varepsilon\|B^{\prime}\|_{L^{\infty}(I)}\,\displaystyle\sum_{j=1}^{d} C_{i} + 
 \left(\mbox{Vol}(\Omega)\right)^{\frac{1}{2}}\,\|f^{\prime}_{\varepsilon}\|_{L^{\infty}(I)}\displaystyle\sum_{j=1}^{d}\sqrt{C_{i}}.\nonumber\\
 {}
\end{eqnarray}
We may assume that $\varepsilon <1$. Using $\|f^{\prime}_{\varepsilon}\|_{L^{\infty}(I)}\leq \|f^{\prime}\|_{L^{\infty}(I)}$ 
in \eqref{regularized.removecompact.eqn21}, we get
\begin{eqnarray}\label{regularized.removecompact.eqn100}
 \int_{\Omega}\,\Big|\frac{\partial u^{\varepsilon}}{\partial t}(x,t)\Big|\,dx \leq \left(C\,\|B\|_{L^{\infty}(I)}  +\|B^{\prime}\|_{L^{\infty}(I)}\,\displaystyle\sum_{j=1}^{d} C_{i} + 
 \left(\mbox{Vol}(\Omega)\right)^{\frac{1}{2}}\,\|f^{\prime}\|_{L^{\infty}(I)}\displaystyle\sum_{j=1}^{d}\sqrt{C_{i}}\right).\nonumber\\
 {}
\end{eqnarray}
Integrating on both sides of the last inequality w.r.t. $t$ on the interval $[0,T]$ yields 
\eqref{regularized.removecompact.eqn26}.\\
Using the conclusion of  {Step 2} in the proof of Theorem \ref{BVEstimate.thm1}, namely \eqref{BVestimateeqn24}, with 
$f=f_{\varepsilon}$ and $u_{0}=u_{0\varepsilon}$ yields
\begin{eqnarray}\label{regularized.removecompact.eqn23}
 TV_{\Omega}(u^{\varepsilon})\leq TV_{\Omega}(u_{0\varepsilon}).\nonumber\\
\end{eqnarray}
Applying \eqref{regularized.removecompact.eqn22abc} in \eqref{regularized.removecompact.eqn23}, we get
\begin{eqnarray}\label{regularized.removecompact.eqn24}
 TV_{\Omega}(u^{\varepsilon})\leq\left(\mbox{Vol}(\Omega)\right)^{\frac{1}{2}}\displaystyle\sum_{j=1}^{d}\sqrt{C_{i}}.
\end{eqnarray}
Once again integrating w.r.t. $t$ over the interval $[0,T]$ on both sides of the inequality 
\eqref{regularized.removecompact.eqn24} yields \eqref{regularized.removecompact.eqn19}.\\
We now use the compact embedding of $BV(\Omega_{T})\cap L^{1}(\Omega_{T})$ in $L^{1}(\Omega_{T})$ to conclude that 
there exists a subsequence $\left(u^{\varepsilon_{k}}\right)$ of $\left(u^{\varepsilon}\right)$ and a function $u$ in 
$L^{1}(\Omega_{T})$ such that $u^{\varepsilon_{k}}\to u$ as $k\to\infty$ in $L^{1}(\Omega_{T})$ as well as {\it a.e.} 
$\left(x,t\right)\in\Omega_{T}$. We still denote the subsequence $\left(u^{\varepsilon_{k}}\right)$ by 
$\left(u^{\varepsilon}\right)$.\\
\vspace{0.2cm}\\
\textbf{Proof of Theorem \ref{theorem3}:} We now show that the function $u$ as asserted in Lemma \ref{paper2.initialdata.lemma1}
is the unique entropy solution to \eqref{ivp.cl}. We prove entropy solution in two steps. In Step 1, we show that $u$ is  
a weak solution to \eqref{ivp.cl} and in Step 2, we show that weak solution $u$ is the unique entropy solution to \eqref{ivp.cl}.\\
\vspace{0.2cm}\\
\textbf{Step 1:} Let $\phi\in\mathcal{D}(\Omega\times [0,\infty)$. Multiplying \eqref{regularized.IBVP.Compact.a} by $\phi$ and using
integration by parts, we arrive at
\begin{equation}\label{regularized.removecompact.eqn27}
\begin{split}
 \int_{0}^{T}\hspace{0.1cm}\int_{\Omega}u^{\varepsilon}\hspace{0.1cm}\frac{\partial \phi}{\partial t}\hspace{0.1cm}dx\hspace{0.1cm}dt +
 \varepsilon\hspace{0.1cm}\displaystyle\sum_{j=1}^{d}\int_{0}^{T}\hspace{0.1cm}\int_{\Omega}\hspace{0.1cm}B(u^{\varepsilon})
 \hspace{0.1cm}\frac{\partial u^{\varepsilon}}{\partial x_{j}}\hspace{0.1cm}\frac{\partial \phi}{\partial x_{j}}\hspace{0.1cm}dx\hspace{0.1cm}dt 
 + \displaystyle\sum_{j=1}^{d}\int_{0}^{T}\hspace{0.1cm}\int_{\Omega}\hspace{0.1cm}f_{j\varepsilon}(u^{\varepsilon})
 \hspace{0.1cm}\frac{\partial \phi}{\partial x_{j}}\hspace{0.1cm}dx\hspace{0.1cm}dt\\ =\int_{\Omega}u_{0\varepsilon}
 \hspace{0.1cm}\phi(x,0)\hspace{0.1cm}dx\hspace{0.1cm}dt.
 \end{split}
\end{equation}
Observe that 
\begin{eqnarray}\label{regularized.removecompact.eqn30}
 \left|\int_{\Omega}\left(u_{0\varepsilon}(x)-u_{0}(x)\right)\,\phi(x,0)\,dx\right|\leq \|\phi\|_{L^{\infty}(\Omega_{T})}\|u_{0\varepsilon}-u_{0}\|_{L^{1}(\Omega)}.
\end{eqnarray}
Since $u_{0\varepsilon}\to u_{0}$ in $H^{1}_{0}(\Omega)$, therefore $u_{0\varepsilon}\to u_{0}$ in $L^{1}(\Omega)$ 
as $\varepsilon\to 0$. Using 
$$\int_{\Omega}u_{0\varepsilon}\hspace{0.1cm}\phi(x,0)\hspace{0.1cm}dx\to \int_{\Omega}u_{0}
 \hspace{0.1cm}\phi(x,0)\hspace{0.1cm}dx\,\,\mbox{as}\,\,\varepsilon\to 0$$
 and passing to the limit in \eqref{regularized.removecompact.eqn27} as $\varepsilon\to 0$ by following Step 1 in the proof of 
 Theorem \ref{theorem2}, we conclude that $u$ is a weak solution of \eqref{ivp.cl}, {\it i.e.,} for all 
 $\phi\in\mathcal{D}(\Omega\times[0,T))$, the limit $u$ satisfies
 \begin{equation}\label{regularized.removecompact.eqn34}
\begin{split}
 \int_{0}^{T}\hspace{0.1cm}\int_{\Omega}u\hspace{0.1cm}\frac{\partial \phi}{\partial t}\hspace{0.1cm}dx\hspace{0.1cm}dt +
 \displaystyle\sum_{j=1}^{d}\int_{0}^{T}\hspace{0.1cm}\int_{\Omega}\hspace{0.1cm}f_{j}(u)\hspace{0.1cm}\frac{\partial \phi}{\partial x_{j}}\hspace{0.1cm}dx\hspace{0.1cm}dt =\int_{\Omega}u_{0}(x)\hspace{0.1cm}\phi(x,0)\hspace{0.1cm}dx\hspace{0.1cm}dt,
 \end{split}
\end{equation}
\textbf{Step 2:} Following Step 2 in the proof of Theorem \ref{theorem2}, we conclude that the limit 
$u$ satisfies the entropy inequality \eqref{B.entropysolution.eqn1}. We do not prove the uniqueness of the entropy solution as it was already 
established in \cite{MR542510}. 
\end{proof}





\section{Acknowledgements}
Ramesh Mondal thanks Council for Scientific and Industrial Research (CSIR), New Delhi. A part of
this work was done when he was a CSIR-SRF at the Indian Institute of Technology Bombay.

\end{document}